\documentclass[twoside, 12pt]{article}
\usepackage[english]{babel}
\usepackage[latin1]{inputenc}
\usepackage{amsmath,amssymb}
\usepackage{amsfonts}
\usepackage{latexsym}
\usepackage{graphics,graphicx,psfrag}
\usepackage{upref}
\usepackage{comment}
\usepackage{hyperref}
\usepackage{caption,accents}

\usepackage{cite}
\usepackage[dvipsnames]{color, xcolor}
\usepackage{changes}

\allowdisplaybreaks[1]

\numberwithin{equation}{section}

\newcommand{\calk}{\mathcal{K}}
\newcommand{\partt}{\frac{\partial}{\partial t}}
\newcommand{\ca}{\mathcal{A}}
\newcommand{\supp}{\mathop{\mathrm{supp}\,}\nolimits}

\newcommand{\ds}{\displaystyle}
\newcommand{\real}{\mathbb{R}}
\newcommand{\rn}{\mathbb{R}^n}
\newcommand{\nat}{\mathbb{N}}

\newcommand{\rd}{\mathbb{R}^d}
\newcommand{\ud}{\mathrm{d}}
\newcommand{\uD}{\mathrm{D}}

\newtheorem{lemma}{Lemma}[section]
\newtheorem{proposition}[lemma]{Proposition}
\newtheorem{theorem}[lemma]{Theorem}
\newtheorem{corollary}[lemma]{Corollary}
\newtheorem{definition}[lemma]{Definition}
\newtheorem{assumption}[lemma]{Assumption}
\newtheorem{rem}[lemma]{Remark}
\newcommand{\remark}[1]{\begin{rem}{\upshape #1}\end{rem}}
\newtheorem{example}[lemma]{Example}
\newcommand{\expl}[1]{\begin{example}{\upshape #1}\end{example}}

\newcommand{\proofstart}{\mbox{P\,r\,o\,o\,f\, :\quad}}
\newcommand{\proofend}{\nopagebreak\hfill\raisebox{0.3em}{\fbox{}}\\}
\newenvironment{proof}{\proofstart}{\proofend}

\newcommand{\V}{\mathcal{K}}
\newcommand{\K}{{K}}

\newcommand{\beq}{\begin{equation}}
\newcommand{\eeq}{\end{equation}}

\begin{document}

\title{Besov regularity of parabolic and hyperbolic PDEs}
\author{S.~Dahlke\thanks{Some parts of this paper have been written during a research stay of this author at the Erwin Schr\"odinger International Institute for Mathematics and Physics (ESI). The support of ESI is greatly acknowledged.} \ and C.~Schneider\thanks{The work of this author has been supported by Deutsche Forschungsgemeinschaft (DFG), Grant No. SCHN 1509/1-1.
}}

%\date{Version: {\color{black}\today}}
\date{}
\maketitle

%\footnotetext{\textit{Math Subject Classifications.}
%                     }

%\footnotetext{\textit{Keywords and Phrases.}
 %                   }

\begin{abstract}
This paper is concerned with the regularity of  solutions {to} linear and nonlinear evolution equations on nonsmooth domains.   In particular, we study the smoothness in the specific scale 
$\ B^r_{\tau,\tau}, \ \frac{1}{\tau}=\frac{r}{d}+\frac{1}{p}\ $ of Besov spaces.   The regularity in these spaces determines the approximation order
that can be achieved by adaptive and other nonlinear approximation schemes. We show that for all cases under consideration the Besov regularity is high enough to 
justify the use of adaptive algorithms. 

{\em Math Subject Classifications. Primary:}   35B65,  35K55, 46E35.   {\em Secondary:} 35L15, 35A02,  35K05,  65M12.

\end{abstract}

\tableofcontents

\section{Introduction}\label{introduction}

This paper is concerned with regularity estimates of the solutions {to}
 evolution equations in nonsmooth  domains $\mathcal{O}$ contained in
${\mathbb R}^d $. In particular, we study parabolic equations of the form

\begin{equation} \label{parab-1a-i}
\left\{\begin{array}{rl}
\partt u+(-1)^m{L(x,t;D_x)}u\ =\ f \, &  \text{ in } K\times(0,T), \\
\frac{\partial^{k-1}}{\partial \nu^{k-1}}u\Big|_{\Gamma_{j,T}}\ =\ 0, & \   k=1,\ldots, m, \ j=1,\ldots, n,\\ 
u\big|_{t=0}\ =\ 0 \, & \text{ in } K
\end{array} \right\}
\end{equation}

in polyhedral cones $K\subset {\mathbb R}^3$ as well as associated  semilinear versions of 
(\ref{parab-1a-i}). 
 Here the partial differential operator is given by
\begin{equation}\label{L-op-i}
{L(x,t;D_x)}=\sum_{|\alpha|, |\beta|=0}^m D^{\alpha}_x({a_{\alpha \beta}(x,t)}D^{\beta}_x).
\end{equation}

We will also be concerned with  hyperbolic problems

\begin{equation} \label{hyp-1a-i}
\left\{\begin{array}{rl}
\frac{\partial^2}{\partial t^2}u+L(x,t,D_x) u \ = \ f  & \text{ in } \Omega\times (0,T), \\[0.1cm]
u(x,0) \ = \ \partt u(x,0)  \ = \ 0   & \text{ in } \Omega,\\[0.1cm]
u\big|_{\partial \Omega\times (0,T)} \ =\ 0 &    \\
\end{array}
\right\}
\end{equation}

in specific Lipschitz domains $\Omega\subset \real^d$, $d>2$,  where

 \[
L(x,t,D_x)u=-\sum_{i,j=1}^{{d}}\frac{\partial}{\partial{x_j}}\left(a_{ij}(x,t)\frac{\partial}{\partial {x_i}}u\right)
+\sum_{i=1}^{d}b_i(x,t)\frac{\partial}{\partial {x_i}}u+c(x,t)u,
\]
and   given function $f$. We study the spatial regularity of the
solutions to  (\ref{parab-1a-i}) and  (\ref{hyp-1a-i}) in specific
non-standard  smoothness spaces, i.e., in the so-called {\em adaptivity
scale {of Besov spaces}}

\begin{equation} \label{adaptivityscale}
B^r_{\tau,\tau}(\mathcal{O}), \quad \frac{1}{\tau}=\frac{r}{d}+\frac{1}{p}.
\end{equation}

The motivation for these kinds of studies can be explained as follows.  Evolution equations of the form (\ref{parab-1a-i}) and  (\ref{hyp-1a-i})
play important roles in the modelling of a lot of problems in science and engineering. As a classical example corresponding to (\ref{parab-1a-i}) let us
mention the heat equation that describes the variation in  temperature in a given region over time.
 In many cases,  analytic forms of the solutions
are not available, so that numerical schemes for their  constructive approximation are needed.  When it comes to practical applications, very often systems with
hundreds of thousands or even millions of degrees of freedom have to be handled.  In this case, {\em adaptive} strategies are often unavoidable to increase
efficiency. Essentially, an adaptive scheme is an updating strategy, where additional degrees of freedom are only spent in regions where
the numerical approximation is still `far away' from the exact solution.  Although the underlying idea is convincing, the following
general question arises:   what is the order of convergence that can be achieved by adaptive algorithms, and is it higher than the convergence order of
classical nonadaptive (uniform) schemes, which are much easier to design and to implement? As a rule of thumb, the following statement can be made:  the
convergence order that can be achieved by adaptive algorithms is determined by the regularity of the exact solution  in the adaptivity scale  (\ref{adaptivityscale})
of Besov spaces, whereas the convergence order that is available for  nonadaptive schemes
depends on  the classical Sobolev smoothness. In particular, for adaptive
wavelet schemes, these relations  can be {shown precisely}, we refer, e.g. to {\cite{DDV97, CDD1, DDD, Ste09}} for details.
Quite recently, it has turned out that the same interrelations also hold for the very important and  widespread adaptive finite
 element schemes, cf.  \cite{GM09}. Therefore, we can draw the following conclusion: adaptivity is justified,
 if the Besov regularity of the solution
in the Besov scale
(\ref{adaptivityscale}) is higher than its Sobolev smoothness!

For the case of {\em elliptic} partial differential equations, a lot of positive results in this direction are already established  \cite{Dah98, Dah99a, Dah99b, Dah02, DDHSW}. It is well--known
that  if the domain under consideration, the right--hand side and the coefficients are sufficiently smooth, then the problem is completely regular \cite{ADN59}, and there
is no reason why the  Besov smoothness should be higher than the Sobolev regularity. However, on general Lipschitz domains and in particular in polyhedral domains, the situation changes
dramatically.  On  these domains, singularities at the boundary may occur that diminish the Sobolev regularity of the solution significantly \cite{JK95, Gris92, Gris11}.
However, the analysis in the above mentioned papers shows that these boundary singularities do not influence the Besov regularity too much, so that the use of
adaptive algorithms for elliptic PDEs is completely justified!

 In this paper, we study similar questions for evolution equations of the form  
(\ref{parab-1a-i}),  
(\ref{hyp-1a-i}), and of associated 
semilinear versions of them. We show that in all these cases the Besov regularity is high enough 
to justify the use of adaptive algorithms. To the best of our knowledge, not so many results in this direction are available
so far.  For parabolic equations, first  results   for the special case of the heat equation have been reported 
in \cite{AGI08, AGI10, AG12}, but for a slightly different scale of Besov spaces.  For hyperbolic equations, let us mention
the seminal paper of DeVore and Lucier \cite{DVL90} which is concerned with concervation laws in 1D. \\ 
The main ingredients to prove our
results are regularity estimates in so--called {\em Kondratiev spaces}.  The study of solutions to PDEs in Kondratiev
spaces has already quite a long history. We refer, e.g. to \cite{KO83, JK95, Gris92, MR10}
(this list is clearly not complete). The basic idea is the following. As already outlined above, on  nonsmooth domains the solutions to PDEs 
as well as their  derivatives might  become highly singular as one approaches the boundary. Nevertheless, it has turned out
that their strong growth can to some extent be compensated by  means of specific weights that consist of the regularized distance
to the singular set of the domain to some power. We refer to Section \ref{Sect-2} for a detailed description of these spaces. 
 Recent studies have also shown that these Kondratiev spaces are very much related with Besov spaces in
 the adaptivity scale 
 (\ref{adaptivityscale}) in the sense that powerful embedding results exist, see, e.g. \cite{Han15}. So, Besov regularity results can be
established by first studying the equation under consideration in Kondratiev spaces and then using known (or deriving new)  embeddings
into Besov spaces.

We carry out this program in the following steps.  {In Section \ref{Sect-4}  we consider} parabolic evolution equations
in polyhedral cones contained in ${\mathbb R}^3$.  For these problems, regularity estimates in Kondratiev spaces have been
derived in  \cite{LL15}. However, for our purposes, it has been necessary to modify these results in order to treat semilinear problems later on as well,  see Section \ref{Sect-3}. Then, by a combination with embedding results from 
\cite{Han15}, we obtain our first main results.  They tell  us that if the right-hand side as well as its time derivatives are
contained in specific Kondratiev spaces, then, for every $t \in (0,T)$  the spatial Besov smoothness  
of the solution to (\ref{parab-1a-i}) is always larger than  $2m$, provided that some technical conditions on the operator pencil are satisfied, see Theorems \ref{thm-parab-Besov} and \ref{thm-parab-Besov-2}.
The reader should observe that the results are independent of the shape of the cone, and that the classical Sobolev smoothness
is usually limited by  $m$, see \cite{LL15}.  Therefore, for every $t$, the spatial Besov regularity is more than twice as
high as the Sobolev smoothness, which of course justifies the use of (spatial) adaptive algorithms. 
Moreover,  for smooth domains
and  right--hand sides in $L_2,$ the best one could  expect would be smoothness order $2m$ in 
the classical Sobolev scale. So, the Besov smoothness on polyhedral cones is at least as high as the Sobolev smoothness 
on smooth domains. 

Then, in Subsection \ref{Subsect-4.2} we generalize this result to semilinear equations of the form

\begin{equation} \label{parab-nonlin-1-i}
\left\{\begin{array}{rl}
\partt u+(-1)^m{L(x;D_x)}u +\varepsilon u^{M}\ =\ f \, &  \text{ in } K_T, \\
\frac{\partial^{k-1}}{\partial \nu^{k-1}}u\Big|_{\Gamma_{j,T}}\ =\ 0, & \   k=1,\ldots, m, \ j=1,\ldots, n,\\ 
u\big|_{t=0}\ =\ 0 \, & \text{ in } K,
\end{array} \right\}
\end{equation}

where $\varepsilon>0$ and $M\in \nat$. 
We show that in a sufficiently small ball containing the solution of the corresponding linear equation, there exists
a unique solution to  
(\ref{parab-nonlin-1-i}) possessing the same Besov smoothness in the scale 
(\ref{adaptivityscale}).  The proof is performed by a technically quite involved application of the Banach fixed point theorem. We show that 
(\ref{parab-nonlin-1-i}) has a unique solution in both, the classical scale of  Sobolev spaces as well as in the scale of Kondratiev spaces, and then again 
the result follows by an application of the embedding results from \cite{Han15}. The final result is stated in Theorem \ref{nonlin-B-reg3}. 

The next natural step is to also  study  the regularity in time direction.  We show that the mapping  $t\mapsto u(t, \cdot)$  is in fact a $C^l$-map into the adaptivity scale
 of Besov spaces,  
precisely, 
$$u \in \mathcal{C}^{l,\frac{1}{2}}((0,T), B^{\gamma}_{\tau,\infty}(K)), $$
see Theorem \ref{Hoelder-Besov-reg}. 

{In Section \ref{Sect-4a} we turn our attention to parabolic Besov regularity on general Lipschitz domains. Once again, our results rely on regularity estimates in Kondratiev spaces, which, for the case of general Lipschitz domains, have been derived in  \cite{Kim11}.  As in the previous section  we  use embedding results of Kondratiev spaces into the scale of Besov spaces for general Lipschitz domains as e.g. obtained in  \cite{Cio-Diss}. Comparing the regularity results  for general Lipschitz domains with the ones for  polyhedral cones from Section \ref{Sect-4}, it turns out that (as  expected) the results for the more specific  cones are much stronger. Furthermore, it is important to note that the analysis in Section \ref{Sect-4a} is restricted to second order operators, in contrast to the differential operators of general order  we considered before, cf. \eqref{L-op-i}. \\
}
Finally, in Section \ref{Sect-5} we study hyperbolic equations of the form 
(\ref{hyp-1a-i}). Also for these kinds of equations regularity estimates in Kondratiev spaces have 
been derived in \cite{LT15}, so that it is tempting to proceed in a similar way as in the parabolic
case. But then the following problem occurs:  the specific domains treated in \cite{LT15}
are not directly covered by the theory presented in \cite{Han15}. Therefore, we proceed in a slightly different way.
We use the  fact that Besov spaces can be characterized by wavelet expansions, and therefore, in order to establish
Besov smoothness of the solutions to (\ref{hyp-1a-i}), their wavelet coefficients have to be estimated.  This can be
done by once again exploiting the regularity estimates in Kondratiev spaces. \\

\section{Function spaces}\label{Sect-2}

\subsection{Preliminaries}

We start by collecting some general notation used throughout the paper. As usual,  we denote by $\nat$ the set of all natural numbers, $\nat_0=\mathbb N\cup\{0\}$, and 
$\real^d$, $d\in\nat$,  the $d$-dimensional real Euclidean space with $|x|$, for $x\in\real^d$, denoting the Euclidean norm of $x$. 
By $\mathbb{Z}^d$ we denote the lattice of all points in $\real^d$ with integer components. \\
%For $a\in\rr$, let   $[a]:=\max\{k\in\zz: k\leq a\}$ and $a_+:=\max (a,0)$.
%If $a,b\in\rr$,  then $a\vee b:=\max \{a,b\}$.
% $a_+:=\max(a,0)$ and let $[a]$ denote its integer part. 
%For $p\in (0,\infty]$, the number $p'$ is defined by
%$1/p':=(1-1/p)_+$ with the convention that $1/\infty=0$. 
We denote by  $c$ a generic positive constant which is independent of the main parameters, but its value may change from line to line. 
The expression $A\lesssim B$ means that $ A \leq c\,B$. If $A \lesssim
B$ and $B\lesssim A$, then we write $A \sim B$.  

Given two quasi-Banach spaces $X$ and $Y$, we write $X\hookrightarrow Y$ if $X\subset Y$ and the natural embedding is bounded. By $\supp f$ we denote the support of the function $f$. For a domain $\Omega\subset \real^d$ and $r\in \nat\cup \{\infty\}$ we write $C^r(\Omega)$ for the space of all {real}-valued $r$-times continuously differentiable functions, 
%where $C^m(\mathcal{O}):=\{f: \ D^{\alpha}f\in C(\Omega)\text{ for all }|\alpha|\leq m\}$  and $C(\mathcal{O})$ denotes the space of all %complex-valued bounded uniformly continuous functions, normed by 
%\[
%\| \psi| C^m(\mathcal{O})\|:=\max_{|\alpha|\leq m}\sup_{x\in \mathcal{O}}|D^{\alpha}\psi(x)|<\infty,
%\]
whereas $C(\Omega)$ is the space of bounded uniformly continuous functions, and  $\mathcal{D}(\Omega)$ for the set of test functions, i.e., the collection of all infinitely differentiable functions with  {support compactly contained in $\Omega$. Moreover,  $L^1_{\text{loc}}(\Omega)$ denotes the space of locally integrable functions on $\Omega$.} \\
For  a multi-index  $\alpha = (\alpha_1, \ldots,\alpha_d)\in \nat_0^d$ with with  $|\alpha| := \alpha_1+\ldots+ \alpha_d=r$, $r\in \nat_0$,  and an $r$-times differentiable function $u:\Omega\rightarrow \real$, we write 
\[
D^{(\alpha)}u=\frac{\partial^{|\alpha|}}{\partial x_1^{\alpha_1}\dots \partial x_d^{\alpha_d}} u
\]
for the corresponding classical partial derivative as well as $u^{(k)}:=D^{(k)}u$ in the one-dimensional case. Hence, the space $C^r(\Omega)$ is normed by 
\[
\| u| C^r(\Omega)\|:=\max_{|\alpha|\leq r}\sup_{x\in \Omega}|D^{(\alpha)}u(x)|<\infty. 
\]
Moreover, $\mathcal{S}(\real^d)$ denotes the Schwartz space of rapidly decreasing functions. The set of distributions on $\Omega$ will be denoted by $\mathcal{D}'(\Omega)$, whereas $\mathcal{S}'(\real^d)$ denotes the set of tempered distributions on $\real^d$. The terms {\em distribution} and {\em generalized function} will be used synonymously. For the application of a distribution $u\in \mathcal{D}'(\Omega)$ to a test function $\varphi\in \mathcal{D}(\Omega)$ we write $(u,\varphi)$. The same notation will be used if $u\in \mathcal{S}'(\real^d)$ and $\varphi\in \mathcal{S}(\real^d)$ (and also for the inner product in $L_2(\Omega)$).  For $u\in \mathcal{D}'(\Omega)$  and a multi-index $\alpha = (\alpha_1, \ldots,\alpha_d)\in \nat_0^d$, we write $D^{\alpha}u$ for the $\alpha$-th {\em generalized} or {\em distributional derivative} of $u$ with respect to $x=(x_1,\ldots, x_d)\in \Omega$, i.e., $D^{\alpha}u$ is a distribution on $\Omega$, uniquely determined by the formula   
\[
(D^{\alpha}u,\varphi):=(-1)^{|\alpha|}(u,D^{(\alpha)}\varphi), \qquad \varphi \in \mathcal{D}(\Omega). 
\]
{In particular, if  $u\in L^1_{\text{loc}}(\Omega)$ and  there exists a function $v\in L^1_{\text{loc}}(\Omega)$ such that 
\[
\int_\Omega v(x)\varphi(x)\ud x=(-1)^{|\alpha|}\int_{\Omega}u(x)D^{(\alpha)}\varphi(x)\ud x \qquad \text{for all} \qquad \varphi \in \mathcal{D}(\Omega), 
\]
we say that $v$ is the {\em $\alpha$-th weak derivative} of $u$ and  write $D^{\alpha}u=v$. 
}
We also use the notation $
\frac{\partial^k}{\partial x_j^k}u:=D^{\beta}u
$ as well as $\partial_{x_j^k}:=D^{\beta}u$,   for some 
multi-index  $\beta=(0,\ldots, k, \ldots,0)$ with $\beta_j=k$, $k\in \nat$. Furthermore, for $m\in \nat_0$, we write $D^mu$ for any (generalized) $m$-th order derivative of $u$, where $D^0u:=u$ and $Du:=D^1u$. Sometimes we shall use subscripts such as $D^m_x$ or  $D^m_t $ to emphasize that we only take derivatives with respect to $x=(x_1, \ldots, x_d)\in \Omega$ or $t\in \real$.

{For our analysis of parabolic and hyperbolic problems we shall deal with several different types of domains, which we introduce now.} \\
Let  $D$  denote some bounded polyhedral domain in $\real^d$. For $0<T<\infty$ put $D_T=D\times (0,T)$.  
As a special case of a polyhedral domain in $\real^3$  we will consider a cone (unbounded) with edges defined as follows.

\begin{definition}\label{def_polyhedral_cone}
Let \\ 
\begin{minipage}{0.5\textwidth}
\[
\K:=\{x\in \mathbb{R}^3: \ x/|x|=w\in \Omega\}
\]
be a polyhedral cone in $\mathbb{R}^3$ with vertex at the origin. Suppose that the boundary $\partial \K$ consists of the vertex $x=0$, the edges (half-lines) $M_1,\ldots, M_n$ and smooth faces $\Gamma_1,\ldots, \Gamma_n$. Moreover,  $\Omega=\K\cap S^2$ is a curvilinear polygon on the unit sphere bounded by the arcs $\gamma_1,\ldots, \gamma_n$. The angle at the edge $M_k$ will be denoted by $\theta_k$, $k=1,\ldots,n$. \\
Furthermore,  put 
$ 
\Gamma_{j,T}=\Gamma_j\times(0,T)$ and  $K_T=K\times (0,T). $ \\
\end{minipage}\hfill \begin{minipage}{0.4\textwidth}
\includegraphics[width=5cm]{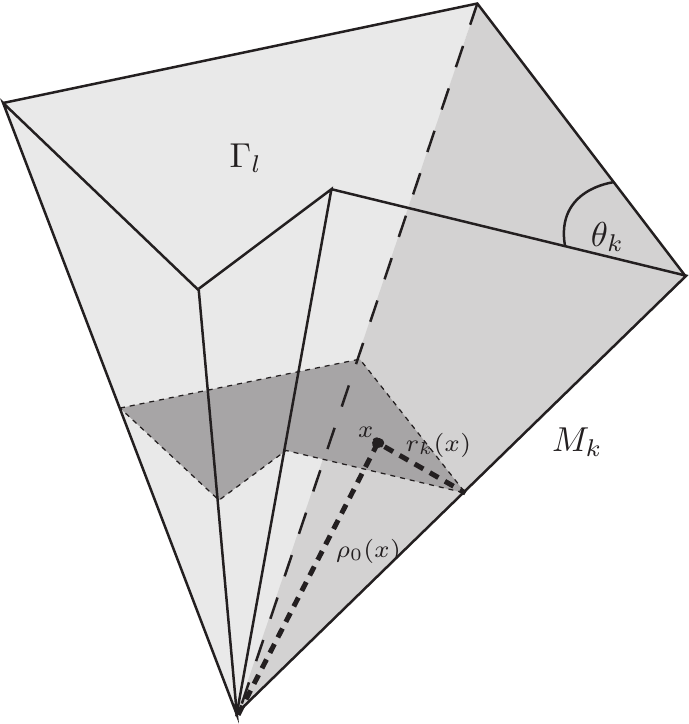}
\captionof{figure}{Polyhedral cone $K$}
\end{minipage}\\
\end{definition}

We shall also deal with the truncated cone 
\begin{equation}\label{trunc-cone}
K_0:=\{x\in K: \ |x|<r_0\} 
\end{equation}
for some real number $r_0>0$ and put $K_{0,T}:=K_0\times (0,T)$.

%%%%%%
% Maybe possible generalization to polyhedral domains in higher dimensions
%%%%%%
%Let $D$ be some bounded polyhedral domain  in $\mathbb{R}^d$ with (open) faces $\Gamma_j$, $j=1,\ldots, N$ and edges %%$M_{k}$, $k=1,\ldots, n$, which are the intersections of $\overline{\Gamma}_i$ and $\overline{\Gamma}_j$. The angle at the %edge $M_{k}$ will be denoted by $\omega_{k}$. By $V_m$, $m=1,\ldots, M$, we denote the set of vertices of $D$ and %assume that $V_1$ is located at the origin.   \\

For hyperbolic problems we will consider the following domains.  

\begin{definition}\label{def_special_Lip}
Let  $\Omega\subset \real^d$, $d>2$,  be a Lipschitz domain, whose boundary $\partial \Omega$ consists of two (smooth) surfaces $\Gamma_1$ and $\Gamma_2$ intersecting along a manifold $l_0$. We assume that in a neighbourhood of each point of $l_0$ the set $\overline{\Omega}$ is diffeomorphic to a dihedral angle. {For any $P\in l_0$ {define $T_1(P)$ and $T_2(P)$ as (part of) the tangent spaces in $P$ w.r.t. $\Gamma_1$ and $\Gamma_2$.}} 
Furthermore, $\Omega_T:=\Omega\times (0,T)$. \\

\begin{minipage}{\textwidth}
\begin{center} \includegraphics[width=11cm]{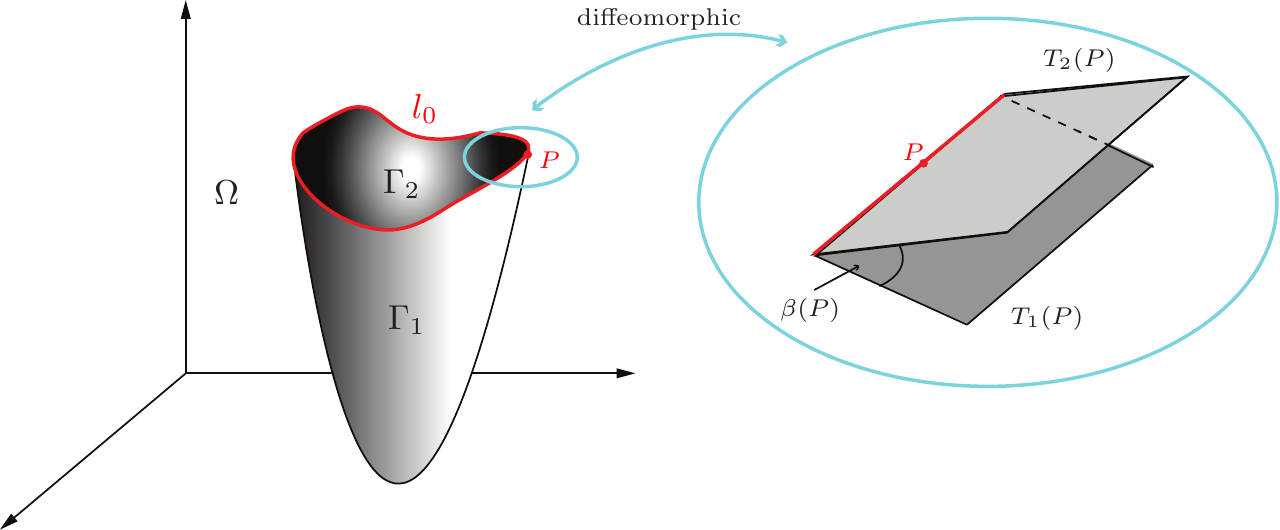}\end{center}
\captionof{figure}{Special Lipschitz domain $\Omega$}
\end{minipage}\\
\end{definition}

\subsection{Sobolev and H\"older spaces}

Unless stated otherwise, let $\mathcal{O}$ stand for either one of the domains $D$, $K$ or $\Omega$ defined above and put  $\mathcal{O}_T:=\mathcal{O}\times (0,T)$. 
We introduce some Sobolev spaces. Let $W_p^m(\mathcal{O})$, $1<p<\infty$, $m\in \nat_0$, denote the Sobolev space containing all complex-valued functions $u(x)$ defined on $\mathcal{O}$ such that 
\[
\|u|W_p^m(\mathcal{O})\|=\left(\sum_{|\alpha|\leq m}\int_{\mathcal{O}}|D^{\alpha}_xu(x)|^p\ud x\right)^{1/p}<\infty.
\]
By $\mathring{W}_p^m(\mathcal{O})$ we denote the closure of $\mathcal{D}(\mathcal{O})$ in $W_p^m(\mathcal{O})$. Moreover,  $W^{-m}_p(\mathcal{O})$ stands for the dual space $\left(\mathring{W}^m_p(\mathcal{O})\right)'$ of $\mathring{W}^m_p(\mathcal{O})$. The duality pairing is denoted by $(u, v)$ for $u\in W^{-m}_p(\mathcal{O})$ and $v\in\mathring{W}^m_p(\mathcal{O})$.  

\medskip

{For $s\geq 0$, fractional Sobolev spaces $W^s_2(\real^d)$ are defined as the spaces which contain  all complex-valued  functions $u(x)$ defined on $\real^d$} such that for $s=k+\lambda$ with $k\in \nat_0$ and $\lambda \in (0,1)$ it holds 
\begin{equation}\label{slobodeckij}
\|u|W^s_2(\real^d)\|=\|u|W^k_2(\real^d)\|+\sum_{|\alpha|=k}\left(\int_{\real^d\times \real^d}\frac{|D^{\alpha}u(x)-D^{\alpha}u(y)|^2}{|x-y|^{d+2\lambda}}\ud x\ud y\right)^{1/2}<\infty.
\end{equation}
It is well known that an equivalent norm is given by 
\[
\|u|W^s_2(\real^d)\|\sim \|\left((1+|\xi|^2)^{s/2}\hat{u}\right)^{\vee}|L_2(\real^d)\|, 
\]
(where $\hat{u}$ and $\check{u}$ denote the Fourier transform and its inverse, respectively), cf. \cite{Hack92}. 
Corresponding spaces on domains can be defined via restriction, i.e., we put 
\begin{eqnarray*}
W^s_2(\mathcal{O})&=& \left\{f\in L_2(\mathcal{O}): \ \exists g\in W^s_{2}(\real^d), \ g\big|_{\mathcal{O}}=f \right\},\\
\|f|W^s_{2}(\mathcal{O})\|&=& \inf_{g|_{\mathcal{O}}=f}\|f|W^s_{2}(\real^d)\|. 
\end{eqnarray*}

\paragraph{The spaces $W^m_p(I,X)$ and $\mathcal{C}^{k,\alpha}(I,X)$}

Consider a Banach space $X$ and an open interval $I=(0,T)\subset \real$ with $T<\infty$.  We write $C(I,X)$ for the space consisting of all bounded and (uniformly) continuous functions $u: I\rightarrow X$ normed by 
\[
\|u|C(I,X)\|:=\max_{t\in I}\|u(t)|X\|.
\] 
Moreover, we say that $u\in C^k(I,X)$, $k\in \nat_0$, if $u$ has a Taylor expansion 
\[
u(t+h)=u(t)+u'(t)h+\frac{1}{2}u''(t)h^2+\ldots + \frac{1}{k!}u^{(k)}(t)h^k+r_k(t,h)
\] 
for all $t+h,t\in I$ such that 
\begin{itemize}
\item $u^{(j)}(t)$ depends continuously on $t$ for all $j=0,\ldots, k$,
\item $\ds \lim_{|h|\rightarrow 0}\frac{\|r_k(t,h)|X\|}{|h|^k}= 0$. 
\end{itemize}
The space $C^k(I,X)$ is then equipped with the following norm 
\[
\|u|C^k(I,X)\|:=\sum_{j=0}^k \|u^{(j)}|C(I,X)\|.
\]
Given $\alpha\in (0,1)$, we denote by $\mathcal{C}^{\alpha}(I,X)$ the H\"older space containing all $u\in C(I,X)$ such that 
\begin{eqnarray*}
\|u|{\mathcal{C}^{\alpha}}(I,X)\|&:=&\|u|C(I,X)\|+|u|_{C^{\alpha}({I,X})}\\
&=&\|u|C(I,X)\|+\sup_{t,s\in I\atop t\neq s}\frac{\|u(t)-u(s)|X\|}{|t-s|^{\alpha}}<\infty.
\end{eqnarray*}
Consequently, $\mathcal{C}^{k,\alpha}(I,X)$ contains all functions $u\in C(I,X)$ such that 
\[
\|u|{\mathcal{C}^{k,\alpha}}(I,X)\|:=\|u|C^k(I,X)\|+|u^{(k)}|_{C^{\alpha}({I,X})}<\infty.
\]

The above concepts extend to spaces $C^k(I,Y)$ {and $\mathcal{C}^{k,\alpha}(I,Y)$, respectively}, where $Y$ is a quasi-Banach space. For some further comments we refer to  Remark \ref{quasi-B-deriv} in  Appendix \ref{Appendix-A}. 

Let us briefly recall the definition of Lebesgue and Sobolev spaces for functions with values in a Banach space $X$. We denote by $L_p(I,X)$, $1\leq p\leq \infty$, the space of (equivalence classes of) measurable functions $u:I\rightarrow X$ such that the mapping $t\mapsto \|u(t)\|_X$ belongs to $L_p(I)$, which is endowed with the norm 
\[
\|u|{L_p(I,X)}\|=
\begin{cases}
\ds \left(\int_I \|u(t)|{X}\|^p\ud t\right)^{1/p} & \text{if } p<\infty, \\
\ds \mathrm{ess}\ \sup_{t\in I}\|u(t)|X\| & \text{if } p=\infty. 
\end{cases}
\]
{The definition of weak  derivatives of  Banach-space valued distributions is a natural generalization of the one for real-valued distributions. We refer to \cite[Part I, Sect. 3]{Cap14} in this context. Let  $\mathcal{D}'(I,X):=\mathcal{L}(\mathcal{D}(I),X)$ be the space of $X$-valued distributions,  where $\mathcal{L}(U,V)$ denotes the space of all linear continuous functions from $U$ to $V$. For the application of a distribution $u\in \mathcal{D'}(I,X)$ to a test function $\varphi\in \mathcal{D}(I)$, we use the notation $(u,\varphi)$. 
For  $u\in \mathcal{D}'(I,X)$ and $k\in \nat$,  the  {$k$-th \em generalized} or {\em distributional derivative} ${\partial_{t^k}}u$ is defined as an $X$-valued  distribution satisfying 
\[
({\partial_{t^k}}u,\varphi):=(-1)^k (u, {\partial_{t^k}}\varphi), \qquad \varphi\in \mathcal{D}(I).
\]
In particular, if $u:I\rightarrow X$ is an integrable function  and there exists an integrable function $v:I\rightarrow X$ satisfying  
\[
\int_I v(t)\varphi(t)\ud t=(-1)^k\int_I u(t)\partial_{t^k}\varphi(t)\ud t\qquad \text{for all} \quad \varphi\in \mathcal{D}(I), 
\]
where the integrals above are  Bochner integrals, cf. \cite{Cohn13}, we say that $v$ is the {\em $k$-th weak derivative} of $u$ and write  ${\partial_{t^k}}u=v$.  
%\begin{equation}\label{weak-deriv-order}
%kjkj
%\end{equation}
}
For $m\in \nat_0$ we denote by   $W^m_p(I,X)$  the space of all  {functions $u\in L_p(I,X)$}, whose weak derivatives of order $0\leq k\leq m$ belong to $L_p(I,X)$,  normed by  
\[
\|u\|_{W^m_p(I,X)}=\begin{cases}
\ds \left(\sum_{k=0}^{m}\left\|\partial_{t^k} u|{L_p(I,X)}\right\|^p\right)^{1/p}& \text{if } p<\infty, \\
\ds \max_{0\leq k\leq m}\left\|\partial_{t^k}u|{L_{\infty}(I,X)}\right\| & \text{if }p=\infty.
\end{cases}
\]
$L_p(I,X)$ and $W^m_p(I,X)$ are Banach spaces. \\

In the course of this paper, we will also need a version of Sobolev's embedding theorem for Banach-space valued functions. For the reader's convenience we give a short proof in Appendix \ref{Appendix-A}, since a suitable reference was not found.

\begin{theorem}[Sobolev embedding]\label{thm-sob-emb}
Let $1<p<\infty$ and $m\in \nat$. Then 
\begin{equation}
W^{m}_p(I,X)\hookrightarrow \mathcal{C}^{m-1,1-\frac 1p}(I,X). 
\end{equation}
%for parameters satisfying 
%$k=\left\lfloor{m-\frac 1p}\right\rfloor=m-1$,  $\alpha=m-\frac 1p-k=1-\frac 1p$, $0<\alpha<1$. 
%$k=m-1$, $\alpha=1-\frac 1p$, and $0<\alpha<1$. 
\end{theorem}

\remark{In particular, we have 
$\ 
W^1_p(I,X)\hookrightarrow \mathcal{C}^{0,1-\frac 1p}(I,X)
\ $
for $m=1$. This was proven in \cite[Thm. 1.4.38]{CH98}. 
}

\medskip

With the notation introduced above we  further put  for brevity 
\[
L_p(\mathcal{O}_T):=L_p((0,T), L_p(\mathcal{O})) 
\]
and 
$$
W^{m,l}_p(\mathcal{O}_T):=W^{l-1}_p((0,T), W^m_p(\mathcal{O}))\cap W^{l}_p((0,T),L_p(\mathcal{O})),\qquad l\in \nat, 
$$ 
normed by 
\[\|u|W^{m,l}_p(\mathcal{O}_T)\|=\|u|W^{l-1}_p((0,T), W^m_p(\mathcal{O}))\|+\|u|W^{l}((0,T),L_p(\mathcal{O}))\|.\]
%i.e., the Sobolev space containing  all complex functions $u(x,t)$ defined on $\mathcal{O}_T$ such that 
%\[
%\|u\|_{W_p^{m,l}(\mathcal{O}_T)}=
%\left(\int_{\mathcal{O}_T}\left(\sum_{|\alpha|\leq m}|D^{\alpha}_xu(x,t)|^p+\sum_{j=1}^l|\partial_{t^j}u(x,t)|^p\right)\ud x\ud t
%\right)^{1/p}<\infty.
%\]
The space ${\mathring{W}}_p^{m,l}(\mathcal{O}_T)$ is the closure in $W_p^{m,l}(\mathcal{O}_T)$ of the set consisting of all functions $u\in C^{\infty}(\mathcal{O}_T)$, which vanish near $\partial \mathcal{O}\times [0,T]$.\\

\subsection{Weighted Sobolev spaces}

In the sequel we shall further  consider several types of  weighted Sobolev spaces. The first {ones} are the so-called {\em Kondratiev spaces} $\V^m_{p,a}(\mathcal{O})$, defined as the collection of all functions $u(x)$ such that  

\begin{equation}\label{Kondratiev-1}
\|u|\V^m_{p,a}(\mathcal{O})\|:=\left(\sum_{|\alpha|\leq m}\int_{\mathcal{O}} |\varrho(x)|^{p(|\alpha |-a)}|D^{\alpha}_x u(x)|^p\ud x\right)^{1/p}<\infty,
\end{equation}
where $a\in \real$, $1<p<\infty$, $m\in \nat_0$, $\alpha\in \nat^n_0$, and the weight function $\varrho: D\rightarrow [0,1]$ is the smooth distance to the singular set of $\mathcal{O}$, i.e., $\varrho$ is a smooth function and in the vicinity of the singular set $S$  it is {equivalent} to the distance to that set.  In particular,  if $\mathcal{O}=D$ or $\mathcal{O}=K$ then in  2D the  singular set  $S$ consists of the vertices of the polygon whereas in  3D the set $S$ consists of vertices and edges of the polyhedra/polyhedral cone. \\ %On the other hand if $\mathcal{O}=\Omega$, then $S=l_0$ and in the vicinity of $l_0$ the function $\varrho $ is equal to $\mathrm{dist}(x,l_0)$.  \\

Generalizing the above concept to functions depending on the time $t\in (0,T)$, we define Kondratiev type spaces  on $\mathcal{O}_T$, denoted by $L_q((0,T),\V^m_{p,a}(\mathcal{O}))$, which   contain all functions $u(x,t)$ such that 
\begin{align}
\|u|&L_q((0,T), \V^m_{p,a}(\mathcal{O}))\|\notag\\
&:=\left(\int_{(0,T)}\left(\sum_{|\alpha|\leq m}\int_{D} |\varrho(x)|^{p(|\alpha |-a)}|D^{\alpha}_x u(x,t)|^p\ud x\right)^{q/p}\ud t\right)^{1/q}<\infty, \label{Kondratiev-3}
\end{align}
with $0<q\leq \infty$ and  parameters $a,p,m$  as above.

Finally, concerning the special Lipschitz domains $\Omega\subset \real^d$, $d>2$,  from Definition \ref{def_special_Lip}
the corresponding Sobolev spaces $\mathcal{K}^m_{p,a}(\Omega)$ we consider are defined as in \eqref{Kondratiev-1} and contain all functions $u(x)$ such that 
\begin{equation}\label{Kondratiev-5}
\|u|\mathcal{K}^m_{p,a}(\Omega)\|:=\left(\sum_{|\alpha|\leq m}\int_{\Omega}|\rho(x)|^{p(|\alpha|-a)}|D^{\alpha}_x u(x)|^p\ud x\right)^{1/p}<\infty,
\end{equation}
where the weight function $\rho:\Omega\rightarrow [0,1]$ now is the smooth distance to $l_0$, i.e., $\rho$ is a smooth function and in the vicinity of $l_0$ it is equivalent to $\mathrm{dist}(x,l_0)$. The spaces $L_q((0,T),\mathcal{K}^{m}_{p,a}(\Omega))$ are defined in an obvious way {analogously} to \eqref{Kondratiev-3}.  \\

\paragraph{Some properties of weighted Sobolev spaces}

\begin{itemize}
\item Clearly, by definition we have the following type of embeddings 
\begin{equation}\label{kondratiev-emb}
K^{m}_{p,a}(\mathcal{O})\hookrightarrow K^{m'}_{p,a}(\mathcal{O}), \qquad K^{m}_{p,a}(\mathcal{O})\hookrightarrow K^{m}_{p,a'}(\mathcal{O}),
\end{equation}
if  $m'<m$ and $a'<a$. 
\item A function ${\varphi} \in C^m(\mathcal{O})$ is a pointwise multiplier for $\mathcal{K}^m_{p,a}(\mathcal{O})$, i.e., ${\varphi} u\in \mathcal{K}^m_{p,a}(\mathcal{O})$ for all $u\in \mathcal{K}^m_{p,a}(\mathcal{O})$ and 
\begin{equation}\label{multiplier}
\|{\varphi} u|\mathcal{K}^m_{p,a}(\mathcal{O})\|\leq c_{{\varphi}}\|u|\mathcal{K}^m_{p,a}(\mathcal{O})\|. 
\end{equation}
\end{itemize}

Concerning pointwise multiplication the following results are proven in \cite{DHS17a}. %\cite[Cor.~31, Thm. 43]{DHS17a}. 
Note that the domains of polyhedral type considered there include our polyhedral cones $K\subset \real^3$ as well as bounded polyhedral domains $D\subset \real^d$.

%\begin{theorem}\label{thm-pointwise-mult}
%Let $\frac{d}{2}<p<\infty$, $m\in \nat$, and $a\geq \frac{d}{p}-1$. Then there exists a constant $c$ such that 
%\[
%\|uv| \mathcal{K}^{m-1}_{a-1,p}(K)\|\leq c\|u|\mathcal{K}^{m+1}_{a+1,p}(K)\|\cdot \|v|\mathcal{K}^{m-1}_{a-1,p}(K)\|
%\]
%holds for all $u\in \mathcal{K}^{m+1}_{a+1,p}(K)$ and $v\in \mathcal{K}^{m-1}_{a-1,p}(K)$.
%\end{theorem}

{
\begin{corollary}\label{thm-pointwise-mult-2}
\begin{itemize}
\item[(i)] Let $m\in \nat$, $a\geq \frac dp$, and either $1<p<\infty$ and $m>\frac dp$ or $p=1$ and $m\geq d$. 
Then the Kondratiev space $\calk^m_{a,p}(K)$ is an algebra with respect to pointwise multiplication, i.e.,  there exists a constant $c$ such that 
\[
\|uv| \mathcal{K}^{m}_{a,p}(K)\|\leq c\|u|\mathcal{K}^{m}_{a,p}(K)\|\cdot \|v|\mathcal{K}^{m}_{a,p}(K)\|
\]
holds for all $u,v\in \mathcal{K}^{m}_{a,p}(K)$.
\item[(ii)] Let $\frac{d}{2}<p<\infty$, $m\in \nat$, and $a\geq \frac{d}{p}-1$. Then there exists a constant $c$ such that 
\[
\|uv| \mathcal{K}^{m-1}_{a-1,p}(K)\|\leq c\|u|\mathcal{K}^{m+1}_{a+1,p}(K)\|\cdot \|v|\mathcal{K}^{m-1}_{a-1,p}(K)\|
\]
holds for all $u\in \mathcal{K}^{m+1}_{a+1,p}(K)$ and $v\in \mathcal{K}^{m-1}_{a-1,p}(K)$.
\end{itemize}
\end{corollary}
}

Furthermore, we shall  need  a lifting property for Kondratiev spaces. For classical Sobolev spaces by definition it is clear that 
%for  $u\in W^m_p(\mathcal{O})$ and  $\alpha\in \nat_0^n$ with $|\alpha|\leq m$,  we have
\begin{equation}\label{lift}
u\in W^m_p\quad \Longrightarrow \quad D^{\alpha}u\in W^{m-|\alpha|}_p, 
\end{equation}
for $\alpha\in \nat_0^n$ with $|\alpha|\leq m$. For a generalization to Besov and Triebel-Lizorkin spaces we  also refer to \cite[p. 22, Prop. 2]{RS96} in this context. 
In the following theorem we will study the behaviour of 
$ \  u\rightarrow D^{\alpha} u \ $  
in Kondratiev spaces, which turns out to be very similar as for Sobolev spaces. 

\medskip 

\begin{theorem}\label{thm-lift-kondratiev}
Let  $a\in \real$, $1<p<\infty$, and $m\in \nat_0$. Then for  $u\in \mathcal{K}^m_{p,a}(K)$  and $\alpha\in \nat_0^n$ with $|\alpha|\leq m$, we have 
\[
D^{\alpha}u\in \mathcal{K}^{m-|\alpha|}_{p,a-|\alpha|}(K). 
\] 
\end{theorem}

\begin{proof}
The result  follows immediately from the following observation 
\begin{align*}
\left\|D^{\alpha}u|\mathcal{K}^{m-|\alpha|}_{p,a-|\alpha|}(K)\right\|^p
& = \sum_{|\beta|\leq m-|\alpha|}\int_K \rho(x)^{p(|\beta|-(a-|\alpha|))}|D^{\beta}(D^{\alpha}u(x))|^p \ud x\\
& = \sum_{|\beta|+|\alpha|\leq m}\int_K \rho(x)^{p(|\beta|+|\alpha|-a)}|D^{\alpha+\beta}u(x)|^p\ud x\\
& \lesssim \sum_{|\gamma|\leq m}\int_K \rho(x)^{p(|\gamma|-a)}|D^{\gamma}u(x)|^p\ud x = \|u|\mathcal{K}^m_{p,a}(K)\|^p. 
\end{align*}
\end{proof}

\subsection{Besov spaces}
\label{sect-Besov}

\subsubsection{Definition of Besov spaces}

Besov spaces  can be defined in {various different} ways. {In this subsection, we start with a  definition via higher order differences as can be found in \cite{Tri83}. Moreover, we recall the fact that  under certain restrictions on the parameters the Besov spaces allow a characterization in terms of wavelet decompositions.} 
In this context we refer, e.g.  to \cite{Coh03, Mey92}.  {In particular, this wavelet characterization} will turn out to be extremely useful when proving embeddings of weighted Sobolev spaces into Besov spaces from the non-linear approximation scale \eqref{adaptivityscale}.  For further information on Besov spaces and related function spaces as well as equivalent definitions, we refer to \cite{Tri06} and the references given there.  

If $f$
is an arbitrary function on $\real^d$, $h\in\real^d$ and $r\in\nat$, then
\[
(\Delta_h^1 f)(x)=f(x+h)-f(x) \quad \text{and} \quad
(\Delta_h^{r+1} f)(x)=\Delta_h^1(\Delta_h^r f)(x)
\]
are the usual iterated differences. 
Given a function $f\in L_p(\real^d)$ the \textit{$r$-th modulus of smoothness} is defined
by
\[
\omega_r(f,t)_p=\sup_{|h|\leq t} \|\Delta_h^r f\mid L_p(\real^d)\|, \quad
t>0, \quad 0< p\leq \infty.
\]

\begin{definition}\label{defi-Besov}
Let $0<p,q\leq \infty$, $s>0$, and $r\in\nat$ such that $r>s$. Then the Besov
space $B^s_{p,q}(\real^d)$ contains all $f\in L_p(\real^d)$ such that
\[
\|f|B^s_{p,q}(\real^d)\| = \|f|L_p(\real^d)\| + \left(\int_0^1 t^{-sq} \omega_r(f,t)_p^q
\ \frac{\ud t}{t}\right)^{1/q}<\infty
\]
$($with the usual modification if $q=\infty)$.
\end{definition}

\begin{remark}{
Definition~\ref{defi-Besov} is independent of $r$, meaning that different values
of $r>s$ result in (quasi-)norms which are equivalent.  
Furthermore, the spaces are quasi-Banach spaces (Banach spaces if $p,q\geq 1$).
Note that we deal with subspaces of $L_p(\rn)$, in particular, for $s>0$ and $0<q\leq\infty$, we have the embedding
\[
 B^s_{p,q}(\real^d)\hookrightarrow L_p(\real^d),\qquad  0<p\leq \infty. 
\]}
\end{remark}

\subsubsection{Wavelet characterization of Besov spaces}

Wavelets are specific orthonormal bases for $L_2(\mathbb{R})$ that are obtained by dilating, translating and scaling one fixed function, the so--called  {\em mother wavelet} $\psi$. The mother wavelet is usually constructed by means of a so-called {\em multiresolution analysis,} that is, a sequence  $\{V_j\}_{j \in \mathbb{Z}}$  of shift-invariant, closed subspaces of $L_2(\mathbb{R})$ whose union is dense in $L_2$ while their intersection is zero. Moreover, all the spaces are related via dyadic dilation, and the space  $V_0$  is spanned  by the translates of  one fixed function $\phi$, called the {\em generator}.  In her fundamental work \cite{Dau1, Dau2}  I. Daubechies has shown that there exist families of compactly supported wavelets. By taking tensor products, a compactly supported orthonormal basis for $L_2(\mathbb{R}^d)$ can be constructed which will also be used in this paper.\\
Let ${\phi}$ be a father wavelet of tensor product type on $\real^d$ and let $\Psi'=\{\psi_i: \ i=1,\ldots, 2^d-1\}$ be the set containing the corresponding multivariate mother wavelets such that, for a given $r\in \nat$ and some $N>0$ the following localization, smoothness and vanishing moment conditions hold. For all $\psi\in \Psi'$, 
\begin{align}
\supp{\phi}, \ \supp \psi   & \ \subset \ [-N,N]^d, \label{wavelet-1}\\
{\phi}, \ \psi  & \ \in \ C^r(\real^d), \label{wavelet-2}\\
\int_{\real^d} x^{\alpha}\psi(x)\ud x & \ =\ 0 \quad \text{ for all } \alpha \in \nat_0^d \ \text{ with } \  \ |\alpha|\leq r. \label{wavelet-3}
\end{align}  
We refer again to  \cite{Dau1, Dau2} for a detailed discussion. 
The set of all dyadic cubes in $\real^d$ with measure at most $1$ is denoted by 
\[
\mathcal{D}^{+}:=\left\{I\subset \real^d: \ I=2^{-j}([0,1]^d+k), \ j\in \nat_0, \ k\in \mathbb{Z}^d\right\}
\]
and we set $\mathcal{D}_j:=\{I\in \mathcal{D}^+: \ |I|=2^{-jd}\}.$ 
For the dyadic shifts and dilations of the father wavelet and the corresponding wavelets we use the abbreviations 
\begin{equation}\label{wavelet-4}
{\phi}_k(x):={\phi}(x-k), \quad \psi_{I}(x):=2^{jd/2}\psi(2^jx-k) \qquad \text{for}\quad  j\in \nat_0, \ k\in \mathbb{Z}^d, \ \psi\in \Psi'. 
\end{equation}
It follows that 
\[
\left\{{\phi}_k, \ \psi_{I}:  \ k\in \mathbb{Z}^d, \  I\in \mathcal{D}^+, \  \psi\in \Psi'\right\}
\]
is an orthonormal basis in $L_2(\real^d)$. %, hence, we can write every function $f\in L_2(\real^d)$ as 
%\[
%f=\sum_k a_k {\phi}_k+\sum_{i,j,k}a_{i,j,k}\psi_{i,j,k}. 
%\] 
Denote by $Q(I)$ some dyadic cube (of minimal size) such that $\supp \psi_I \subset Q(I)$ for every $\psi\in \Psi'$. Then, we clearly have $Q(I)=2^{-j}k+2^{-j}Q$ for some dyadic cube $Q$. Put $\Lambda'=\mathcal{D}^{+}\times \Psi'$.  
%\textcolor{red}{maybe leave dual basis out --> problems with notation later on...} 
%That is, there exist functions $\tilde{{\phi}}$ and $\tilde{\psi}_i$, $i=1,\ldots, 2^d-1$, such that conditions \eqref{wavelet-1}-\eqref{wavelet-3} hold if ${\phi}$ and $\psi$ are replaced by $\tilde{{\phi}}$ and $\tilde{\psi}$, respectively, and such that the biorthogonality relations 
%\[
%\langle \tilde{{\phi}}_k,\psi_{I}\rangle=\langle \tilde{\psi}_{I},{\phi}_k\rangle =0, \quad  \langle \tilde{{\phi}}_k, {\phi}_l\rangle=\delta_{k,l}, \quad %\langle \tilde{\psi}_{I},\psi_{I'}\rangle=\begin{cases}1, & I=I', \\ 0, & \text{otherwise}\end{cases},
%\]
%are fulfilled. Here we use analogous abbreviations as in \eqref{wavelet-4} for the dyadic shifts and dilations of $\tilde{{\phi}}$ and $\tilde{\psi}_i$, and $\delta_{k,l}$ denotes the Kronecker symbol. \\
Then, every function $f\in L_2(\real^d)$ can be written as 
\[
f=
%%P_0f+\sum_{(I,\psi)\in \Lambda'}\langle f, {\psi}_I\rangle \psi_I:=
\sum_{k\in \mathbb{Z}^d}\langle f,{{\phi}}_k\rangle {{\phi}}_k +\sum_{(I,\psi)\in \Lambda'}\langle f, {\psi}_I\rangle \psi_I.  
\]
Later on, it will be convenient to include ${\phi}$ into the set $\Psi'$. We use the notation ${\phi}_I:=0$ for $|I|<1$, ${\phi}_I={\phi}(\cdot-k)$ for $I=k+[0,1]^d$, and can simply write 
\[
f=\sum_{(I,\psi)\in \Lambda}\langle f, {\psi}_I\rangle \psi_I, \qquad \Lambda=\mathcal{D}^+\times \Psi, \quad \Psi=\Psi'\cup \{{\phi}\}.
\]

We describe Besov spaces on $\real^d$ by decay properties of the wavelet coefficients, if the parameters fulfill certain conditions.  \\

\begin{theorem}[{Wavelet decomposition of Besov spaces}]\label{thm-wavelet-dec}
Let $0<p,q<\infty$ and $s>\max\left\{0,d(1/p-1)\right\}$. Choose $r\in \nat$ such that $r>s$ and construct a wavelet Riesz basis as described above. Then a function $f\in L_p(\real^d)$ belongs to the Besov space $B^s_{p,q}(\real^d)$ if, and only if, 
\begin{equation}\label{besov-decomp}
f=\sum_{k\in \mathbb{Z}^d}\langle f,{{\phi}}_k\rangle {\phi}_k +\sum_{(I,\psi)\in \Lambda'}\langle f, {\psi}_I\rangle \psi_I  
\end{equation}
(convergence in $\mathcal{S}'(\real^d)$) with 
\begin{align}
\|f|B^s_{p,q}(\real^d)\| 
&\sim  \left(\sum_{k\in \mathbb{Z}^d} |\langle f,{{\phi}}_k\rangle|^p\right)^{1/p} + \notag\\
& \qquad   \left(\sum_{j=0}^{\infty}2^{j\left(s+d(\frac 12-\frac 1p)\right)q}\left(\sum_{(I,\psi)\in \mathcal{D}_j\times \Psi'}|\langle f, {\psi}_{I}\rangle|^p\right)^{q/p}\right)^{1/q}<\infty.\label{besov-norm}
\end{align}
%{(with the usual modification if $q=\infty$).}
\end{theorem}

{
\begin{remark}{
\begin{itemize}
\item[(i)] For parameters $q=\infty$ we  use the usual modification (replacing the outer sum by a supremum), i.e., 
\begin{align*}
\|f|B^s_{p,\infty}(\real^d)\|
&\sim \left(\sum_{k\in \mathbb{Z}^d} |\langle f,{{\phi}}_k\rangle|^p\right)^{1/p}+ \notag\\ 
& \qquad \sup_{j\geq 0}2^{j\left(s+d(\frac 12-\frac 1p)\right)}\left(\sum_{(I,\psi)\in \mathcal{D}_j\times \Psi'}|\langle f, {\psi}_{I}\rangle|^p\right)^{1/p}<\infty.  
\end{align*}
\item[(ii)] In particular, for our adaptivity scale \eqref{adaptivityscale}, i.e., $B^s_{\tau,\tau}(\real^d)$ with $s=d\left(\frac{1}{\tau}-\frac 1p\right)$,  we see that the norm  \eqref{besov-norm} becomes 
\begin{align}
\|f|B^s_{\tau,\tau}(\real^d)\|&\sim  \left(\sum_{k\in \mathbb{Z}^d} |\langle f,{{\phi}}_k\rangle|^{\tau}\right)^{1/\tau} +   \notag\\ 
 & \qquad  \left(\sum_{j=0}^{\infty}2^{jd\left(\frac 12-\frac 1p\right)\tau}\sum_{(I,\psi)\in \mathcal{D}_j\times \Psi'}|\langle f, {\psi}_{I}\rangle|^{\tau}\right)^{1/\tau}. \label{besov-norm2}
\end{align}
\end{itemize}
}
\end{remark}
}

Corresponding function spaces on domains can be introduced via restriction, i.e., 
\begin{eqnarray*}
B^s_{p,q}(\mathcal{O})&=& \left\{f\in \mathcal{D}'(\mathcal{O}): \ \exists g\in B^s_{p,q}(\real^d), \ g\big|_{\mathcal{O}}=f \right\},\\
\|f|B^s_{p,q}(\mathcal{O})\|&=& \inf_{g|_{\mathcal{O}}=f}\|f|B^s_{p,q}(\real^d)\|. 
\end{eqnarray*}
Alternative (different or equivalent) versions of this definition can be found, depending on possible additional properties of the distributions $g$ (most often their support). We refer to \cite{Tri08} for details and references.

\section{Regularity results in weighted Sobolev spaces}\label{Sect-3}

\subsection{Parabolic regularity results}

%\subsection{Besov regularity on smooth and polyhedral cones - using results from Luong and Loi}

\subsubsection{The fundamental problem}

Let $m\in \nat$. We consider the following first initial-boundary value problem 

\begin{equation} \label{parab-1a}
\left\{\begin{array}{rl}
\partt u+(-1)^m{L(x,t;D_x)}u\ =\ f \, &  \text{ in } K_T, \\
\frac{\partial^{k-1}u}{\partial \nu^{k-1}}\Big|_{\Gamma_{j,T}}\ =\ 0, & \   k=1,\ldots, m, \ j=1,\ldots, n,\\ 
u\big|_{t=0}\ =\ 0 \, & \text{ in } K.
\end{array} \right\}
\end{equation}

Here {$f$ is a function given on $K_T$, $\nu$ denotes the exterior normal to $\Gamma_{j,T}$}, and  the partial differential operator $L$ is given by
\[{L(x,t;D_x)}=\sum_{|\alpha|, |\beta|=0}^m D^{\alpha}_x({a_{\alpha \beta}(x,t)}D^{\beta}_x),\]
where $a_{\alpha \beta}$ are bounded real-valued functions from $C^{\infty}(K_T)$ with  $a_{\alpha \beta}=(-1)^{|\alpha|+|\beta|}{a}_{\beta \alpha}$. % and $\overline{a}_{\beta \alpha}$ denotes the conjugate of $a_{\beta \alpha}$. 
Furthermore, the operator $L$ is assumed to be strongly elliptic {{ uniformly with respect to $t\in (0,T)$}}, i.e., 
\begin{equation}\label{operator_L}
\sum_{|\alpha|, |\beta|=m}a_{\alpha \beta}\xi^{\alpha}\xi^{\beta}\geq c|\xi|^{2m} \qquad {\text{for all}}\quad  (x,t)\in K_T, \quad \xi\in \mathbb{R}^d.
\end{equation}

Moreover, a function $u\in {\mathring{W}}_p^{m,1}(K_T)$ is called a {\em generalized solution} of problem \eqref{parab-1a} if, and only if, $u(x,0)=0$ for all $x\in K$ and the equality 
\[
\left(\partial_t u, v\right)+(-1)^m \left({L(x,t;D_x)}u,v\right)=\left(f,v\right)\quad \text{a.e. } t\in [0,T],
\]
holds for all $v\in \mathring{W}^m_2(K)$. 

%\textcolor{red}{Muss da noch die Notation $(\cdot, \cdot)_{L_2(K)}$ erkl\"art werden?}

Concerning the Sobolev regularity of the generalized solution of problem \eqref{parab-1a} the following result may be found in \cite[Thm.~2.1., L.~3.1]{LL15}.

\begin{proposition}\label{GenSol_Sobolev}
Let $l\in \nat_0$ and assume that 
\begin{itemize}
\item[(i)] {$\sup \{|\partial_{t^k}a_{\alpha \beta}(x,t)|: \ (x,t)\in K_T,\  k\leq l+1\}<\infty$. }
\item[(ii)] $f\in W^l_2((0,T), L_2(K))$, \ $\partial_{t^k} f(x,0)=0$, \ $0\leq k\leq l-1$
\end{itemize}
{(where the second condition in (ii) is not applicable if $l=0$). Then  problem \eqref{parab-1a} has a unique  generalized solution  $u\in \mathring{W}_2^{m,l+1}(K_T)$ with a priori estimate 
\begin{equation}\label{form_gen_sol_sobolev}
\|u\|_{W^{m,l+1}_{2}(K_T)}\lesssim \|f\|_{W^l_2((0,T),L_2(K))}. 
\end{equation}
}
\end{proposition}

\subsubsection{Operator pencils} 

In order to state the global regularity results in weighted Sobolev spaces for our parabolic problem  we need to  define operator pencils generated by the Dirichlet problem for elliptic equations in the cone ${K}$. {Let us recall the basic facts, further informations on this subject may be found in  \cite[Sect.~2.3, Sect.~3.2.]{MR10}.} Let $M_k$ be an edge of the cone ${K}$, and let $\Gamma_{k_{\pm}}$ be the faces adjacent to $M_k$. Then by $\mathcal{D}_k$ we denote the dihedron which is bounded by the half-planes $\mathring{\Gamma}_{k_{\pm}}$ tangent to $\Gamma_{k_{\pm}}$ at $M_k$. Let $r,\varphi$ be polar coordinates in the plane perpendicular to $M_k$ such that 
\[
\mathring{\Gamma}_{k_{\pm}}=\left\{x\in \real^3: \ r>0, \ \varphi=\pm \frac{\theta_k}{2}\right\}.
\]
Fixing $t\in [0,T]$, we define the operator $A_k(\lambda,t)$ as follows: 
\[
A_k(\lambda,t)U=r^{2m-\lambda}L^{0}(0,t,D_x)(r^{\lambda}U), 
\]
where $u(x)=r^{\lambda}U(\varphi)$,  $\lambda\in \mathbb{C}$,  {$U$ is a function on $I_k:=\left(\frac{-\theta_k}{2}, \frac{\theta_k}{2}\right)$,}  and 
\[
L^{0}(0,t,D_x)=\sum_{|\alpha|=|\beta|=m}D^{\alpha}_x(a_{\alpha\beta}(0,t)D_x^{\beta}). 
\] 
The operator $A_k(\lambda,t)$ realizes a continuous mapping 
\[
W^{2m}_2(I_k)\cap \mathring{W}^m_2(I_k)\rightarrow L_2(I_k),
\]
for every $\lambda\in \mathbb{C}$. A complex number $\lambda_0$ is called an eigenvalue of the pencil $A_k(\lambda,t)$ if there exists  a nonzero function $U\in W^{2m}_2(I_k)\cap \mathring{W}^m_2(I_k)$ such that $A_k(\lambda_0,t)U=0$. We denote by {$\delta_{\pm}^{(k)}(t)$} the greatest positive real numbers such that the strip 
\[
m-1-\delta_{-}^{(k)}(t)<\mathrm{Re}\lambda<m-1+\delta_{+}^{(k)}(t)
\]
is free of eigenvalues of the pencil $A_k(\lambda,t)$. Furthermore, we put 
\[
{\delta_{\pm}^{(k)}}=\inf_{t\in [0,T]}{\delta_{\pm}^{(k)}(t)}, \qquad k=1,\ldots, n. 
\] 
Moreover, we introduce spherical coordinates $\rho=|x|$, $\omega=\frac{x}{|x|}$ in {$K$} and define 
\begin{equation}\label{op-pencil}
\mathfrak{U}(\lambda,t)U=\rho^{2m-\lambda}L^{0}(0,t,D_x)(\rho^{\lambda}U),
\end{equation}
where $u(x)=\rho^{\lambda}U(\omega)$ {and $U$ is a function on $\Omega=\{\omega:\ \omega=\frac{x}{|x|}, \ x\in K\}$}. 
The operator $\mathfrak{U}(\lambda,t)$ realizes a continuous mapping 
\[
W^{2m}_2(\Omega)\cap \mathring{W}_2^{m}(\Omega)\rightarrow L_2(\Omega).
\]
An eigenvalue of  $\mathfrak{U}(\lambda,t)$ is a complex number $\lambda_0$ such that $\mathfrak{U}(\lambda_0,t)U=0$ for some nonzero function $U\in W^{2m}_2(\Omega)\cap \mathring{W}_2^{m}(\Omega)$. 

{For the considerations concerning regularity that will be presented in the next subsection, we  need the following technical assumptions. }

\begin{assumption}\label{assumptions} 
Consider the operator pencil $\mathfrak{U}(\lambda,t)$, $t\in [0,T]$, from \eqref{op-pencil}. {For two  Kondratiev spaces $\calk^{\gamma}_{p,b}(K)$ and $\calk^{\gamma'}_{p,b'}(K)$  as  defined in \eqref{Kondratiev-1} with weight parameters $b,b'\in \real$,  we  assume that the closed strip between the lines $\mathrm{Re}\lambda=b+2m-\frac 32$ and $\mathrm{Re}\lambda=b'+2m-\frac 32$ does not contain eigenvalues of    $\mathfrak{U}(\lambda,t)$. Moreover, $b$ and $b'$ satisfy}
{\begin{equation}\label{restr-1-a}
%-\delta_+^{(k)}<\delta_k-m<\delta_{-}^{(k)}, \quad k=1,\ldots, n.
-\delta_-^{(k)}<b+m<\delta_{+}^{(k)}, \qquad -\delta_-^{(k)}<b'+m<\delta_{+}^{(k)}, \quad k=1,\ldots, n.
\end{equation}}
\end{assumption}

{\begin{rem}
{
Some remarks concerning Assumption \ref{assumptions}  seem to be in order. The values of $\delta_{\pm}^{(k)}$  determine the range of $b$ and $b'$. Let us assume that the operator $L$ from \eqref{operator_L} does not depend on the time $t$.  Then the results mentioned in \cite[p.~1]{Koz91} indicate that  the strip  $|\mathrm{Re}\lambda-m+1|\leq \frac 12$ contains no eigenvalues of $A_k(\lambda)$. Thus, we obtain  
$ 
m-\frac 32\leq \mathrm{Re}\lambda \leq m-\frac 12, 
$ 
i.e., it follows that $\delta_{\pm}^{(k)}= \frac 12$ which yields 
\[
-\frac 12<b+m<\frac 12. 
\]
If we additionally assume that our polyhedral cone $K$ is convex, i.e., $\theta_k\in (0,\pi)$, then the above results can be improved. As is stated in  \cite[p.~1]{Koz91} in this case we even know that the strip  $|\mathrm{Re}\lambda-m+1|\leq 1$ does not contain eigenvalues of $A_k(\lambda)$. Hence, 
$
-2+m\leq \mathrm{Re}\lambda\leq m
$ 
yielding $\delta^{(k)}_{\pm}=1$ and 
\[
-1<b+m<1. 
\]
In particular, for $m=1$ we see that $b\in (-2,0)$. 
Finally, let us  recall the example given in \cite[p.403]{LL15} where the heat equation ($L=\Delta$ and  $m=1$) is considered. In good agreement which the results above one knows  precisely in this case that   $\delta^{(k)}_{\pm}=\frac{\pi}{\theta_k}$ (i.e., $\delta_{\pm}\geq 1$ for angles $\theta_k\in (0,\pi)$ and worst case $\delta_{\pm}^{(k)}=\frac 12$ if $\theta_k=2\pi$). Therefore, \eqref{restr-1-a} now reads as 
\[
-\frac{\pi}{\theta_k}<{b}+1<\frac{\pi}{\theta_k}. 
\]
Moreover, we additionally require in Assumption \ref{assumptions} that the closed strip between the lines $\mathrm{Re}\lambda=b+2m-\frac 32$ and $\mathrm{Re}\lambda=b'+2m-\frac 32$  does not contain eigenvalues of the operator pencil $\mathfrak{A}(\lambda)$. 
Later on, in Theorem \ref{thm-weighted-sob-reg} we choose $b=a+2m\gamma_m$ and $b'=-m$, leading to the condition that  the strip  
\[
\left[m-\frac 32, a+2m(\gamma_m+1)-\frac 32\right], \qquad -m\leq a\leq m, 
\] 
is free of eigenvalues. We see that if the spectrum is real and discrete without cluster points,
then for $\gamma_m=0$ it is always possible to find some $a\in [-m,-m+\varepsilon)$, $\varepsilon>0$, satisfying our condition as long as $m-\frac 32$ is not an eigenvalue of $\mathfrak{A}(\lambda)$.  However, if we look at our nonlinear results established in  Theorem \ref{nonlin-B-reg1} the situation becomes more delicate. There we have the additional restrictions  $\gamma_m\geq 1$, $m\geq 2$, and $m\geq a\geq -\frac 12$. This gives for $\gamma_m=1$ and $m=2$ the condition that the strip  
\[
\left[\frac 12, a+\frac {13}{2}\right], \qquad -\frac 12\leq  a\leq 2, 
\]
is free of eigenvalues. Let us finally turn our attention once again to  the heat equation. In this case  the eigenvalues of $\mathfrak{A}(\lambda)$ are given by 
\[
\lambda^{\pm}=-\frac 12\pm \sqrt{\Lambda+\frac 14},
\] 
where $\Lambda$ denote the eigenvalues of the Laplace-Beltrami operator on $\Omega=K\cap S^2$. It is well known that the spectrum of the Laplace-Beltrami operator is a countable set of positive eigenvalues, cf. \cite[Sect.~2.2.1]{KMR01}. Hence, the interval $[-1,0]$ is free of eigenvalues of the pencils $\mathfrak{A}(\lambda)$. 
We denote the smallest positive eigenvalue of the $\mathfrak{A}(\lambda)$ by $\lambda_1^{+}$. In general it is not possible to determine precise values for $\lambda_1^+$ and arbitrary $\Omega$. 
%The the interval $[-1-\lambda_1^+,\lambda_1^+]$ does not contain eigenvalues of the pencils $\mathfrak{A}(\lambda)$. In particular, 
 But for the special case that $K$ is a smooth cone with  opening angle $\theta_0$ and $\Omega=\Omega_{\theta_0}$ a spherical cap, in  \cite[Fig.~7]{DS18}  one finds different values of $\lambda_1^+$ depending on $\theta_0\in (0,\pi)$. 
In particular, it can be seen that for small $\theta_0$ the eigenvalues may become quite large, e.g. we have $\lambda_1^+=1$ for $\theta_0=90^{\circ}$ and  even $\lambda_1^+>27$ for angle $\theta_0=5^{\circ}$. \\
}
%Using Theorem \ref{th-Sob-coinc} we see that condition \eqref{restr-1-a} for the spaces $\mathcal{K}^m_{p,a}$ becomes 
%\begin{equation}\label{restr-1-aa}
%-\delta_-^{(k)}<a+m<\delta_{+}^{(k)}, \quad k=1,\ldots, n.
%\end{equation}
\end{rem}
}

\subsubsection{Regularity results in weighted Sobolev spaces}  

Concerning weighted Sobolev regularity for the  parabolic problem  \eqref{parab-1a} first fundamental results can be found in 
{ \cite[Thms.~3.3,~3.4]{LL15}, which form the starting point for our investigations. However, for our purposes we slightly modify the results according to our needs and give a detailed proof. In particular, we obtain an {\em a priori} estimate for  the derivatives of the solution $u$ within our scale of Kondratiev spaces, which in turn is needed in Theorem \ref{nonlin-B-reg1} for proving  the existence of a solution of the nonlinear problem \eqref{parab-nonlin-1} in weighted Sobolev spaces. 
}

{For our considerations below we rely on known results for elliptic equations. Therefore, we now }
{
consider the following Dirichlet problem for elliptic equations 
\begin{equation}
\begin{cases}Lu=F& \text{on}\quad  {K},\\
\frac{\partial^k u}{\partial \nu^k}\big|_{\Gamma_j}=0, & k=1,\ldots, m, \ j=1,\ldots, n. 
\end{cases}\qquad \label{ellipt-pde}\end{equation}
The following lemma on the regularity of solutions to elliptic boundary value problems in domains of polyhedral type is taken from \cite[Cor.~4.1.10, Thm.~4.1.11]{MR10}. We rewrite it for  our scale of Kondratiev spaces.  
}

{
%-----------------------------------------------------------------------------
\begin{lemma}[Weighted Sobolev regularity for elliptic PDEs] \label{mazja_ross}\hfill  \\
Let $u\in \mathcal{K}^{\gamma}_{2,a+2m}(K)$ be a solution of \eqref{ellipt-pde}, where 
\[
F\in \mathcal{K}^{\gamma-2m}_{2,a}(K)\cap \mathcal{K}^{\gamma'-2m}_{2,a'}(K), \qquad \gamma\geq m, \quad \gamma'\geq m. 
\] 
{Suppose that $\mathcal{K}^{\gamma}_{2,a}(K)$ and $\mathcal{K}^{\gamma'}_{2,a'}(K)$ satisfy Assumption \ref{assumptions}.} 
%Suppose that the closed strip between the lines $\mathrm{Re}\lambda =a+2m-\frac 32$ and $\mathrm{Re}\lambda=a'+2m-\frac 32$ is free of eigenvalues of the pencil $\mathfrak{U}$ and that the components of $a$ and $a'$ satisfy the inequalities 
%\[
%-\delta_{-}^{(k)}<a+m<\delta_{+}^{(k)}, \qquad -\delta_{-}^{(k)}<a'+m<\delta_{+}^{(k)}.
%\]
 Then $u\in \mathcal{K}^{\gamma'}_{2,a'+2m}(K)$ and 
 \[
 \|u|\mathcal{K}^{\gamma'}_{2,a'+2m}(K)\|\leq C\|F|\mathcal{K}^{\gamma'-2m}_{2,a'}(K)\|, 
 \] 
where $C$ is a constant independent of $u$ and $F$. 
\end{lemma}
%------------------------------------------------------------------------------
}

{
\begin{rem}\label{rem_mazja_ross_ell}
In particular, we use the fact that from Prop. \ref{GenSol_Sobolev} we already know that if for $k=0,\ldots, l+1$
\begin{equation}\label{ell_1}
\partial_{t^k}f(t) \in L_2(K)=\mathcal{K}^0_{2,0}(K) \hookrightarrow \mathcal{K}^{-m}_{2,-m}(K), 
\end{equation}
where the latter embedding follows from the corresponding duality assertion, i.e., we have  
$ 
\mathcal{K}^m_{2,m}(K) \hookrightarrow \mathcal{K}^{0}_{2,0}(K)
$  
since $m\geq 0$, then the solution $u$ of  \eqref{parab-1a} satisfies 
\begin{equation}\label{ell_2}
\partial_{t^k}u(t)\in \accentset{\circ}{W}^m_2(K)\hookrightarrow \accentset{\circ}{\mathcal{K}}^m_{2,m}(K)\hookrightarrow \mathcal{K}^0_{2,a}(K), \qquad {a\leq m}, 
\end{equation}
where the first embedding is taken from \cite[Lem.~3.1.6]{MR10} and the second embedding for Kondratiev spaces holds whenever 
$m\geq a$. 
{We additionally require in our later considerations that 
\begin{equation}\label{ell_3}
\partial_{t^k}u(t)\in \calk^0_{2,a}(K)\hookrightarrow \calk^{-m}_{2,-m}(K), 
\end{equation}
which holds for $a\geq -m$. 
}
From \eqref{ell_1} and \eqref{ell_2} we see that it is possible to take $\gamma=m$ and $a=-m$ in  Lemma \ref{mazja_ross}, {
i.e., if  $f(t)\in \calk^{-m}_{2,-m}(K)$ then $u(t)\in \calk^m_{2,m}(K)$. {Note that all our arguments with $u(t)$ and $f(t)$, respectively, hold for a.e. $t\in [0,T]$. However, since lower order time derivatives are continuous w.r.t. suitable spaces (but not necessarily the highest one, cf. the proof of Thm. \ref{Hoelder-Besov-reg}), we will suppress this distinction in the sequel.}
}
\end{rem}
}

{
Using similar arguments as in \cite[Thm.~3.3]{LL15} we are now able to show the following regularity result in Kondratiev spaces.} 

{
\begin{theorem}[Weighted Sobolev regularity I]\label{thm-weighted-sob-reg}
Let $K\subset \real^3$ be a polyhedral cone. Let $\gamma\in \nat $ with  $\gamma\geq 2m$ and put $\gamma_m:=\lfloor \frac{\gamma-1}{2m}\rfloor$. Furthermore, let  $a\in \real$ with   ${a\in [-m,m]}$.  Assume that the right hand side $f$ of \eqref{parab-1a} satisfies 
\begin{itemize}
\item[(i)] $\partial_{t^k} f\in L_2(K_T)\cap L_2((0,T),\mathcal{K}^{2m(\gamma_m-k)}_{2,a+2m(\gamma_m-k)}(K))$, \ $k=0,\ldots, \gamma_m$, \\ and 
$\partial_{t^{\gamma_m+1}} f\in L_2(K_T)$. 
\item[(ii)] $\partial_{t^k} f(x,0)=0$, \quad  $k=0,1,\ldots, {\gamma_m}.$
\end{itemize}
{Furthermore, let  Assumption \ref{assumptions}  hold for weight parameters $b=a+2m(\gamma_m-i)$, where $i=0,\ldots, \gamma_m$, and  $b'=-m$.}
Then for the generalized solution $u\in {\mathring{W}}_2^{m,\gamma_m+2}(K_T)$ of problem \eqref{parab-1a} we have 
$$\partial_{t^{l+1}} u\in L_2((0,T),\mathcal{K}^{2m(\gamma_m-l)}_{2,a+2m(\gamma_m-l)}(K))$$ for $l=-1,0,\ldots, \gamma_m$. In particular, for the derivatives $\partial_{t^{l+1}} u$ up to order $\gamma_m+1$ we have the a priori estimate 
\begin{align}
\sum_{l=-1}^{\gamma_m}& \|\partial_{t^{l+1}} u|{L_2((0,T),\mathcal{K}^{2m(\gamma_m-l)}_{2,a+2m(\gamma_m-l)}(K))}\|\notag\\
&\lesssim  \sum_{k=0}^{\gamma_m}\|\partial_{t^k} f|{L_2((0,T), \mathcal{K}^{2m(\gamma_m-k)}_{2,a+2m(\gamma_m-k)}(K))}\|+\sum_{k=0}^{\gamma_m+1}\|\partial_{t^k} f|{L_2(K_T)}\|,\label{weighted-sobolev-est}
\end{align}
where the  constant is  independent of $u$ and $f$. 
\end{theorem}
}

{
\begin{proof}
We proof the theorem by  induction. Let $\gamma=2m$, then we have  $\gamma_m=0$.  Since by our assumptions 
$f, \partial_t f\in L_2(K_T)$ it follows from  Proposition \ref{GenSol_Sobolev} that $\partial_t u(t)\in   \mathring{W}^m_2(K)\subset \mathcal{K}^0_{2,a}(K)$ if ${a\leq m}$. In this case Proposition \ref{GenSol_Sobolev} gives the estimate 
\begin{equation}\label{apriori1}
\|\partial_t u|L_2((0,T),\calk^0_{2,a}(K))\|\lesssim \|\partial_t u|L_2((0,T),\mathring{W}^m_2(K))\|\leq \sum_{k=0}^1\|\partial_{t^k}f|L_2(K_T)\|. 
\end{equation}
Moreover, since 
\[
(-1)^m Lu=f-\partial_{t}u,
\]
where for fixed $t$ the right hand side belongs to $\calk^0_{2,a}(K)$ an application of Lemma \ref{mazja_ross} (with $\gamma=m$, $a=-m$, $\gamma'=2m$, $a'=a$) gives $u(t)\in \calk^{2m}_{2,a+2m}(K)$ with the a priori estimate 
\begin{align*}
\|u(t)|\calk^{2m}_{2,a+2m}(K)\|
&\lesssim \|f(t)|\calk^0_{2,a}(K)\|+\|\partial_t u|\calk^0_{2,a}(K)\|.%\\
%& \lesssim \|f(t)|\calk^0_{2,a}(K)\|+\sum_{k=0}^1\|\partial_{t^k}f(t)|L_2(K)\|,
\end{align*}
Now integration w.r.t. the parameter $t$ together with  \eqref{apriori1} proves the claim for $\gamma=2m$, i.e., 
\[
\|u|L_2((0,T),\calk^{2m}_{2,a+2m}(K)\|\lesssim \|f|L_2((0,T),\calk^0_{2,a}(K))\|+\sum_{k=0}^1\|\partial_{t^k}f|L_2(K_T)\|. 
\] 
Assume inductively that our assumption holds for $\gamma-1$. This means, in particular, that we have the following  a priori estimate 
\begin{align}
&\sum_{l=-1}^{(\gamma-1)_m} \|\partial_{t^{l+1}} u|{L_2((0,T),\mathcal{K}^{2m((\gamma-1)_m-l)}_{2,a+2m((\gamma-1)_m-l)}(K))}\|\notag\\
& \quad \lesssim  \sum_{k=0}^{(\gamma-1)_m}\|\partial_{t^k} f|{L_2((0,T), \mathcal{K}^{2m((\gamma-1)_m-k)}_{2,a+2m((\gamma-1)_m-k)}(K))}\|+\sum_{k=0}^{(\gamma-1)_m+1}\|\partial_{t^k} f|{L_2(K_T)}\|.\label{apriori2}
\end{align}
We are going to show that the claim then holds for $\gamma$ as well.  If $l=\gamma_m$  by our assumptions on $f$ we have  $\partial_{t^{k}} f\in L_2(K_T)$ for $k=0,\ldots, \gamma_m+1$, and by Proposition \ref{GenSol_Sobolev} together with  \eqref{ell_2} we have 
\[
\partial_{t^{\gamma_m+1}}u(t)\in \mathring{W}^m_2(K)\subset \mathcal{K}^0_{2,a}(K), \quad {a\leq m}. 
\]
In particular, Proposition \ref{GenSol_Sobolev} provides us with the a priori estimate 
\begin{align}\label{a_priori_1}
\|\partial_{t^{\gamma_m+1}}u|L_2((0,T), \mathcal{K}^0_{2,a}(K))\| &\lesssim \|\partial_{t^{\gamma_m+1}}u|L_2((0,T), \mathring{W}^m_2(K))\| \notag\\
& \lesssim \sum_{k=0}^{\gamma_m+1}\|\partial_{t^k} f| L_2(K_T)\|,  
\end{align}
which shows the claim for $l=\gamma_m$ and arbitrary $\gamma$. 
Hence, the claim holds for $l=\gamma_m$. 
%We now show that that result holds for $l=\gamma_m-1$. 
%Differentiating \eqref{parab-1a} $\gamma_m$-times gives (since we assume that $L$ does not depend on $t$) 
%\begin{equation}
%(-1)^mL (\partial_{t^{\gamma_m}}u)=\partial_{t^{\gamma_m}}f-\partial_{t^{\gamma_m+1}}u.
%\end{equation} 
%From our initial assumptions on $f$ we see that $\partial_{t^{\gamma_m}}f(t)\in \mathcal{K}^{0}_{2,a}(K)$ and from the considerations above we have  $\partial_{t^{\gamma_m+1}}u(t)\in \mathcal{K}^{0}_{2,a}(K)$. An application of Lemma \ref{mazja_ross} (with $\gamma'=2m$, $a'=a$, and where according to Remark \ref{rem_mazja_ross_ell} we can put $\gamma=m$, $a=-m$)  yields 
%\[
%\partial_{t^{(\gamma_m-1)+1}}u(t)\in \mathcal{K}^{2m}_{2,a+2m}(K). 
%\] 
%Moreover, 
%\begin{align*}
%\|\partial_{t^{(\gamma_m-1)+1}}u(t)| \mathcal{K}^{2m}_{2,a+2m}(K)\|
%&\lesssim \|\partial_{t^{\gamma_m}}f(t)-\partial_{t^{\gamma_m+1}}u(t)|\mathcal{K}^{0}_{2,a}(K)\|\\ 
%&\lesssim \|\partial_{t^{\gamma_m}}f(t)| \mathcal{K}^{0}_{2,a}(K)\|+\|\partial_{t^{\gamma_m+1}}u|\mathcal{K}^{0}_{2,a}(K)\|.\\
%%& \lesssim \|\partial_{t^{\gamma_m}} f|{L_2((0,T), \mathcal{K}^{0}_{2,a}(K))}\|+\sum_{k=0}^{\gamma_m+1}\|\partial_{t^k} f|{L_2(K_T)}\|
%\end{align*}
%Integration w.r.t. the parameter $t$ together with \eqref{a_priori_1} gives 
%\begin{align*}
%\|\partial_{t^{(\gamma_m-1)+1}}u& | L_2((0,T),\mathcal{K}^{2m}_{2,a+2m}(K))\|\\ 
%& \lesssim \|\partial_{t^{\gamma_m}} f|{L_2((0,T), \mathcal{K}^{0}_{2,a}(K))}\|+\sum_{k=0}^{\gamma_m+1}\|\partial_{t^k} f|{L_2(K_T)}\|
%\end{align*}
%This shows that the claim holds for $l=\gamma_m-1$. 
We proceed by backwards induction. 
Suppose now the result holds for $l=\gamma_m, \gamma_m-1, \ldots, i$ where $0\leq i\leq \gamma_m$. We show that it then also holds for $i-1$. Differentiating \eqref{parab-1a} $i$-times gives
\begin{equation}\label{diff-i}
(-1)^mL (\partial_{t^i}u)=\partial_{t^i}f-\partial_{t^{i+1}}u-(-1)^m\sum_{k=0}^{i-1}{i\choose k}\partial_{t^{i-k}}L (\partial_{t^k}u).
\end{equation} 
From our initial assumptions on $f$ we see that $\partial_{t^i}f(t)\in \mathcal{K}^{2m(\gamma_m-i)}_{2,a+2m(\gamma_m-i)}(K)$ and from the inductive assumptions it follows that $\partial_{t^{i+1}}u(t)\in \mathcal{K}^{2m(\gamma_m-i)}_{2,a+2m(\gamma_m-i)}(K)$ and 
\begin{align}
\partial_{t^{(k-1)+1}}u \ \in \  & L_2((0,T),\calk^{2m((\gamma-1)_m-(k-1))}_{2,a+2m((\gamma-1)_m-(k-1))}(K))\notag\\
&\hookrightarrow 
L_2((0,T),\calk^{2m(\gamma_m-1-(k-1))}_{2,a+2m((\gamma_m-1-(k-1))}(K))\notag\\
&\hookrightarrow 
L_2((0,T),\calk^{2m(\gamma_m-k)}_{2,a+2m((\gamma_m-k)}(K)) \notag\\
&\hookrightarrow 
L_2((0,T),\calk^{2m(\gamma_m-i+1)}_{2,a+2m((\gamma_m-i+1)}(K)),  \label{apriori-help}
\end{align} 
where we used  $(\gamma-1)_m=\lfloor \frac{\gamma-2}{2m}\rfloor \geq \gamma_m-1$ in the second step and the fact that $k=0,\ldots, i-1$ in the last step. From \eqref{apriori-help} we see that 
\[
\partial_{t^{i-k}}L (\partial_{t^k}u)\in L_2((0,T),\calk^{2m(\gamma_m-i)}_{2,a+2m((\gamma_m-i)}(K)), 
\]
hence, the right hand side of \eqref{diff-i} belongs to $ L_2((0,T),\calk^{2m(\gamma_m-i)}_{2,a+2m((\gamma_m-i)}(K))$. 
An application of Lemma \ref{mazja_ross} (now with $\gamma'=2m(\gamma_m-(i-1))$, $a'=a+2m(\gamma_m-i)$ and again taking $\gamma=m$, $a=-m$ according to Remark \ref{rem_mazja_ross_ell}) yields 
\[
\partial_{t^{(i-1)+1}}u(t)\in \mathcal{K}^{2m(\gamma_m-(i-1))}_{2,a+2m(\gamma_m-(i-1))}(K). 
\] 
Moreover, we have the a priori estimate 
\begin{align*}
\|&\partial_{t^{(i-1)+1}}u(t)| \mathcal{K}^{2m(\gamma_m-(i-1))}_{2,a+2m(\gamma_m-(i-1))}(K)\|\\
&\lesssim \|\partial_{t^i}f(t)| \mathcal{K}^{2m(\gamma_m-i)}_{2,a+2m(\gamma_m-i)}(K)\|+\|\partial_{t^{i+1}}u(t)|\mathcal{K}^{2m(\gamma_m-i)}_{2,a+2m(\gamma_m-i)}(K)\|\\
& \qquad + \sum_{k=0}^{i-1}\|\partial_{t^{i-k}}L(\partial_{t^k}u)(t)|\calk^{2m(\gamma_m-i)}_{2,a+2m(\gamma_m-i)}(K)\| \\
&\lesssim \|\partial_{t^i}f(t)| \mathcal{K}^{2m(\gamma_m-i)}_{2,a+2m(\gamma_m-i)}(K)\|+\|\partial_{t^{i+1}}u(t)|\mathcal{K}^{2m(\gamma_m-i)}_{2,a+2m(\gamma_m-i)}(K)\|\\
& \qquad + \sum_{k=0}^{i-1}\|(\partial_{t^k}u)(t)|\calk^{2m(\gamma_m-i+1)}_{2,a+2m(\gamma_m-i+1)}(K)\| \\
&\lesssim \|\partial_{t^i}f(t)| \mathcal{K}^{2m(\gamma_m-i)}_{2,a+2m(\gamma_m-i)}(K)\|+\|\partial_{t^{i+1}}u(t)|\mathcal{K}^{2m(\gamma_m-i)}_{2,a+2m(\gamma_m-i)}(K)\|\\
& \qquad + \sum_{k=0}^{i-1}\|(\partial_{t^k}u)(t)|\calk^{2m((\gamma-1)_m-(k-1))}_{2,a+2m((\gamma-1)_m-(k-1))}(K)\| \\
%& \lesssim \sum_{k=0}^{\gamma_m}\|\partial_{t^k} f|{L_2((0,T), \mathcal{K}^{2m((\gamma-1)_m-(k-1))}_{2,a+((\gamma-1)_m-(k-1))2m}(K))}\|+\sum_{k=0}^{\gamma_m+1}%\|\partial_{t^k} f|{L_2(K_T)}\|,
\end{align*}
where we used \eqref{apriori-help} in the last step. 
Integration w.r.t. the parameter $t$ together with the inductive assumptions on $\partial_{t^{i+1}}u$  and  $\partial_{t^k} u$ (cf.  \eqref{apriori2})  gives the a priori estimate 
\begin{align*}
\|&\partial_{t^{(i-1)+1}}u| L_2((0,T),\mathcal{K}^{2m(\gamma_m-(i-1))}_{2,a+2m(\gamma_m-(i-1))}(K))\|\\
&\lesssim \|\partial_{t^{i}} f|{L_2((0,T), \mathcal{K}^{2m(\gamma_m-i)}_{2,a+2m(\gamma_m-i)}(K))}\|+\|\partial_{t^{i+1}} u|{L_2((0,T), \mathcal{K}^{2m(\gamma_m-i)}_{2,a+2m(\gamma_m-i)}(K))}\| \\
& \quad + \sum_{k=0}^{i-1}\|(\partial_{t^k}u)|L_2((0,T),\calk^{2m((\gamma-1)_m-(k-1))}_{2,a+2m((\gamma-1)_m-(k-1))}(K))\| \\
& \lesssim \sum_{k=0}^{\gamma_m}\|\partial_{t^k} f|{L_2((0,T), \mathcal{K}^{2m(\gamma_m-k)}_{2,a+2m(\gamma_m-k)}(K))}\|+\sum_{k=0}^{\gamma_m+1}\|\partial_{t^k} f|{L_2(K_T)}\| \\
& \quad +  \sum_{k=0}^{(\gamma-1)_m}\|\partial_{t^k} f|{L_2((0,T), \mathcal{K}^{2m((\gamma-1)_m-k)}_{2,a+2m((\gamma-1)_m-k)}(K))}\|+\sum_{k=0}^{(\gamma-1)_m+1}\|\partial_{t^k} f|{L_2(K_T)}\| \\
& \lesssim  \sum_{k=0}^{\gamma_m}\|\partial_{t^k} f|{L_2((0,T), \mathcal{K}^{2m(\gamma_m-k)}_{2,a+2m(\gamma_m-k)}(K))}\|+\sum_{k=0}^{\gamma_m+1}\|\partial_{t^k} f|{L_2(K_T)}\|, 
\end{align*}
where in the last step we used the fact that $(\gamma-1)_m\leq \gamma_m$. This  
shows that the claim is true for $i-1$ and completes the proof. 
\end{proof}
}

\remark{\label{rem-restr}
 The existence of the solution $u\in \mathring{W}_2^{m,{\gamma_m+2}}(K_T)$  follows from  Proposition \ref{GenSol_Sobolev} {using  $l=\gamma_m+1$}.
%\item  In \cite[Thm.~3.4]{LL15} similar results as in the above Theorem are stated using a different scale of weighted Sobolev spaces denoted by  $V^{l,0}_{\beta}(K_T)$ containing all functions $u(x,t)$ such that  
%\begin{align}
%\|u\|_{V^{l,0}_{\beta}(K_T)}
%:=\left(\int_{K_T}\sum_{|\alpha|\leq l}\varrho^{2(\beta-l+|\alpha|)}|D_x^{\alpha}u(x,t)|^2\ud x\ud t\right)^{1/2}
%<\infty, \label{def-weighted-sobolev}
%\end{align}
%with  $\varrho(x)$, $l$ as before, and  $\beta\in \mathbb{R}$. Comparing this scale with our scale $L_2((0,T),\calk^{l}_{2,a}(K))$, we see that these spaces coincide if the parameters are linked via 
%\begin{equation}\label{link}
%\beta-l=-a. 
%\end{equation}
%The condition $-\delta_{+}^{(k)}<{\beta}-l+m<\delta_{-}^{(k)}$ from \cite{LL15} for the data in $V^{l-2m,0}_{\beta}(K_T)$ therefore had to be rewritten (now using \eqref{link} with $l$ replaced by $l-2m$) as 
%\[
%-\delta_{+}^{(k)}<(\beta+l-2m)-l+m<\delta_{-}^{(k)},\quad \text{i.e.,} \quad -\delta_{+}^{(k)}<\tilde{\delta_k}-m<\delta_{-}^{(k)}
%\]
%resulting in the assumption \eqref{restr-1-a} above. In particular, we see that our assumption here is independent of $\gamma\geq 2m$.
%\end{itemize}
%\end{comment}
}

{
\remark{\label{discuss_aa}\label{rem_heat_eq} We discuss the  restrictions on the parameter $a$ appearing in Theorem \ref{thm-weighted-sob-reg}. 
\begin{itemize}
\item[(i)] {If $K$ is a polyhedral cone} we  require 
\begin{align*}
1) \quad &-\delta^{(k)}_{-}<a+2m(\gamma_m-i)+m< \delta^{(k)}_{+}, \quad i=0,\ldots, \gamma_m,\\
2) \quad & {-m}\leq a\leq m.
%3) \quad & a>-2m(\gamma_m+1). 
\end{align*}
{Condition 1) can be removed if the cone $K$ is smooth without edges. }
If we consider the heat equation, then $m=1$ and  $\delta^{(k)}_{\pm}=\frac{\pi}{\theta_k}$, {cf. \cite[Sect.~4]{LL15}.} For the extremal case that $\theta_k=2\pi$ we obtain 
\begin{align*}
1) \quad &-\frac 32<a+2(\gamma_m-i)< -\frac 12, \quad i=0,\ldots, \gamma_m,\\
2) \quad  &{-1\leq } a\leq 1.
%3) \quad & a>-2(\gamma_m+1). 
\end{align*}
We see that 1) is only satisfied for $\gamma_m-i=0$. Thus, the theorem holds in this case only for $\gamma_m=0$ and $a\in \left({-1},-\frac 12\right)$, {i.e., regularity $\gamma\leq 2m=2$}.  \\
{The results improve for smaller angles, i.e., choosing $\theta_k=\frac{\pi}{4}$ we obtain the following conditions
\begin{align*}
1) \quad &-5<a+2(\gamma_m-i)< 3, \quad i=0,\ldots, \gamma_m,\\
2) \quad & -1\leq a\leq 1. 
%3) \quad & a>-2(\gamma_m+1). 
\end{align*}
%For $\gamma_m=0$ we can choose $a\in (-2,1)$. Otherwise we see that 3) (clearly also 2)) is covered by 1). From 1) 
We see that $-1<a$ and $a<3-2(\gamma_m-i)$. However, 
\[
-1<a<3-2(\gamma_m-i)
\]
is satisfied for $\gamma_m-i<2$, i.e., we {take $\gamma_m=1$ ($\gamma\leq 4m$)} in this case  and  $a\in \left(-1,1\right)$. } {The same phenomenon has been known  for elliptic problems: the shape of the domain limits the regularity of the solutions  even for smooth data. }\\
{For general angles $\theta_k$ the restrictions on $a$ for the heat equation read as 
\begin{align*}
1) \quad &-\frac{\pi}{\theta_k}-1<a+2(\gamma_m-i)< \frac{\pi}{\theta_k}-1, \quad i=0,\ldots, \gamma_m,\\
2) \quad  & -1\leq  a\leq 1. 
\end{align*}
1) and 2) together yield 
\[
-1\leq a<\min\left(1, \frac{\pi}{\theta_k}-1-2(\gamma_m-i)\right). 
\]
In particular, for the upper bound we see that 
\[\min\left(1, \frac{\pi}{\theta_k}-1-2(\gamma_m-i)\right)=
\begin{cases}
\min(1, \frac{\pi}{\theta_k}-1),& \gamma_m=0, \\
\min(1, \frac{\pi}{\theta_k}-3),& \gamma_m=1.
\end{cases}
\]
Therefore, when $\gamma_m=0$, since it holds that  $-1<\frac{\pi}{\theta_k}-1$,  it is always possible to find suitable parameters $a$.  On  the other hand for  $\gamma_m=1$ we require   $-1<\frac{\pi}{\theta_k}-3$, which is true for angles $\theta_k<\frac{\pi}{2}$ only. 
\item[(ii)] {Later on, in Section \ref{Sect-4} we want to study the  Besov regularity of problem \eqref{parab-nonlin-1}. This will be performed by using embedding results of Kondratiev spaces into Besov spaces. Since all functions in the adaptivity scale \eqref{adaptivityscale} of Besov spaces are locally integrable, the same must hold for the Kondratiev spaces which requires $a>0$. For the linear problem \eqref{parab-1a} this is no restriction since from 
 Theorem \ref{thm-weighted-sob-reg} it follows that our solution satisfies $u\in L_2((0,T),\mathcal{K}^{\gamma}_{2,a'}(K))$, where $a'=a+2m(\gamma_m+1)>0$. Thus, we always have a locally integrable solution in this case. \\
}
 }However, the above calculations in (i) tell us that for non-convex polyhedral cones, i.e., $\theta_k>\pi$ for some $k=1,\ldots, n$, we can only choose $\gamma_m=0$ in Theorem \ref{thm-weighted-sob-reg} and the restriction on $a$ becomes
 \[
 -1\leq a<\min\left(1, \frac{\pi}{\theta_k}-1\right)<0. 
 \]
%For the heat equation ($m=1$) we have $\delta_{\pm}^{(k)}=\frac{\pi}{\theta_k}$, {cf. \cite[Sect.~4]{LL15}.} Therefore,  Theorem \ref{thm-weighted-sob-reg}  with the restriction $a\in [-1,1]$ and   \eqref{restr-1-a} leads to 
%\[\max\left(-1-\frac{\pi}{\theta_k},-1\right)<a<\min\left(1,-1+\frac{\pi}{\theta_k}\right),\]
%which yields 
%\[
%-1<a<\min\left(1,-1+\frac{\pi}{\theta_k}\right)=
%\begin{cases}
%1 & \text{if }\ \theta_k<\frac{\pi}{2},\\ 
%>0 & \text{if }\ \theta_k<\pi. 
%\end{cases}
%\]
From this we see that it is possible to choose $a>0$ only as long as the polyhedral cone $K$ is convex. 
\end{itemize}
}
}

{The regularity results obtained in Theorem \ref{thm-weighted-sob-reg} {only hold under
certain restrictions on the parameter $a$ we are able to choose}. We  refer to the {previous} discussion in Remark \ref{discuss_aa}, where the heat equation is treated in detail. In particular, it turns out that we are not able to choose $\gamma_m>0$ if we have a non-convex polyhedral cone, since there is no suitable $a$ satisfying all of our requirements in this case.  Therefore, in order to be able to treat non-convex {cones as well, we need stronger
assumptions.  It turns out that we obtain a positive result if the
right--hand side $f$ is arbitrarily  smooth  with respect to time. With this
additional} assumption we are able to weaken the restrictions on the parameter $a$ and allow a larger range. {However, as a drawback, these results are
hard to apply to nonlinear equations since the right-hand sides are not
taken from a Banach or a quasi-Banach space.}
}

{
\begin{theorem}[Weighted Sobolev regularity II]\label{thm-weighted-sob-reg-2}
Let $K\subset \real^3$ be a polyhedral cone and  $\eta\in \nat $ with  $\eta\geq 2m$. Moreover, let  $a\in \real$ with   $a\in [-m,m]$.  Assume that the right hand side $f$ of \eqref{parab-1a} satisfies 
\begin{itemize}
\item[(i)] $ f\in \bigcap_{l=0}^{\infty}W^l_2((0,T), L_2(K)\cap \mathcal{K}^{\eta-2m}_{2,a}(K))$. 
\item[(ii)] $\partial_{t^l} f(x,0)=0$, \quad  $l\in \nat_0.$
\end{itemize}
Furthermore, let  Assumption \ref{assumptions}  hold for weight parameters $b=a$ and  $b'=-m$. 
Then for the generalized solution $u\in  \bigcap_{l=0}^{\infty}{\mathring{W}}_2^{m,l+1}(K_T)$ of problem \eqref{parab-1a} we have 
$$\partial_{t^{l}} u\in L_2((0,T),\mathcal{K}^{\eta}_{2,a+2m}(K))\qquad \text{for all}\quad l\in \nat_0. $$
In particular, for the derivative $\partial_{t^l} u$  we have the a priori estimate 
\begin{align}
\sum_{k=0}^{l}& \|\partial_{t^{k}} u|{L_2((0,T),\mathcal{K}^{\eta}_{2,a+2m}(K))}\|\notag\\
&\lesssim  \sum_{k=0}^{l+(\eta-2m)}\|\partial_{t^k} f|{L_2((0,T), \mathcal{K}^{\eta-2m}_{2,a}(K))}\|+\sum_{k=0}^{l+1+(\eta-2m)}\|\partial_{t^k} f|{L_2(K_T)}\|,\label{weighted-sobolev-est-2}
\end{align}
where the  constant is  independent of $u$ and $f$. 
\end{theorem}
}

{
\begin{proof}
We proof the theorem by  induction. Let $\eta=2m$.  Since by our assumptions 
$f, \partial_t f\in L_2(K_T)$ it follows from  Proposition \ref{GenSol_Sobolev} that $\partial_t u(t) \in   \mathring{W}^m_2(K)\subset \mathcal{K}^0_{2,a}(K)$ if ${a\leq m}$. In this case Proposition \ref{GenSol_Sobolev} gives the estimate 
\begin{equation}\label{apriori1a}
\|\partial_t u|L_2((0,T),\calk^0_{2,a}(K))\|\lesssim \|\partial_t u|L_2((0,T),\mathring{W}^m_2(K))\|\lesssim \sum_{k=0}^1\|\partial_{t^k}f|L_2(K_T)\|. 
\end{equation}
Moreover, since 
\[
 Lu=(-1)^m(f-\partial_{t}u)=:F,
\]
where for fixed $t$ the right hand side $F(t)$ belongs to $\calk^0_{2,a}(K)$ (and due to the fact that $a\geq -m$, we  have $\partial_t u(t) \in \mathring{W}^m_2(K)\hookrightarrow \calk^0_{2,a}(K)\hookrightarrow \calk^{-m}_{2,-m}(K)$, thus, also $F(t)\in \calk^{-m}_{2,-m}(K)$) an application of Lemma \ref{mazja_ross} (with $\gamma=m$, $a=-m$, $\gamma'=2m$, $a'=a$) gives $u(t)\in \calk^{2m}_{2,a+2m}(K)$ with the a priori estimate 
\begin{align*}
\|u(t)|\calk^{2m}_{2,a+2m}(K)\|
&\lesssim \|f(t)|\calk^0_{2,a}(K)\|+\|\partial_t u(t)|\calk^0_{2,a}(K)\|.%\\
%& \lesssim \|f(t)|\calk^0_{2,a}(K)\|+\sum_{k=0}^1\|\partial_{t^k}f(t)|L_2(K)\|,
\end{align*}
Now  integration w.r.t. the parameter $t$ and  \eqref{apriori1a}  prove the claim for $\eta=2m$ and $l=0$, i.e., 
\[
\|u|L_2((0,T),\calk^{2m}_{2,a+2m}(K))\|\lesssim \|f|L_2((0,T),\calk^0_{2,a}(K))\|+\sum_{k=0}^1\|\partial_{t^k}f|L_2(K_T)\|.
\]
We now assume that the claim is true for $\eta=2m$ and $k=0,\ldots, l-1$. Then differentiating  \eqref{parab-1a} $l$-times gives
\begin{equation}\label{diff-ia}
L (\partial_{t^l}u)=(-1)^m\left(\partial_{t^l}f-\partial_{t^{l+1}}u-(-1)^m\sum_{k=0}^{l-1}{l\choose k}\partial_{t^{l-k}}L (\partial_{t^k}u)\right)=:F.
\end{equation} 
By our initial assumptions $\partial_{t^l}f(t)\in \calk^0_{2,a}(K)$ and $\partial_{t^{l+1}}u(t)\in \mathring{W}^m_2(K)\hookrightarrow \calk^0_{2,a}(K)$. Moreover, the inductive assumptions provide us with  $\partial_{t^k}u(t)\in \calk^{2m}_{2,a+2m}(K)$, thus, $\partial_{t^{l-k}}L(\partial_{t^k}u)(t)\in \calk^{0}_{2,a}(K)$ and we see that $F(t)\in \calk^{0}_{2,a}(K)$. Applying Lemma \ref{mazja_ross} (with $\gamma=m$, $a=-m$, $\gamma'=2m$, $a'=a$) gives $\partial_{t^l}u(t)\in \calk^{2m}_{2,a+2m}(K)$ and we have the following a priori estimate 
\begin{align*}
\|\partial_{t^l}u(t)&| \mathcal{K}^{2m}_{2,a+2m}(K)\|\\
&\lesssim \|\partial_{t^l}f(t)| \mathcal{K}^{0}_{2,a}(K)\|+\|\partial_{t^{l+1}}u(t)|\mathcal{K}^{0}_{2,a}(K)\|\\
& \qquad + \sum_{k=0}^{l-1} \|\partial_{t^{l-k}}L (\partial_{t^k}u)(t)|\mathcal{K}^{0}_{2,a}(K)\| \\
&\lesssim \|\partial_{t^l}f(t)| \mathcal{K}^{0}_{2,a}(K)\|+\|\partial_{t^{l+1}}u(t)|\mathring{W}^m_2(K)\|\\
& \qquad + \sum_{k=0}^{l-1} \| (\partial_{t^k}u)(t)|\mathcal{K}^{2m}_{2,a+2m}(K)\|. % \\
%&\lesssim \sum_{k=0}^{l}\|\partial_{t^k}f(t)| \mathcal{K}^{0}_{2,a}(K)\|+\sum_{k=0}^{l+1}\|\partial_{t^k}f(t)| %L_2(K)\|. 
\end{align*}
Integration w.r.t. $t$, our inductive assumptions, and \eqref{form_gen_sol_sobolev}   give 
\begin{align*}
\|\partial_{t^l}u&| L_2((0,T),\mathcal{K}^{2m}_{2,a+2m}(K))\|\\
&\lesssim \sum_{k=0}^{l}\|\partial_{t^k}f| L_2((0,T),\mathcal{K}^{0}_{2,a}(K))\|+\sum_{k=0}^{l+1}\|\partial_{t^k}f| L_2(K_T)\|. 
\end{align*}
Assume now inductively that our assumption holds for $\eta-1$ and all derivatives $l\in \nat$. 
%This means, in particular, that we have the following  a priori estimate 
%\begin{align}
%&\sum_{l=-1}^{(\gamma-1)_m} \|\partial_{t^{l+1}} u|{L_2((0,T),\mathcal{K}^{2m((\gamma-1)_m-l)}_{2,a+2m((\gamma-1)_m-l)}(K))}\|\notag\\
%& \quad \lesssim  \sum_{k=0}^{(\gamma-1)_m}\|\partial_{t^k} f|{L_2((0,T), \mathcal{K}^{2m((\gamma-1)_m-k)}_{2,a+2m((\gamma-1)_m-k)}(K))}\|+\sum_{k=0}^{(\gamma-1)_m+1}\|\partial_{t^k} f|{L_2(K_T)}\|.\label{apriori2}
%\end{align}
We are going to show first that the claim then holds for $\eta$ and $l=0$ as well.  Looking at 
\[
Lu=(-1)^m(f-\partial_{t}u)=:F,
\]
we see that  $f(t)\in \calk^{\eta-2m}_{2,a}(K)$ and by our inductive assumption $\partial_t u(t)\in \calk^{\eta-1}_{2,a+2m}(K)\hookrightarrow \calk^{\eta-2m}_{2,a}(K)$. Therefore, $F(t)\in \calk^{\eta-2m}_{2,a}(K)$ and an application of Lemma \ref{mazja_ross} (with $\gamma=m$, $a=-m$, $\gamma'=\eta$, $a'=a$) gives $u(t)\in \calk^{\eta}_{2,a+2m}(K)$ with a priori estimate 
\begin{align*}
\|u(t)| \mathcal{K}^{\eta}_{2,a+2m}(K)\|
&\lesssim \|f(t)| \mathcal{K}^{\eta-2m}_{2,a}(K)\|+\|\partial_{t}u(t)|\mathcal{K}^{\eta-2m}_{2,a}(K)\| \\
&\lesssim \|f(t)| \mathcal{K}^{\eta-2m}_{2,a}(K)\|+\|\partial_{t}u(t)|\mathcal{K}^{\eta-1}_{2,a+2m}(K)\|.  
%&\lesssim \sum_{k=0}^{1}\|\partial_{t^k}f(t)| \mathcal{K}^{\eta-2m}_{2,a}(K)\|+\sum_{k=0}^{2}\|\partial_{t^k}f(t)| L_2(K)\|. 
\end{align*}
Integration w.r.t. the parameter $t$ and our inductive assumptions  show that the claim is true for $\eta$ and $l=0$, i.e., 
\begin{align*}
\|u&| L_2((0,T),\mathcal{K}^{\eta}_{2,a+2m}(K))\|\\
&\lesssim \sum_{k=0}^{1+(\eta-1-2m)}\|\partial_{t^k}f| L_2((0,T),\mathcal{K}^{\eta-2m}_{2,a}(K))\|+\sum_{k=0}^{2+(\eta-1-2m)}\|\partial_{t^k}f| L_2(K_T)\| \\
&= \sum_{k=0}^{\eta-2m}\|\partial_{t^k}f| L_2((0,T),\mathcal{K}^{\eta-2m}_{2,a}(K))\|+\sum_{k=0}^{1+(\eta-2m)}\|\partial_{t^k}f| L_2(K_T)\|. 
\end{align*} 
Suppose now that it is true for $\eta$ and derivatives $k=0,\ldots, l-1$.  Differentiating \eqref{parab-1a} $l$-times again gives
\eqref{diff-ia}. 
%L (\partial_{t^l}u)=(-1)^m\left(\partial_{t^l}f-\partial_{t^{l+1}}u-(-1)^m\sum_{k=0}^{l-1}{l\choose k}\partial_{t^{l-k}}L (\partial_{t^k}u)\right)=:F.
%\end{equation} 
From our initial assumptions on $f$ we see that $\partial_{t^l}f(t)\in \mathcal{K}^{\eta-2m}_{2,a}(K)$ and from the inductive assumptions it follows that $\partial_{t^{l+1}}u(t)\in \mathcal{K}^{\eta-1}_{2,a+2m}(K)\hookrightarrow \mathcal{K}^{\eta-2m}_{2,a}(K)$. Moreover, by the inductive assumptions $\partial_{t^k}u(t)\in \calk^{\eta}_{2,a+2m}(K)$ for $k=0,\ldots, l-1$ and therefore 
\[
\partial_{t^{l-k}}L (\partial_{t^k}u)(t)\in \calk^{\eta-2m}_{2,a}(K).
\]
This shows that the right hand side in \eqref{diff-ia} satisfies  $F(t)\in \calk^{\eta-2m}_{2,a}(K)$. Applying  Lemma \ref{mazja_ross} again (with $\gamma=m$, $a=-m$, $\gamma'=\eta$, $a'=a$) gives $u(t)\in \calk^{\eta}_{2,a+2m}(K)$ with a priori estimate 
\begin{align*}
\|&\partial_{t^l}u(t)| \mathcal{K}^{\eta}_{2,a+2m}(K)\|\\
&\lesssim \|\partial_{t^l}f(t)| \mathcal{K}^{\eta-2m}_{2,a}(K)\|+\|\partial_{t^{l+1}}u(t)|\mathcal{K}^{\eta-2m}_{2,a}(K)\|\\
& \qquad + \sum_{k=0}^{l-1} \|\partial_{t^{l-k}}L (\partial_{t^k}u)(t)|\mathcal{K}^{\eta-2m}_{2,a}(K)\| \\
&\lesssim \|\partial_{t^l}f(t)| \mathcal{K}^{\eta-2m}_{2,a}(K)\|+\|\partial_{t^{l+1}}u(t)|\mathcal{K}^{\eta-1}_{2,a+2m}(K)\|\\
& \qquad + \sum_{k=0}^{l-1} \| (\partial_{t^k}u)(t)|\mathcal{K}^{\eta}_{2,a+2m}(K)\|. \\
%&\lesssim \sum_{k=0}^{l}\|\partial_{t^k}f(t)| \mathcal{K}^{\eta-2m}_{2,a}(K)\|+\sum_{k=0}^{l+1}\|\partial_{t^k}%f(t)| L_2(K)\|. 
\end{align*}
Integration w.r.t. the parameter $t$ together with our inductive assumptions gives  
\begin{align*}
\|&\partial_{t^l}u| L_2((0,T),\mathcal{K}^{\eta}_{2,a+2m}(K))\|\\
&\lesssim \|\partial_{t^l}f(t)| \mathcal{K}^{\eta-2m}_{2,a}(K)\| \\ 
 & \quad +\sum_{k=0}^{l+1+(\eta-1-2m)}\|\partial_{t^k}f| L_2((0,T),\mathcal{K}^{\eta-2m}_{2,a}(K))\|+\sum_{k=0}^{l+2+(\eta-1-2m)}\|\partial_{t^k}f| L_2(K_T)\|\\
& \quad + \sum_{k=0}^{l-1+(\eta-2m)}\|\partial_{t^k}f| L_2((0,T),\mathcal{K}^{\eta-2m}_{2,a}(K))\|
 +\sum_{k=0}^{l-1+1+(\eta-2m)}\|\partial_{t^k}f| L_2(K_T)\| \\
&\lesssim \sum_{k=0}^{l+(\eta-2m)}\|\partial_{t^k}f| L_2((0,T),\mathcal{K}^{\eta-2m}_{2,a}(K))\|+\sum_{k=0}^{l+1+(\eta-2m)}\|\partial_{t^k}f| L_2(K_T)\|, 
\end{align*}
which finishes the proof. 
\end{proof}
}

{
\remark{In Theorem \ref{thm-weighted-sob-reg-2} compared to Theorem \ref{thm-weighted-sob-reg} we only require the parameter $a$ to satisfy $a\in [-m,m]$ and $\delta_{-}^{(k)}<a+m<\delta_{+}^{(k)}$ independent of the regularity parameter {$\eta$} which can be arbitrary high. In particular, for the heat equation this leads to the restriction 
\[
-1\leq  a<\min\left(1, \frac{\pi}{\theta_k}-1\right).  
\]
We see that even in the extremal case when $\theta_k=2\pi$ we can still take $-1\leq a<-\frac 12$ (resulting in $u\in L_2((0,T), \calk^{\eta}_{a+2m}(K))$ being locally integrable since $a+2m>0$) and choose {$\eta$} arbitrary high, this way also covering non-convex polyhedral cones with our results. 
}
}

\subsection{Hyperbolic regularity results}

{In this subsection we recall regularity results of linear hyperbolic equations of second order from  \cite{LT15}. }

\subsubsection{The fundamental problem} We consider the following initial-boundary value problem 

\begin{equation} \label{hyp-1a}
\left\{\begin{array}{rl}
\frac{\partial^2}{\partial t^2}u+L(x,t,D_x) u \ = \ f(x,t)  & \text{ in } \Omega_T, \\[0.1cm]
u(x,0) \ = \ \partt u(x,0)  \ = \ 0  \  \qquad & \text{ in } \Omega,\\[0.1cm]
u\big|_{\partial \Omega\times (0,T)} \ =\ 0, \  \qquad &    \\
\end{array}
\right\}
\end{equation}
where $\Omega$ is the special Lipschitz domain from Definition \ref{def_special_Lip} and  $L$ is  a linear differential operator of second order on $\Omega_T$ of the following form 
\[
L(x,t,D_x)u=-\sum_{i,j=1}^{{d}}\frac{\partial}{\partial x_j}\left(a_{ij}(x,t)\frac{\partial u}{\partial x_i}\right)
+\sum_{i=1}^{{d}}b_i(x,t)\frac{\partial u}{\partial x_i}+c(x,t)u,
\]
where $a_{ij}(x,t)$, $b_i(x,t)$, and $c(x,t)$ are real-valued functions on $\Omega_T$ belonging to 
$C^{k+1}(\Omega_T)$, {$k\in \nat_0$}. Moreover, assume that the coefficients of $L$ and {their} derivatives are bounded on $\Omega_T$. Suppose that $a_{ij}=a_{ji}$  ($i,j=1,\ldots, n$) are continuous in $x\in \overline{\Omega}$ uniformly with respect to $t\in [0,T]$ and 
\[
\sum_{i,j=1}^{{d}} a_{ij}(x,t)\xi_i\xi_j\geq \mu_0|\xi|^2
\]
for all $\xi\in \rn\setminus \{0\}$ and $(x,t)\in \Omega_T$, where $\mu_0$ is a positive constant.  It is possible to reduce the operator $L$ with coefficients at $P\in l_0$, $t\in (0,T)$, to its canonical form 
\[
L_0^{(2)}:=-\sum_{i,j=1}^2 a_{ij}(P,t)\frac{\partial^2}{\partial x_i\partial x_j}, 
\]
{cf. \cite[p.~460]{LT15} and the references given there}. 
Via this reduction it can be realized, that after a linear transformation of coordinates { the half-spaces $T_1(P)$ and $T_2(P)$ (see Definition \ref{def_special_Lip} for details)} go over into hyperplanes $T_1^{'}$ and $T_2^{'}$, respectively. Furthermore, the angle $\beta$ at $(P,t)$ is transformed to 
\begin{equation}\label{omega}
\omega(P,t)=\arctan \frac{\left[a_{11}(P,t)a_{22}(P,t)-a_{12}^2(P,t)\right]^{1/2}}{a_{22}(P,t)\cot \beta -a_{12}(P,t)}.
\end{equation}
The value $\omega(P,t)$ does not depend on the method by which $L_0^{(2)}$ is reduced to its canonical form. Moreover, the function $\omega(P,t)$ is infinitely differentiable and $\omega(P,t)>0$. Since $\Omega$ is bounded it follows that the manifold $l_0$ is compact and we put 
\begin{equation}\label{omega_max}
\omega:=\max_{P\in l_0, t\in [0,T]}\omega(P,t).
\end{equation}

A function  $u(x,t)$ is called a {\em generalized solution} of problem \eqref{hyp-1a} on $[0,T]$, if, and only if, $u\in L_2([0,T], \mathring{W}^1_2(\Omega))$, $\partial_t u\in L_2([0,T], L_2(\Omega))$, $\partial_{t^2}u \in L_2([0,T], W^{-1}_2(\Omega))$ such that $u(x,0)=\partial_t u(x,0)=0$ and the equality 
\[
(\partial_{t^2}u(\cdot, t),v)+ B[u(t),v;t]
%\int_{\Omega}\left(\sum_{i,j=1}^n a_{ij}(x,t)\frac{\partial u}{\partial x_i}\frac{\partial v}{\partial x_j} +\sum_{i=1}^n b_i(x,t)\frac{\partial u}{\partial x_i}v 
%+c(x,t)uv\right) \ud x
=(f(\cdot, t),v)
\]
holds for all $v\in \mathring{W}^1_2(\Omega)$ and for all $t\in [0,T]$, where 
\[
B[u,v;t]
=\int_{\Omega}\left(\sum_{i,j=1}^{{d}} a_{ij}(x,t)\frac{\partial u}{\partial x_i}\frac{\partial v}{\partial x_j} +\sum_{i=1}^{{d}} b_i(x,t)\frac{\partial u}{\partial x_i}v +c(x,t)uv\right) \ud x. 
\]

\subsubsection{Regularity results in weighted Sobolev spaces}

Concerning the regularity of the solution of \eqref{hyp-1a} in weighted Sobolev spaces,  a reformulation of the results from  \cite[Thms. 2.1, 2.2]{LT15} yields the following.  

%%%%%%%%%%%%%%%%%%%%%%%%%%%%%%%%nonlinear version
%\begin{lemma}\label{thm-hyperbolic-weight-sob-reg}
%Let $m\in \nat$, $m\geq 2$, and assume that \textcolor{red}{$f\in C^{(m-1),1}(\Omega_T,\real^{d+1})$} such that 
%$f_{t^j}\big|_{t=0}=0$, $0\leq j\leq k-2$. Moreover, if $\max(m-1,\frac d2+1)<a<\min(m,-1+\frac{\pi}{\omega})$, then %problem \eqref{hyp-1}-\eqref{hyp-3} has a unique local solution $u\in \mathcal{K}^{m,\infty}_{2,a}(\Omega_T)$ and 
%\[
%\|u| \mathcal{K}^{m,\infty}_{2,a}(\Omega_T)\|\lesssim 1.
%\]
%\end{lemma}

\medskip 

\begin{lemma}\label{thm-hyperbolic-weight-sob-reg}
Let $\Omega\subset \real^d$ be a special Lipschitz domain from Definition \ref{def_special_Lip}. Furthermore,  let $m\in \nat$, $m\geq 2$, $m-1\leq a<\min(m,\frac{\pi}{\omega}-1)$, and   assume $f$ satisfies  
\begin{itemize}
\item[(i)] $\partial_{t^j} f \in  L_{\infty}((0,T),\mathcal{K}^{m-j}_{2,a-j}(\Omega))$, \ $0\leq j\leq m-1$, 
\item[(ii)] $\partial_{t^j} f(x,0)=0 $, \ $0\leq j\leq m-2$.
\end{itemize}
Then for the generalized solution $u$ of problem \eqref{hyp-1a} we have 
$$
\partial_{t^j} u \in L_{\infty}((0,T),W^1_2(\Omega))\cap L_{\infty}((0,T),\mathcal{K}^{m-j}_{2,a-j}(\Omega)), \qquad 0\leq j\leq m, 
$$ 
and the following a priori estimate holds 
\[
\sum_{j=0}^{m} \|\partial_{t^j} u| L_{\infty}((0,T),\mathcal{K}^{m-j}_{2,a-j}(\Omega))\|\lesssim \sum_{j=0}^{m-1} \left\|\partial_{t^j} f|L_{\infty}((0,T),\mathcal{K}^{m-j}_{2,a-j}(\Omega))\right\|. 
\]
\end{lemma}

%\remark{Lemma \ref{thm-hyperbolic-weight-sob-reg} also holds for $a\geq\frac d2$ in case that the nonlinear term $f$ has the simpler form $f(x,t,u)$.}

\remark{For the restriction on $a$ in Lemma \ref{thm-hyperbolic-weight-sob-reg} to make sense we require 
\[
m-1<\frac{\pi}{\omega}-1,\quad {\text{i.e.,}} \quad \omega<\frac{\pi}{m}.
\]
{Unfortunately, this  causes a restriction on the angle $\beta$ of our domain $\Omega$, i.e., $\beta$ has to be small, cf. \eqref{omega} and \eqref{omega_max}. }
}

\medskip

\section{Parabolic Besov regularity on polyhedral cones}\label{Sect-4}

\subsection{Besov regularity of linear parabolic PDEs}

{To prove our results on Besov regularity, we will heavily use embeddings of Kondratiev spaces into Besov spaces. The following theorem that  can be found in \cite[Sect.~5, Thm.~3]{Han15} will be essential in this context.}

\begin{theorem}\label{thm-hansen-1}
Let $D$ be {a} bounded polyhedral domain in $\mathbb{R}^d$. Furthermore, let $s, a\in \real$, $\gamma\in \nat_0$, and  suppose $\min(s,a)>\frac{\delta}{d}\gamma$, where $\delta$ denotes the dimension of the singular set (i.e. $\delta=0$ if there are only vertex singularities, $\delta=1$ if there are edge and vertex singularities etc.). Then there exists some $0<\tau_0\leq p$ such that 
\begin{equation}
\calk^{\gamma}_{p,a}(D)\cap B^s_{p,\infty}(D)\hookrightarrow B^{\gamma}_{\tau,\infty}(D)\hookrightarrow L_p(D),
\end{equation}
for all $\tau_{\ast}<\tau<\tau_0$, where $\frac{1}{\tau_{\ast}}=\frac {\gamma}{d}+\frac 1p$. 

\end{theorem}

\begin{remark}{It also holds that $u\in B^{\gamma}_{\tau,\infty}(D)$ for $\tau\leq \tau_{\ast}$ but these spaces are no longer embedded into $L_p(D)$. {Moreover, since we are interested in embeddings into the adaptivity scale $B^s_{\tau,\tau}(D)$, cf.  \eqref{adaptivityscale}, we later on make use of the embedding 
\[
 B^{\gamma}_{\tau,\infty}(D)\hookrightarrow  B^{\gamma-\varepsilon}_{\tau,\tau}(D), \qquad \varepsilon>0. 
\]
}}
\end{remark}

\medskip

The embedding from Theorem \ref{thm-hansen-1} immediately generalizes to the function spaces defined in \eqref{Kondratiev-3} as follows:

\begin{theorem}\label{thm-hansen-gen}
Let $D$ be some bounded polyhedral domain in $\mathbb{R}^d$ {and assume $k\in \nat_0$ and $0<q\leq \infty$}. Furthermore, let $s, a\in \real$, $\gamma\in \nat_0$, and  suppose $\min(s,a)>\frac{\delta}{d}\gamma$, where $\delta$ denotes the dimension of the singular set. Then there exists some $0<\tau_0\leq p$ such that 
{\begin{equation}\label{emb-hansen-gen-sob}
W^k_q((0,T),\calk^{\gamma}_{p,a}(D))\cap W^k_q((0,T),B^s_{p,\infty}(D))\hookrightarrow W^k_q((0,T),B^{\gamma}_{\tau,\infty}(D))  
\end{equation}}
for all $\tau_{\ast}<\tau<\tau_0$, where $\frac{1}{\tau_{\ast}}=\frac {\gamma}{d}+\frac 1p$. 
\end{theorem}

{
\begin{proof}
Put $X_1:=\calk^{\gamma}_{p,a}(D)$, $X_2:=B^s_{p,\infty}(D)$, and $X=B^{\gamma}_{\tau,\infty}(D)$. Then Theorem \ref{thm-hansen-1} states that 
$$X_1\cap X_2\hookrightarrow X, $$
i.e., for some $x\in X_1\cap X_2$ we have $\|x|X\|\lesssim \|x|X_1\cap X_2\|\sim \|x|X_1\|+\|x|X_2\|$. Using this we calculate for $I:=(0,T)$ that 
\begin{align*}
\|u|&W^k_q((0,T),B^{\gamma}_{\tau,\infty}(D))\|\\
&= \left\|u|W^k_q(I,X)\right\|=\left(\sum_{l=0}^k \|\partial_{t^l}u|L_{q}(I,X)\|^q\right)^{1/q}\\
&\lesssim   \left(\sum_{l=0}^k\int_I\|\partial_{t^l}u(\cdot,t)|X\|^q\ud t\right)^{1/q}\\
&\lesssim   \left(\sum_{l=0}^k\int_I\|\partial_{t^l}u(\cdot,t)|X_1\cap X_2\|^q\ud t\right)^{1/q}\\
&\sim   \left(\sum_{l=0}^k\int_I\|\partial_{t^l}u(\cdot,t)|X_1\|^q\ud t\right)^{1/q}+ \left(\sum_{l=0}^k\int_I\|\partial_{t^l}u(\cdot, t)|X_2\|^q\ud t\right)^{1/q}\\
&= \left\|u|W^k_q(I,X_1)\right\|+\left\|u|W^k_q(I,X_2)\right\|\\ 
&=\left\|u|W^k_q(I,X_1)\cap W^k_q(I,X_2)\right\|\\
&=\left\|u|W^k_q((0,T),\calk^{\gamma}_{p,a}(D))\cap W^k_q((0,T),B^s_{p,\infty}(D))\right\|,
\end{align*}
which establishes \eqref{emb-hansen-gen-sob}. 
\end{proof}
}

{
\remark{For $k=0$ the embedding \eqref{emb-hansen-gen-sob} in Theorem \ref{thm-hansen-gen} reads as 
\begin{equation}\label{hansen_gen_k0}
L_q((0,T),\calk^{\gamma}_{p,a}(D))\cap L_q((0,T),B^s_{p,\infty}(D))\hookrightarrow L_q((0,T),B^{\gamma}_{\tau,\infty}(D)). 
\end{equation}
}
}

\paragraph{Unbounded cone  versus bounded polyhedral domain} We wish to combine the results from \cite{LL15} as stated in Theorem \ref{thm-weighted-sob-reg} with the embedding results  in Theorem \ref{thm-hansen-gen}. The problem arises that Theorem \ref{thm-weighted-sob-reg} holds for unbounded cones $K\subset \real^3$ whereas the embedding results in Theorem \ref{thm-hansen-gen} are true for bounded polyhedral domains $D\subset \real^d$. In order to {avoid this problem} we consider the truncated cone $K_0$ as defined in \eqref{trunc-cone}.  
%and additionally to  conditions (i) and (ii) in Theorem \ref{thm-weighted-sob-reg} assume that $f$ satisfies  
%\[
%\supp  f \subset K_{0,T}. 
%\]
{Then, the additional difficulty occurs that the Kondratiev norm on the truncated cone is not just defined by restriction. Instead, the distance to the new corners produced by the truncation from considering $K_{0}$ instead of $K$ have to be taken into account. } We {solve} this problem by multiplying $u$ with a radial cut-off function ${\varphi}\in C_0^{\infty}(K_0)$ satisfying \\
\begin{minipage}{0.6\textwidth} 
\begin{equation}\label{cutoff}
 {\varphi}(x)\equiv \begin{cases}
1  & \text{on} \quad  \{|x|<r_0-\varepsilon\}\cap  K_0, \\
  0  & \text{on} \quad  \left\{|x|>r_0-\frac{\varepsilon}{2}\right\}\cap K_0.
 \end{cases}
 \end{equation}
 \end{minipage}\hfill
 \begin{minipage}{0.3\textwidth}
\includegraphics[width=3.5cm]{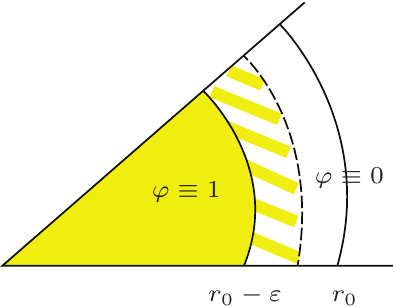}
\captionof{figure}{Illustration of cut-off function $\varphi$}
\end{minipage}\\
This truncation process does not induce serious restrictions for when it comes to practical applications, it is clear that only truncated cones can be considered. 
 Then the regularity of ${\varphi} u$ corresponds to the regularity of $u$ as stated in Theorem \ref{thm-weighted-sob-reg} and  we obtain  
 \begin{align*}
 \|{\varphi} u| L_2((0,T),\mathcal{K}^{\gamma}_{2,a'}(K_0))\|
& \lesssim \|{\varphi} u| L_2((0,T),\mathcal{K}^{\gamma}_{2,a'}(K))\|\\
&\leq c_{{\varphi}}\| u| L_2((0,T),\mathcal{K}^{\gamma}_{2,a'}(K))\|,
 \end{align*}
 from \eqref{multiplier} {with $\gamma=2m(\gamma_m+1)$ and $a'=a+2m(\gamma_m+1)$}.  
Now we are in a position to apply  the embedding results from Theorem \ref{thm-hansen-gen} {when $k=0$, cf. \eqref{hansen_gen_k0},} to the function ${\varphi} u$, which together with the regularity results for weighted Sobolev spaces from Theorem \ref{thm-weighted-sob-reg}  yield  maximal Besov regularity of the solution of the parabolic problem \eqref{parab-1a}. \\

\begin{theorem}[Parabolic Besov regularity I]\label{thm-parab-Besov}
Let $K\subset \real^3$ be a polyhedral cone. Let $\gamma\in \nat $ with  ${\gamma\geq 2m}$ and put $\gamma_m:=\lfloor \frac{\gamma-1}{2m}\rfloor$. Furthermore, let  $a\in \real$ with   ${a\in [-m,m]}$.  Assume that  the right hand side $f$ of \eqref{parab-1a} satisfies 
{
\begin{itemize}
\item[(i)] $\partial_{t^k} f\in L_2(K_T)\cap L_2((0,T),\mathcal{K}^{2m(\gamma_m-k)}_{2,a+2m(\gamma_m-k)}(K))$, \ $k=0,\ldots, \gamma_m$, \\ and 
$\partial_{t^{\gamma_m+1}} f\in L_2(K_T)$. 
\item[(ii)] $\partial_{t^k} f(x,0)=0$, \quad  $k=0,1,\ldots, {\gamma_m}.$
\end{itemize}
}
{Furthermore, let  Assumption \ref{assumptions}  hold for weight parameters $b=a+2m(\gamma_m-i)$, where $i=0,\ldots, \gamma_m$, and  $b'=-m$.  }
%$\supp f_{t^k}\subset K_{0,T}$ for $k=0,\ldots, l+1$ as well as 
%$$f_{t^k}\in L_2((0,T),\mathcal{K}^{\gamma-2m}_{2,a}(K)), \ k=0,\ldots, l+1,\quad \text{and}\quad 
%f_{t^k}(x,0)=0, \ k=0,\ldots, l.  $$
Let ${\varphi}$ denote the cutoff function  from \eqref{cutoff}. Then 
%there exists some $0<\tau_0\leq 2$ such that 
for the generalized solution $u\in {\mathring{W}}_2^{m,{\gamma_m+2}}(K_T)$ of problem \eqref{parab-1a}, we have 
\begin{equation}\label{parab-Besov}
{\varphi} u\in L_{2}((0,T),B^{\eta}_{\tau,\infty}(K)) \qquad \text{for all}\quad {0<\eta<\min(\gamma, 3m), } \quad  \frac 12 <\frac{1}{\tau}<\frac{\eta}{3}+\frac 12.     
%\text{for all} \quad  \tau_{\ast}<\tau<\tau_0 \quad \text{for all} \frac{1}{\tau^{\ast}}=\frac{\gamma}{d}+\frac 12, 
\end{equation}
%\textcolor{blue}{where 
%\begin{equation}\label{cond-gamma}
%  0<\gamma<3\min(m,a+2m(\gamma_m+1)). 
%\end{equation}
%}
In particular, for any $\eta, \tau$ satisfying \eqref{parab-Besov}, we have the a priori estimate 
\begin{align*}
\|{\varphi} u|& L_{2}((0,T),B^{\eta}_{\tau,\infty}(K))\|\\
&\lesssim  \sum_{k=0}^{{\gamma_m}}\|\partial_{t^k} f|{L_2((0,T),{\mathcal{K}^{2m(\gamma_m-k)}_{2,a+2m(\gamma_m-k)}(K))}}\|+\sum_{k=0}^{{\gamma_m}+1}\|\partial_{t^k} f|{L_2(K_T)}\|. 
\end{align*}
\end{theorem}

\begin{proof}
According to Theorem  \ref{thm-weighted-sob-reg} by our assumptions we know  ${\varphi} u\in L_2((0,T), {\mathcal{K}^{2m(\gamma_m+1)}_{2,a+2m(\gamma_m+1)}(K)})$. Together with  Theorem  \ref{thm-hansen-gen} {(choosing $k=0$)} we obtain  
\begin{align*}
{\varphi} u\in &L_2((0,T),{\mathcal{K}^{2m(\gamma_m+1)}_{2,a+2m(\gamma_m+1)}(K)})\cap {\mathring{W}}_2^{m,{\gamma_m+2}}(K_T)\\
&\hookrightarrow  L_2((0,T),{\mathcal{K}^{2m(\gamma_m+1)}_{2,a+2m(\gamma_m+1)}(K)})\cap L_2((0,T),W_2^m(K))\\
&\hookrightarrow   L_2((0,T),{\mathcal{K}^{2m(\gamma_m+1)}_{2,a+2m(\gamma_m+1)}(K)})\cap L_2((0,T),B^m_{2,\infty}(K))\\
&{\hookrightarrow   L_2((0,T),\mathcal{K}^{\eta}_{2,a+2m(\gamma_m+1)}(K)\cap B^m_{2,\infty}(K))}\\
%&{\hookrightarrow   L_2((0,T),\mathcal{K}^{\eta}_{2,a+2m(\gamma_m+1)}(K)\cap B^m_{2,\infty}(K))}\\
&\hookrightarrow    L_2((0,T),B^{\eta}_{\tau,\infty}(K)),   
\end{align*}
{where in the third step we use the fact that  {$ 2m(\gamma_m+1)\geq 2m\left(\frac{\gamma}{2m}-1+1\right)=\gamma$ and choose $\eta\leq \gamma$.}}  Moreover,   the {condition} on $a$ from Theorem \ref{thm-hansen-gen} yields  
$$
{m=}{\min(m,a+2m(\gamma_m+1))}>\frac{\delta}{d}\eta=\frac{\eta}{3}.   
$$ 
{Therefore, the upper bound for $\eta$ is   $$\eta<{\min(\gamma,3m)}. $$ }
Concerning the restriction on $\tau$, Theorem  \ref{thm-hansen-gen} with $\tau_0=2$ gives 
\[
\frac 12<\frac{1}{\tau}<\frac{1}{\tau^{\ast}}=\frac{\eta}{3}+\frac 12.
\]

\end{proof}

{
\remark{\label{discuss_a}
We discuss the role of the weight parameter in our Kondratiev spaces. Note that on the one hand we require $a+2m(\gamma_m+1)>0$ in order to apply the embedding of our Kondratiev spaces  into Besov spaces from Theorem \ref{thm-hansen-gen}. Since we assume $a\in [-m,m]$ this is always true. On the other hand it should be expected that the derivatives of the solution $u$ have singularities near the boundary of the polyhedral cone. Thus, looking at the highest derivative of $u(t)\in \calk^{2m(\gamma_m+1)}_{2,a+2m(\gamma_m+1)}(K)$ we see that we require 
\[
\sum_{|\alpha|=2m(\gamma_m+1)}\int_K \rho(x)^{-ap}|\uD^{\alpha}u(x,t)|^p\ud x<\infty, 
\]
hence, if $a<0$ the derivatives of the solution $u$ might be unbounded near  the boundary of $K$. From this it follows that the range 
\[
{-m}<a<0 
\]
is the most interesting for our considerations. 
}
}

{
\remark{\label{gen-thm-parab-Besov}
The above theorem relies on the fact that  problem \eqref{parab-1a} has a generalized solution $u\in W^{m,{\gamma_m+2}}_2(K_T)={W^{\gamma_m+1}_2}((0,T), W^m_2(K))\cap {W^{\gamma_m+2}_2}((0,T),L_2(K))\hookrightarrow L_2((0,T), W^m_2(K))$, cf. Proposition \ref{GenSol_Sobolev}.  We strongly believe that (in good  agreement with the elliptic case) this result {could be improved by studying the regularity of \eqref{parab-1a} in fractional Sobolev spaces $W^s_2(K)$, $s\geq 0$, cf. \eqref{slobodeckij} and the explanations given. }
In this case (assuming that the generalized solution of \eqref{parab-1a} satisfies $u\in L_2((0,T),W^s_2(K))$ for some $s>0$) under the assumptions of Theorem \ref{thm-parab-Besov}, using Theorem  \ref{thm-weighted-sob-reg} and Theorem \ref{thm-hansen-gen} {(with $k=0$)}, we {would} obtain 
\[
{\varphi} u\in  L_2((0,T),\mathcal{K}^{\eta}_{2,a'}(K))\cap L_2((0,T),W_2^s(K))\hookrightarrow L_2((0,T),B^{\eta}_{\tau,\infty}(K)),
\]
where $a'=a+ 2m(\gamma_m+1)\geq a+2m$ and again $\frac 12<\frac{1}{\tau}<\frac{\eta}{3}+\frac 12$ but the restriction on $\gamma$ now reads as 
\[
\eta<3\min(s, a').  
\]
%for some $a'=a+2m(\gamma_m+1)\geq a+2m$ \text{and} $\textcolor{blue}{a\leq m}$. 
For general Lipschitz domains {and $m=1$} we expect {that the solution to \eqref{parab-1a} is contained in $W^s_2(K)$ for} all $s<\frac 32$ (as was shown in the elliptic case for the Poisson equation in \cite{JK95}), which then leads to $\eta<\frac 92$. For convex domains it probably even  holds that $s=2$ (for the heat equation this was already proven in \cite{Wo07}), which in turn would yield  $\eta<6$. 
{However, to establish these kind of regularity results is clearly beyond the scope of this paper and will be the topic of further studies.}}
}

{
\expl{\label{ex_heat_eq} As a parabolic model case for \eqref{parab-1a} we consider the heat equation 
\begin{eqnarray*}
\partial_t u-\Delta u&=&f  \ \text{ on }{K_{T}},\\
u\big|_{t=0}&=&0  \ \text{ on } {K}, 
\end{eqnarray*}
{where $f$ satisfies the assumptions of Theorem \ref{thm-parab-Besov} for some $\gamma\geq 2m$ that we can choose large enough. }
For the parameters $m=1$  and $s=m=1$ Theorem \ref{thm-parab-Besov}  yields the  maximal Besov regularity
\begin{equation}\label{max_B_heat}
{\eta<\min(3,\gamma)=3,}  
\end{equation}
for the solution ${\varphi} u$ of \eqref{parab-1a}.   If we additionally assume that our polyhedral cone $K_0$ is convex, i.e., 
$\theta_k\in (0,\pi)$ for all $k=1, \ldots, n$, we can do even better. In this case now $s=2$, cf. \cite[Thm.~6.2]{Wo07}, thus, according to what was said in Remark \ref{gen-thm-parab-Besov}, the upper bound for the maximal Besov regularity is 
\[
{\eta<\min(3s,\gamma)= 3\cdot 2=6, }
\]
for $0\leq a<\min\left(1,\frac{\pi}{\theta_k}-1\right)$, cf. Remark \ref{discuss_aa}. 
%Here for the corresponding elliptic problem (Poisson equation) within the scale of fractional Sobolev spaces we actually have regularity $u\in H^{\frac 32}(K_{0})$,  cf. \cite{JK95}. Since as upper bound for the maximal Besov regularity  we have $m<\frac{d}{d-1}s$ on Lipschitz domains \textcolor{red}{reference}, we obtain in this case  $m<\frac 32\cdot \frac 32=\frac 94$.  \\
%Note that the results from \eqref{max_B_heat} are probably not optimal since one would expect the same Sobolev regularity as in the elliptic case -- i.e., $s=\frac 32$ as mentioned above -- which in turn would lead  to maximal  Besov regularity $m<\frac{9}{2}$.\\  
}
}

{
Also Theorem \ref{thm-weighted-sob-reg-2} gives rise to the following parabolic Besov regularity results. 
}

{
\begin{theorem}[Parabolic Besov regularity II]\label{thm-parab-Besov-2}
Let $K\subset \real^3$ be a polyhedral cone. Let $\gamma\in \nat $ with  ${\gamma\geq 2m}$. Moreover, let  $a\in \real$ with   ${a\in [-m,m]}$.  Assume that  the right hand side $f$ of \eqref{parab-1a} satisfies 
\begin{itemize}
\item[(i)] $f\in \bigcap_{l=0}^{\infty}W^l_2((0,T),L_2(K)\cap \calk^{\gamma-2m}_{2,a}(K))$. 
$\partial_{t^{\gamma_m+1}} f\in L_2(K_T)$. 
\item[(ii)] $\partial_{t^l} f(x,0)=0$, \quad  $l\in \nat_0$. 
\end{itemize}
Furthermore, let  Assumption \ref{assumptions}  hold for weight parameters $b=a$ and  $b'=-m$.  
Let ${\varphi}$ denote the cutoff function  from \eqref{cutoff}. Then 
%there exists some $0<\tau_0\leq 2$ such that 
for the generalized solution $\bigcap_{l=0}^{\infty}u\in {\mathring{W}}_2^{m,{l+1}}(K_T)$ of problem \eqref{parab-1a}, we have 
\begin{equation}\label{parab-Besov2}
{\varphi} u\in L_{2}((0,T),B^{{\eta}}_{\tau,\infty}(K)) \quad \text{for all}\quad {0<{\eta<\min(\gamma,3m)}, } \quad  \frac 12 <\frac{1}{\tau}<\frac{{\eta}}{3}+\frac 12.     
%\text{for all} \quad  \tau_{\ast}<\tau<\tau_0 \quad \text{for all} \frac{1}{\tau^{\ast}}=\frac{\gamma}{d}+\frac 12, 
\end{equation}
In particular, for any ${\eta}, \tau$ satisfying \eqref{parab-Besov2}, we have the a priori estimate 
\begin{align*}
\|{\varphi} u|& L_{2}((0,T),B^{{\eta}}_{\tau,\infty}(K))\|\\
&\lesssim  \sum_{k=0}^{\gamma-2m}\|\partial_{t^k} f|{L_2((0,T),{\mathcal{K}^{\gamma-2m}_{2,a}(K))}}\|+\sum_{k=0}^{(\gamma-2m)+1}\|\partial_{t^k} f|{L_2(K_T)}\|. 
\end{align*}
\end{theorem}
}

{
\begin{proof}
According to Theorem  \ref{thm-weighted-sob-reg-2} by our assumptions we know  
${\varphi} u\in L_2((0,T), {\mathcal{K}^{\gamma}_{2,a+2m}(K)})$. Together with  Theorem  \ref{thm-hansen-gen} {(choosing $k=0$)} we obtain  
\begin{align*}
{\varphi} u\in &L_2((0,T),{\mathcal{K}^{\gamma}_{2,a+2m}(K)})\cap {\mathring{W}}_2^{m,{1}}(K_T)\\
&\hookrightarrow  L_2((0,T),\mathcal{K}^{\gamma}_{2,a+2m}(K))\cap L_2((0,T),W_2^m(K))\\
&\hookrightarrow   L_2((0,T),\mathcal{K}^{\gamma}_{2,a+2m}(K))\cap L_2((0,T),B^m_{2,\infty}(K))\\
&\hookrightarrow   L_2((0,T),\mathcal{K}^{\eta}_{2,a+2m}(K)\cap B^m_{2,\infty}(K))\\
%&{\hookrightarrow   L_2((0,T),\mathcal{K}^{\gamma}_{2,a+2m}(K)\cap L_2((0,T),B^m_{2,\infty}(K))}\\
&\hookrightarrow    L_2((0,T),B^{{\eta}}_{\tau,\infty}(K)),   
\end{align*}
 where {$\eta \leq \gamma$ in the second but last line. Moreover, } the condition on $a$ from Theorem \ref{thm-hansen-gen} yields  
$$
m=\min(m,a+2m)>\frac{\delta}{d}\eta=\frac{{\eta}}{3}.   
$$ 
{Therefore, the upper bound for ${\eta}$ is   $${\eta<\min(3m,\gamma)}. $$ }
Concerning the restriction on $\tau$, Theorem  \ref{thm-hansen-gen} with $\tau_0=2$ gives 
\[
\frac 12<\frac{1}{\tau}<\frac{1}{\tau^{\ast}}=\frac{{\eta}}{3}+\frac 12.
\]
\end{proof}
}

\subsection{Besov regularity of  nonlinear parabolic PDEs}\label{Subsect-4.2}

%\textcolor{red}{
%\begin{itemize}
%\item Spuren mit einbauen (wie im nichtlinearen Besov-Paper (mit zus\"atzlichem Raum schneiden... $H^1_0$?))
%\end{itemize}
%}

\subsubsection{The fundamental problem}

We modify \eqref{parab-1a} and now consider for some $\varepsilon>0$, $M\in \nat$,  the following nonlinear parabolic problem 
\begin{equation} \label{parab-nonlin-1}
\left\{\begin{array}{rl}
\frac{\partial }{\partial t}u+(-1)^mL(x,t;D_x)u +\varepsilon u^{M}\ =\ f \, &  \text{ in } K_T, \\
\frac{\partial^{k-1}u}{\partial \nu^{k-1}}\Big|_{\Gamma_{j,T}}\ =\ 0, & \   k=1,\ldots, m, \ j=1,\ldots, n,\\ 
u\big|_{t=0}\ =\ 0 \, & \text{ in } K. 
\end{array} \right\}
\end{equation}

We are interested in the Besov regularity of solutions $u$ of problem \eqref{parab-nonlin-1}. Our strategy to {reach}  this goal is as follows. 
{We show that the regularity estimates in {Kondratiev and Sobolev spaces  as stated in Theorem \ref{thm-weighted-sob-reg} and Proposition \ref{GenSol_Sobolev}} carry over to \eqref{parab-nonlin-1}, provided that $\varepsilon$ is sufficiently small. Then, we proceed as in the proof of Theorem \ref{parab-Besov}, i.e., we once again use embedding results of Kondratiev  spaces into Besov spaces. To establish Kondratiev regularity we interpret \eqref{parab-nonlin-1} as a fixed point problem in the following way. 
}
 Let $D$ and $S$ be Banach-spaces ({$D$ and $S$ will be specified in the theorem below}) and let $\tilde{L}^{-1}:D\rightarrow S$ be the linear operator defined  via
\begin{equation} \label{tildeL}
\tilde{L}u:=\frac{\partial}{\partial t}u+(-1)^mLu. 
 \end{equation}
 {Equation  \eqref{parab-nonlin-1}  is equivalent to 
\[
\tilde{L}u=f-\varepsilon u^{M}=:Nu,
\] 
where $N:S\rightarrow D$ is a nonlinear operator. If we can show that $N$ maps $S$ into $D$, then a solution to \eqref{parab-nonlin-1} is a fixed point of the problem }
%Concerning the existence of a solution of \eqref{parab-nonlin-1} we are looking for a fixed point of the problem 
\[
(\tilde{L}^{-1}\circ N)u=u.
\]
Our aim is to apply Banach's fixed point theorem, which will also guarantee uniqueness of the solution if we can show that  $T:=(\tilde{L}^{-1}\circ N): S_0\rightarrow S_0$ is a contraction mapping, i.e., 
\[
\|T(x)-T(y)|S\|\leq q\|x-y|S\| \quad \text{for all}\quad x,y\in S_0, \quad q\in [0,1),
\]
where the corresponding metric space $S_0\subset S$ is a small {closed} ball with center $\tilde{L}^{-1}f$  (the solution of the corresponding linear problem) and suitably chosen radius $R>0$. \\

\subsubsection{Nonlinear regularity results}

%\textcolor{red}{Insert boundary conditions in data space.}

\begin{theorem}[Nonlinear Kondratiev regularity]\label{nonlin-B-reg1}
%\textcolor{red}{We require $a>\frac 12$ in order to apply multiplier theorem!}
Let $\tilde{L}$ and $N$ be as described above. Assume the assumptions of Theorem \ref{thm-weighted-sob-reg} are satisfied and, additionally, we have {$\gamma_m\geq 1$}, $m\geq 2$,  and {$a\geq -\frac 12$}.  
%$$a+2m(\gamma_m-k)>\frac 32, \qquad k=0,\ldots, \gamma_m.$$ 
Let 
\[
D_1:=\bigcap_{k=0}^{\gamma_m}W_2^{k}((0,T),\mathcal{K}^{2m(\gamma_m-k)}_{2,a+2m(\gamma_m-k)}(K)), \quad 
D_2:=W^{\gamma_m+1}_2((0,T),L_2(K))
\]
and consider the data space 
\begin{align*}
D &:=\{f\in D_1 
\cap D_2: \  
\partial_{t^k} f(\cdot,0)=0, \quad k=0,\ldots, \gamma_m\}.   
\end{align*}
Moreover, let 
\begin{align*}
S_1&:= \bigcap_{k=0}^{\gamma_m+1}W_2^{k}((0,T),\mathcal{K}^{2m(\gamma_m-(k-1))}_{2,a+2m(\gamma_m-(k-1))}(K)), \quad \\ 
S_2&:={\mathring{W}^{m,\gamma_m+2}_2(K_T)}, %\hookrightarrowW^{\gamma_m+1}_2((0,T), {\mathring{W}^m_2}(K))\cap W^{\gamma_{m}+2}_2((0,T), L_2(K)), 
\end{align*}
and  consider the solution space 
$S:=S_1\cap S_2$. Suppose that $f\in D$
and put  $\eta:=\|f|D\|$ and $r_0>1$. Moreover, we choose $\varepsilon >0$ so small that 
\[
{
\eta^{2(M-1)} \|\tilde{L}^{-1}\|^{2M-1}\leq \frac{1}{{c}\varepsilon M}(r_0-1)\left(\frac{1}{r_0}\right)^{2M-1}, \qquad \text{if}\quad  r_0\|\tilde{L}^{-1}\|\eta>1,
}
\]
and 
\[\|\tilde{L}^{-1}\|<\frac{r_0-1}{r_0}\left(\frac{1}{{c}\varepsilon M}\right), \qquad \text{if}\quad  r_0\|\tilde{L}^{-1}\|\eta<1,\]
{where $c>0$ denotes the  constant in  \eqref{est-ab} resulting from our estimates below. }
Then there exists a unique  solution $u\in S_0\subset S$ of problem \eqref{parab-nonlin-1}, where $S_0$ denotes a small ball  around $\tilde{L}^{-1}f$ (the solution of the corresponding linear problem) with radius $R=(r_0-1)\eta \|\tilde{L}^{-1}\|$. 
\end{theorem}

\begin{proof} 
Let $u$ be the solution of the linear problem $\tilde{L}u=f$.  Theorem \ref{thm-weighted-sob-reg} and Proposition \ref{GenSol_Sobolev} show  that 
$$\tilde{L}^{-1}: D\rightarrow S   
$$
is a bounded operator. 
{We need to show that 
\begin{equation}\label{prop-uM}
u^M\in D %\qquad \text{and}\qquad (\partial_{t^k}u^M)(\cdot, 0)=0 \quad  \text{for}\quad  k=0,\ldots, \gamma_m, 
\end{equation} 
in order to establish the desired mapping properties of the nonlinear part $N$,  
i.e.,
\[
Nu=f-\varepsilon u^M\in D.  
\]
The fact that $u^M\in D_1\cap D_2$ follows from the estimate \eqref{est-ab}: {in particular, taking $v=0$ in \eqref{est-ab} we get an estimate from above for $\|u^M|D\|$. The upper bound depends on $\|u|S\|$ and several constants which depend on $u$ but are finite whenever we have $u\in S$, see also \eqref{est-4} and \eqref{est-4a}. The dependence on $R$ in \eqref{est-ab} comes from the fact that we choose $u\in B_R(\tilde{L}^{-1}f)$ in $S$ there. {However, the same argument can also be applied to an arbitrary $u\in S$; this would result in a different constant $\tilde{c}$.}}   Since $u\in S\hookrightarrow W^{\gamma_m+2}_2((0,T), L_2(K))\hookrightarrow C^{\gamma_m+1}((0,T), L_2(K))$ we see that the trace operator $\mathrm{Tr}\left(\partial_{t^k}u\right):=\left(\partial_{t^k}u\right)(\cdot,0)$ is well defined for $k=0,\ldots, \gamma_m+1$ (in particular, since $u\in S\hookrightarrow W^{\gamma_m+1}_2((0,T),{\mathring{W}^m_2}(K))\hookrightarrow C^{\gamma_m}((0,T),{\mathring{W}^m_2}(K))$  the values of the trace operator $\mathrm{Tr}\left(\partial_{t^k}u\right)$, $k=0,\ldots, \gamma_m$, belong to  $W^m_2$).
{We show  that $\mathrm{Tr}(\partial_{t^k}u)=0$ for $k=0,\ldots, \gamma_m+1$. For this we use the assumptions on $f$, i.e., $\partial_{t^k}f(\cdot,0)=0$ for $k=0,\ldots, \gamma_m$,  the initial assumption $u(\cdot,0)=0$ in \eqref{parab-nonlin-1}, 
and the fact that
\begin{equation}\label{interchange_trace}
\mathrm{Tr}(Lu)(x,t) = 0. 
\end{equation}
Let us briefly sketch the proof of \eqref{interchange_trace}.  We show \eqref{interchange_trace} first for 
compactly supported functions $u \in C^{\infty}(K_T)$.  Then, the result follows
by density arguments.  
For these functions, we get
\[
\mathrm{Tr} (Lu)(x,t)=\lim_{t\rightarrow 0} Lu(x,t)=L(x,0;D_x)\lim_{t\rightarrow 0}u(x,t)=0,
\]
where the second step follows 
 %
%  that fact that 
 %\begin{equation}\label{interchange_trace}
% \mathrm{Tr} (Lu)(x,t)=\lim_{t\rightarrow 0} Lu(x,t)=L(x,0;D_x)\lim_{t\rightarrow 0}u(x,t)=0,
% \end{equation}
% where  the second step  follows 
by our smoothness assumptions on the coefficients of $L$ and the fact that $\lim_{t\rightarrow 0}D^{\alpha}_x u(x,t)=D^{\alpha}_x \left(\lim_{t\rightarrow 0}u(x,t)\right)$ (this is clear for smooth functions $u\in C^{\infty}(K_T)$, then using density we get the same  for $u\in S$). With this we see that } 
 % we see that }
 \[
(\partial_tu)(\cdot,0)+(-1)^mLu(\cdot,0)=f(\cdot, 0), \qquad \text{i.e.,}\qquad  (\partial_tu)(\cdot,0)=0. 
 \] 
 Differentiation yields 
 \begin{align*}
(\partial_{t^2}u)(\cdot,0)+(-1)^m((\partial_tL)u(\cdot,0)+L(\partial_tu)(\cdot,0))&=\partial_tf(\cdot, 0). % \\ 
%\qquad \text{i.e.,}\qquad  (\partial_{t^2}u)(\cdot,0)&=0. 
 \end{align*}
 i.e., $(\partial_{t^2}u)(\cdot,0)=0$ {using a similar argumentation as in \eqref{interchange_trace}}.  
 By induction we deduce  that $(\partial_{t^k}u)(\cdot,0)=0$ for all $k=0,\ldots, \gamma_m+1$ (in particular, $\|\left(\partial_{t^k}u\right)(\cdot, 0)|W^m_2(K)\|=0$ for $k=0,\ldots, \gamma_m$). 
 Moreover, since by Theorem \ref{thm-sob-emb} (generalized Sobolev embedding)  
 \begin{align*}
 u^M \in D_1\cap D_2 &\hookrightarrow W^{\gamma_m+1}_2((0,T), L_2(K)) 
 \hookrightarrow C^{\gamma_m}((0,T), L_2(K)), 
%  u \in S &\hookrightarrow W^{k+1}_2((0,T), \calk^{2m(\gamma_m-k)}_{2,a+2m(\gamma_m-k)}(K)\cap W^m_2(K)) \\
% & \hookrightarrow C^k((0,T), \calk^{2m(\gamma_m-k)}_{2,a+2m(\gamma_m-k)}(K)\cap W^m_2(K))=:C^k((0,T),X), 
 \end{align*}
 we see that the trace operator $\mathrm{Tr}\left(\partial_{t^k}u^M\right):=\left(\partial_{t^k}u^M\right)(\cdot,0)$ is well defined for $k=0,\ldots, \gamma_m$. By \eqref{est-3a} below the term $\|\left(\partial_{t^k}u^M\right)(\cdot, 0)|L_2(K)\|$ is estimated from above by {powers of} $\|\left(\partial_{t^l}u\right)(\cdot, 0)|W^m_2(K)\|$, $l=0,\ldots, k$. % (and $\|\left(\partial_{t^{\gamma_m+1}}u\right)(\cdot, 0)|L_2(K)\|^{k_{\gamma_m+1}}$). 
 Since all these terms are equal to zero, we obtain $\left(\partial_{t^k}u^M\right)(\cdot, 0)=0$ in the $L_2$-norm for $k=0,\ldots, \gamma_m$. 
 }
%{which establishes the desired mapping properties of the nonlinear part $N$. } 
Hence, {using \eqref{prop-uM}} we can apply Theorem  \ref{thm-weighted-sob-reg}  now with right hand side $Nu$. 
Since 
\[
 (\tilde{L}^{-1}\circ N)(v)-(\tilde{L}^{-1}\circ N)(u)= \tilde{L}^{-1}(f-\varepsilon v^{M})-\tilde{L}^{-1}(f-\varepsilon u^{M}) =\varepsilon  \tilde{L}^{-1}(u^{M}-v^{M})
\]
one sees that $\tilde{L}^{-1}\circ N$ is a contraction if, and only, if 
\begin{equation}\label{est-0}
\varepsilon \| \tilde{L}^{-1}(u^{M}-v^{M})|S\| \leq q\|u-v|S\|\quad \text{for some }q<1, 
\end{equation}
where $u,v\in S_0$ (meaning $u,v\in   B_R(\tilde{L}^{-1}f)$ in $S$). %\textcolor{red}{Is $\tilde{L}^{-1}$ defined on $S_0$ or on $S_0\cap \{\text{initial data}\}$?}
We analyse the {resulting  condition} with the help of the formula  $ u^M-v^M=(u-v)\sum_{j=0}^{M-1} u^jv^{M-1-j}$. This together with Theorem \ref{thm-weighted-sob-reg} gives 
\begin{align}
\|&\tilde{L}^{-1}(u^M-v^M)|S\|\notag\\
&\leq \|\tilde{L}^{-1}\| \|u^M-v^M|D\|\notag\\
&= \|\tilde{L}^{-1}\| \left\|u^M-v^M|
%\bigcap_{k=0}^{\gamma_m}W_2^{k}((0,T),\mathcal{K}^{2m(\gamma_m-k)}_{2,a+2m(\gamma_m-k)}(K))
 %\cap W^{\gamma_m+1}_2((0,T),L_2(K))
D_1\cap D_2 \right\|\notag\\
&= \|\tilde{L}^{-1}\| \left(\|u^M-v^M|D_1\|+\|u^M-v^M|D_2\|\right) \notag \\
%&= \|\tilde{L}^{-1}\| \left\|(u-v)\sum_{j=0}^{M-1} u^jv^{M-1-j}|
%D_1\cap D_2
%\right\|\notag\\
&=  \|\tilde{L}^{-1}\| \left(\left\|(u-v)\sum_{j=0}^{M-1} u^jv^{M-1-j}|
D_1
\right\|+ \left\|(u-v)\sum_{j=0}^{M-1} u^jv^{M-1-j}|
D_2
\right\| \right)\notag\\
&=  \|\tilde{L}^{-1}\| \Bigg( \sum_{k=0}^{\gamma_m}\left\|\partial_{t^k}\left[(u-v)\sum_{j=0}^{M-1} u^jv^{M-1-j}\right]|L_2((0,T),\mathcal{K}^{2m(\gamma_m-k)}_{2,a+2m(\gamma_m-k)}(K))\right\| \notag \\
& \qquad \qquad \qquad  + \sum_{k=0}^{\gamma_m+1}
\left\|\partial_{t^{k}}\left[(u-v)\sum_{j=0}^{M-1} u^jv^{M-1-j}\right]|
L_2(K_T)\right\| \Bigg) 
\notag\\ \label{est-1}
\end{align}
Concerning the derivatives, we use Leibniz's formula twice and  we see that 
\begin{align}
\partial_{t^k}&(u^M-v^M)\notag\\ 
&=\partial_{t^k}\left[(u-v)\sum_{j=0}^{M-1} u^jv^{M-1-j}\right]\notag\\
&= \sum_{l=0}^k{k \choose l}\partial_{t^l}(u-v)\cdot \partial_{t^{k-l}}\left(\sum_{j=0}^{M-1}u^j v^{M-1-j}\right)\notag\\
&= \sum_{l=0}^k{k \choose l}\partial_{t^l}(u-v)\cdot 
\left[\left(\sum_{j=0}^{M-1} \sum_{r=0}^{k-l} {{k-l}\choose r} \partial_{t^r}u^j \cdot \partial_{t^{k-l-r}}v^{M-1-j}\right)\right]. \notag\\\label{est-2}
\end{align}
In order to estimate the terms $\partial_{t^r}u^j$ and $\partial_{t^{k-l-r}}v^{M-1-j}$ we apply Fa\`{a} di Bruno's formula 
\begin{equation}\label{FaaDiBruno}
\partial_{t^r}(f\circ g)=\sum\frac{r!}{k_1!\ldots k_r!}\left(\partial_{t^{k_1+\ldots + k_r}}f\circ g\right)\prod_{{i}=1}^{r}\left(\frac{\partial_{t^{{i}}}g}{{i}!}\right)^{k_{{i}}},
\end{equation}
where  the sum runs over all $r$-tuples of nonnegative integers $(k_1,\ldots, k_r)$ satisfying 
\begin{equation}\label{cond-kr}
1\cdot k_1+2\cdot k_2+\ldots +r\cdot k_r=r.
\end{equation}
{In particular, from \eqref{cond-kr}  we see that $k_{r}\leq 1$, where  $r=1,\ldots, k$. Therefore, the highest derivative $\partial_{t^r}u$ appears at most once.  }
We {apply the formula to} $g=u$ and $f(x)=x^j$ and
  make use of the embeddings  \eqref{kondratiev-emb} and the pointwise multiplier results from Corollary \ref{thm-pointwise-mult-2} {(i) for $k\leq \gamma_m-1$}.   
This yields 
\begin{align}
\Big\|&\partial_{t^r}u^j  \left. | \mathcal{K}^{2m(\gamma_m-k)}_{2,a+2m(\gamma_m-k)}(K)\right\| \notag\\
&\leq  c_{r,j}\left\|\sum_{k_1+\ldots +k_r\leq j, \atop 1\cdot k_1+2\cdot k_2+\ldots +r\cdot k_r=r} u^{j-(k_1+\ldots +k_r)}\prod_{{i}=1}^r \left|\partial_{t^{{i}}}u\right|^{k_{{i}}}| \mathcal{K}^{2m(\gamma_m-k)}_{2,a+2m(\gamma_m-k)}(K)\right\| \notag\\
&\lesssim \sum_{k_1+\ldots +k_r\leq j, \atop 1\cdot k_1+2\cdot k_2+\ldots +r\cdot k_r=r} \left\| u| \mathcal{K}^{2m(\gamma_m-k)}_{2,a+2m(\gamma_m-k)}(K)\right\|^{j-(k_1+\ldots +k_r)} \notag\\
& \qquad \qquad   
\prod_{{i}=1}^{r} \left\| \partial_{t^{{i}}}u| \mathcal{K}^{2m(\gamma_m-k)}_{2,a+2m(\gamma_m-k)}(K)\right\|^{k_{{i}}}.  \notag \\ \label{est-3}
\end{align}
%\textcolor{red}{Note that $\left\| \partial_{t^r}u| \mathcal{K}^{2m(\gamma_m-k)}_{2,a+2m(\gamma_m-k)}(K)\right\|\leq \left\| \partial_{t^r}u| \mathcal{K}^{2m(\gamma_m-(k-1))}_{2,a+2m(\gamma_m-(k-1))}(K)\right\|$. }
{For $k=\gamma_m$ we use Corollary \ref{thm-pointwise-mult-2}(ii)  and obtain similar as above } 
{\begin{align}
\Big\|&\partial_{t^r}u^j  \left. | \mathcal{K}^{0}_{2,a}(K)\right\| \notag\\
&\leq  c_{r,j}\left\|\sum_{k_1+\ldots +k_r\leq j, \atop 1\cdot k_1+2\cdot k_2+\ldots +r\cdot k_r=r} u^{j-(k_1+\ldots +k_r)}\prod_{{i}=1}^r \left|\partial_{t^{{i}}}u\right|^{k_{{i}}}| \mathcal{K}^{0}_{2,a}(K)\right\| \notag\\
&\lesssim \sum_{k_1+\ldots +k_r\leq j, \atop 1\cdot k_1+2\cdot k_2+\ldots +r\cdot k_r=r} \left\| u| \mathcal{K}^{2}_{2,a+2}(K)\right\|^{j-(k_1+\ldots +k_r)} \notag\\
& \qquad \qquad   
\left\| \partial_{t^r}u| \mathcal{K}^{0}_{2,a}(K)\right\|^{k_r}\prod_{{i}=1}^{r-1} \left\| \partial_{t^{{i}}}u| \mathcal{K}^{2}_{2,a+2}(K)\right\|^{k_{{i}}}\notag\\
&\lesssim \sum_{k_1+\ldots +k_r\leq j, \atop 1\cdot k_1+2\cdot k_2+\ldots +r\cdot k_r=r} \left\| u| \mathcal{K}^{2m\gamma_m}_{2,a+2m\gamma_m}(K)\right\|^{j-(k_1+\ldots +k_r)} \notag\\
& \qquad \qquad   
\left\| \partial_{t^r}u| \mathcal{K}^{2m(\gamma_m-r)}_{2,a+2m(\gamma_m-r)}(K)\right\|^{k_r}\prod_{{i}=1}^{r-1} \left\| \partial_{t^{{i}}}u| \mathcal{K}^{2m(\gamma_m-{i})}_{2,a+2m(\gamma_m-{i})}(K)\right\|^{k_{{i}}}.  \notag \\ \label{est-33a}
\end{align}}
{Note that we require $\gamma_m\geq 1$ in the last step. }
{We proceed similarly}  for $\partial_{t^{k-l-r}}v^{M-1-j}$. 
Now \eqref{est-2} together with \eqref{est-3}  {and \eqref{est-33a}} inserted in \eqref{est-1} together with  Corollary \ref{thm-pointwise-mult-2} give \\

$\|\tilde{L}^{-1}\|\|u^M-v^M|D_1\|$
\begin{align}
&\lesssim  \|\tilde{L}^{-1}\|\sum_{k=0}^{\gamma_m}\left(\int_0^T\left\|\partial_{t^k}\left[(u-v)\sum_{j=0}^{M-1} u^jv^{M-1-j}\right]|\mathcal{K}^{2m(\gamma_m-k)}_{2,a+2m(\gamma_m-k)}(K)\right\|^2\ud t\right)^{1/2}\notag\\
& {
\lesssim \|\tilde{L}^{-1}\|\sum_{k=0}^{\gamma_m}\sum_{l=0}^k\sum_{j=0}^{M-1}\sum_{r=0}^{k-l}\Bigg(\int_0^T \left\|\partial_{t^l}(u-v)
|\mathcal{K}^{2m(\gamma_m-k)}_{2,a+2m(\gamma_m-k)}(K)\right\|^2 }\notag \\
& {
\qquad \qquad 
\left\|\partial_{t^r}u^j |\mathcal{K}^{2m(\gamma_m-k)}_{2,a+2m(\gamma_m-k)}(K)\right\|^2
\left\|\partial_{t^{k-l-r}}v^{M-1-j}|\mathcal{K}^{2m(\gamma_m-k)}_{2,a+2m(\gamma_m-k)}(K)\right\|^2
\ud t\Bigg)^{1/2}
}\label{k=gamma_m} \\
& {
\lesssim \|\tilde{L}^{-1}\|\sum_{k=0}^{\gamma_m}\sum_{l=0}^k\sum_{j=0}^{M-1}\sum_{r=0}^{k-l}\Bigg(\int_0^T \left\|\partial_{t^l}(u-v)
|\mathcal{K}^{2m(\gamma_m-k)}_{2,a+2m(\gamma_m-k)}(K)\right\|^2 }\notag \\
& {
%\qquad \qquad 
\sum_{\kappa_1+\ldots+\kappa_r\leq j, \atop \kappa_1+2\kappa_2+\ldots+r\kappa_r=r}
\left\|u |\mathcal{K}^{2m(\gamma_m-k)}_{2,a+2m(\gamma_m-k)}(K)\right\|^{2(j-(\kappa_1+\ldots+\kappa_r))}
\prod_{i=0}^r \left\| \partial_{t^{{i}}}u| \mathcal{K}^{2m(\gamma_m-{i})}_{2,a+2m(\gamma_m-{i})}(K)\right\|^{2\kappa_{{i}}}
}\notag \\
& {
%\qquad 
\sum_{\kappa_1+\ldots+\kappa_{k-l-r}\leq M-1-j, \atop \kappa_1+2\kappa_2+\ldots+(k-l-r)\kappa_{k-l-r}=k-l-r}
\left\|v |\mathcal{K}^{2m(\gamma_m-k)}_{2,a+2m(\gamma_m-k)}(K)\right\|^{2(M-1-j-(\kappa_1+\ldots+\kappa_{k-l-r}))} }\notag \\
&{\qquad \qquad 
\prod_{i=0}^{k-l-r} \left\| \partial_{t^{{i}}}v| \mathcal{K}^{2m(\gamma_m-{i})}_{2,a+2m(\gamma_m-{i})}(K)\right\|^{2\kappa_{{i}}}
\ud t\Bigg)^{1/2}
}\notag \\
& {
\lesssim \|\tilde{L}^{-1}\|\sum_{k=0}^{\gamma_m}M\Bigg(\int_0^T \left\|\partial_{t^k}(u-v)
|\mathcal{K}^{2m(\gamma_m-k)}_{2,a+2m(\gamma_m-k)}(K)\right\|^2 }\notag \\
& {
\qquad 
\sum_{\kappa_1'+\ldots+\kappa_k'\leq \min\{M-1,k\}, \atop \kappa_k'\leq 1}
\max_{w\in \{u,v\}}\left\|w |\mathcal{K}^{2m(\gamma_m-k)}_{2,a+2m(\gamma_m-k)}(K)\right\|^{2(M-1-(\kappa_1'+\ldots+\kappa_k'))}}\notag \\
& {\qquad 
\prod_{i=0}^k  \max\left\{\left\| \partial_{t^{{i}}}u| \mathcal{K}^{2m(\gamma_m-{i})}_{2,a+2m(\gamma_m-{i})}(K)\right\|, \left\| \partial_{t^{{i}}}v| \mathcal{K}^{2m(\gamma_m-{i})}_{2,a+2m(\gamma_m-{i})}(K)\right\|, 1\right\}^{4\kappa_i'} \ud t\Bigg)^{1/2}
}\label{kappa} \\
%&\lesssim \|\tilde{L}^{-1}\|\sum_{k=0}^{\gamma_m}\Bigg(\int_0^T\left\|\partial_{t^k}(u-v)| \mathcal{K}^{2m(\gamma_m-k)}_{2,a%+2m(\gamma_m-k)}(K)\right\|^2 \cdot %\right. 
%\notag\\ 
%& \qquad 
%\left.  
%%\left(
%\sum_{j=0}^{M-1} \sum_{k_1+\ldots + k_{k}\leq j, \atop {k_1+2k_2\ldots+k k_{k}\leq k \atop dfdf}}\max_{w\in \{u,v\}}\left\| w| \mathcal{K}%^{2m(\gamma_m-k)}_{2,a+2m(\gamma_m-k)}(K)\right\|^{2(j-(k_1+\ldots +k_{k}))} \cdot %\right)\cdot 
%%\notag \right.\\
%&& \quad \cdot \left.\left(
%%\left\| v| \mathcal{K}^{\gamma}_{2,a'}(K)\right\|^{2(j-(k_1+\ldots +k_{l+1}))} \prod_{m=1}^{l} \left\| \frac{\partial}{\partial t^m}v| %%\mathcal{K}^{\gamma}_{2,a'}(K)\right\|^{2k_m} \left\| \frac{\partial}{\partial t^{l+1}}v| \mathcal{K}^{\gamma-2m}_{2,a}(K)\right\|^{2k_l}
%%\right)
%\right.\notag\\
%& \qquad \qquad \qquad \left. \prod_{{i}=1}^{k} \left\| \partial_{t^{{i}}}w| \mathcal{K}^{2m(\gamma_m-%k)}_{2,a+2m(\gamma_m-k)}(K)\right\|^{2k_{{i}}} \ud t\right)^{1/2}\notag\\
%\sum_{()} \| u|\mathcal{K}^{\gamma-2m+2}_{2,a+2}(K)\|^{2j} \cdot \|v|\mathcal{K}^{\gamma-2m+2}_{2,a+2}(K)\|^{2(n-1-j)}
%\ud t\right)^{1/2}\\
%&\leq & \|\tilde{L}\|^{-1}\left(\int_0^T\|(u-v)| \mathcal{K}^{\gamma}_{2,a'}(K)\|^2
%\sum_{j=0}^{n-1} \| u|\mathcal{K}^{\gamma}_{2,a'}(K)\|^{2j} \cdot \|v|\mathcal{K}^{\gamma}_{2,a'}(K)\|^{2(n-1-j)}
%\ud t\right)^{1/2}\\
&\lesssim M \|\tilde{L}^{-1}\| \cdot 
\left\|u-v| \bigcap_{k=0}^{\gamma_m+1}W_2^{k}((0,T),\mathcal{K}^{2m(\gamma_m-(k-1))}_{2,a+2m(\gamma_m-(k-1))}(K))\right\|\cdot \notag\\
& \qquad \max_{w\in \{u,v\}}\max_{l=0,\ldots, \gamma_m} \max \Big(\left\| \partial_{t^l}w |L_{\infty}((0,T),\mathcal{K}^{2m(\gamma_m-l)}_{2,a+2m(\gamma_m-l)}(K))\right\|,\  
1\Big)^{{2(M-1)}}.\notag\\ \label{est-4}
\end{align}
{
We give some explanations concerning the estimate above. In \eqref{k=gamma_m} the term with $k=\gamma_m$ requires some special care since we have to apply Corollary \ref{thm-pointwise-mult-2} (ii). In this case we calculate 
\begin{align*}
\Bigg\| & \left.\partial_{\gamma_m}\left[(u-v)\left(\sum_{j=0}^{M-1}u^jv^{M-1-j}\right)\right]|\calk^{0}_{2,a}(K)\right\|\notag \\ 
& \lesssim \left\|\partial_{\gamma_m}(u-v)|\calk^{0}_{2,a}(K)\right\|
\sum_{j=0}^{M-1}\left\|u^jv^{M-1-j}|\calk^{2}_{2,a+2}(K)\right\| \notag\\
& \qquad + \left\|u-v|\calk^{2}_{2,a+2}(K)\right\|
\sum_{j=0}^{M-1}\sum_{r=0}^{\gamma_m}\left\|(\partial_{t^r}u^j)(\partial_{t^{\gamma_m-r}}v^{M-1-j})|\calk^{0}_{2,a}(K)\right\|\\
& \qquad + \left\|\sum_{r=1}^{\gamma_m-1}{\gamma_m\choose r}\partial_r(u-v)\partial_{\gamma_m-r}\left(\sum_{j=0}^{M-1}\ldots\right) |\calk^{0}_{2,a}(K)\right\|.  
\end{align*}
The lower order derivatives in the last line  cause no problems since we can (again) apply  Corollary \ref{thm-pointwise-mult-2}(i) as before. 
The term $\left\|u^jv^{M-1-j}|\calk^{2}_{2,a+2}(K)\right\|$ can now be further estimated with the help of Corollary \ref{thm-pointwise-mult-2}(i). For the term $\sum_{r=0}^{\gamma_m}\left\|(\partial_{t^r}u^j)(\partial_{t^{\gamma_m-r}}v^{M-1-j})|\calk^{0}_{2,a}(K)\right\|$ we again use Corollary \ref{thm-pointwise-mult-2}(ii), then proceed as in \eqref{est-33a} and see that the resulting estimate yields  \eqref{k=gamma_m}.\\ 
Moreover, in \eqref{kappa} we used the fact that in the step before we have two sums with  $\kappa_1+\ldots +\kappa_r\leq j$ and $\kappa_1+\ldots+\kappa_{k-l-r}\leq M-1-j$, i.e., we have $k-l$ different $\kappa_i$'s which leads to at most $k$ different $\kappa_i$'s if $l=0$.  We allow all combinations of $\kappa_i$'s and  redefine the $\kappa_i$'s in the second sum leading to $\kappa_1', \ldots , \kappa_k'$ with $\kappa_1'+\ldots+\kappa_k'\leq M-1$ and replace the old conditions $\kappa_1+2\kappa_2+r\kappa_r\leq r$ and $\kappa_1+2\kappa_2+(k-l-r)\kappa_{k-l-r}\leq k-l-r$ by the weaker ones $\kappa_1'+\ldots+\kappa_k'\leq k$ and $\kappa_k'\leq 1$. This causes no problems since the other terms appearing in this step do not depend on $\kappa_i$ apart from the product term. There, the fact that some of the old $\kappa_i$'s from both sums might coincide is reflected in the new exponent $4\kappa_i'$.   
%The sum of all these new $\kappa_i$, $i=1,\ldots, k$, is bounded from above by $M-1$.  
}
From Theorem \ref{thm-sob-emb} (Sobolev embedding) we {conclude} that 
\begin{eqnarray}
u,v \in S&\hookrightarrow & \bigcap_{k=0}^{\gamma_m+1}W^{{k}}_2((0,T),\mathcal{K}^{2m(\gamma_m-(k-1))}_{2,a+2m(\gamma_m-(k-1))}(K))
%\cap W^{{\gamma_m+1}}_2((0,T),W^m_2(K))
\notag\\
&\hookrightarrow & \bigcap_{k=1}^{\gamma_m+1}\mathcal{C}^{{k-1,\frac 12}}((0,T),\mathcal{K}^{2m(\gamma_m-(k-1))}_{2,a+2m(\gamma_m-(k-1))}(K))
%\cap \mathcal{C}^{{l+1,\frac 12}}((0,T),\mathcal{K}^{\gamma-2m}_{2,a}(K))
\notag\\
&\hookrightarrow & \bigcap_{k=1}^{\gamma_m+1}{C}^{{k-1}}((0,T),\mathcal{K}^{2m(\gamma_m-(k-1))}_{2,a+2m(\gamma_m-(k-1))}(K))\\
%\cap {C}^{{l+1}}((0,T),\mathcal{K}^{\gamma-2m}_{2,a}(K)), 
&=& \bigcap_{l=0}^{\gamma_m}{C}^{{l}}((0,T),\mathcal{K}^{2m(\gamma_m-l)}_{2,a+2m(\gamma_m-l)}(K)), 
\label{est-5}
\end{eqnarray}
hence, the term  {involving the maxima, $\max_{w\in \{u,v\}}\max_{l=0,\ldots, \gamma_m}\max (\ldots)^{M-1}$} in \eqref{est-4} is bounded {by $\max(R+\|\tilde{L}^{-1}f|S\|,1)^{M-1}$}.  
Moreover, since $u$ and $v$ are taken from $B_R(\tilde{L}^{-1}f)$  in {$S=S_1\cap  S_2$}, we obtain from \eqref{est-4},  
 \begin{align}
\|\tilde{L}^{-1}\| & \|u^M-v^M|D_1\|\notag \\
&\leq  {c_0}\|\tilde{L}^{-1}\|M\max(R+\|\tilde{L}^{-1}f|S\|,1)^{{2(M-1)}}\|u-v| S\|\notag\\
&\leq  {c_2}\|\tilde{L}^{-1}\|M\max(R+\|\tilde{L}^{-1}\|\cdot \|f|D\|,1)^{{2(M-1)}} \|u-v| S\|\notag\\
&= {c_2}\|\tilde{L}^{-1}\|M\max(R+\|\tilde{L}^{-1}\| \eta,1 )^{{2(M-1)}} \|u-v| S\|,   \label{est-ball}
\end{align}
%\textcolor{red}{hier muss evtl. noch $\max(r+\|\tilde{L}\|^{-1} \eta, 1)$ stehen...was dann als Bedingung fuer die Kontration unten $\|\tilde{L}^{-1}\|<\frac 1n$ ergibt.}
{where we put $\eta:=\|f|D\|$ in the last line, $c_0$ denotes the constant resulting from \eqref{est-3} and \eqref{est-4} and $c_2=c_0c_1$ with $c_1$ being the constant from the estimates in Theorem \ref{thm-weighted-sob-reg}}. 

We now turn our attention towards the second term $\|\tilde{L}^{-1}\| \|u^M-v^M|D_2\|$ in \eqref{est-1} and calculate \\
{
%$\|\tilde{L}\|^{-1}\|(u^M-v^M)|D_2\|$
\begin{align}
\|\tilde{L}^{-1}\|& \|(u^M-v^M)|D_2\| \notag\\
%&\leq &\|\tilde{L}\|^{-1}\|(u^M-v^M)|D\|\notag\\
%&= &\|\tilde{L}\|^{-1}\|(u^M-v^M)|W^{\gamma_m+1}((0,T),L_2(K))\|\notag\\
&= \|\tilde{L}^{-1}\|\left\|(u-v)\sum_{j=0}^{M-1} u^jv^{M-1-j}|W^{\gamma_m+1}((0,T),L_2(K))\right\|\notag\\
&= \|\tilde{L}^{-1}\|\sum_{k=0}^{\gamma_m+1}\left\|\partial_{t^k}\left[(u-v)\sum_{j=0}^{M-1} u^jv^{M-1-j}\right]|L_2((0,T),L_2(K))\right\| \notag\\
&=  \|\tilde{L}^{-1}\|\sum_{k=0}^{\gamma_m+1}\left\|\sum_{l=0}^k{k \choose l}\partial_{t^l}(u-v)\cdot \right. \notag\\
& \qquad \left.\left[\left(\sum_{j=0}^{M-1} \sum_{r=0}^{k-l} {{k-l}\choose r} \partial_{t^{r}}u^j \cdot \partial_{t^{k-l-r}}v^{M-1-j}\right)\right]|L_2(K_T)\right\| \notag\\
&\lesssim  \|\tilde{L}^{-1}\|\sum_{k=0}^{\gamma_m+1}\left\|\sum_{l=0}^k|\partial_{t^l}(u-v)|\cdot \right.\notag\\ 
& \qquad \left.\left[\left(\sum_{j=0}^{M-1} \sum_{r=0}^{k-l}  |\partial_{t^{r}}u^j \cdot \partial_{t^{k-l-r}}v^{M-1-j}|\right)\right]|L_2(K_T)\right\|, \notag\\
\label{est-1a}
\end{align}
where we used Leibniz's formula twice as in \eqref{est-2} in the second but last line. Again  Fa\`{a} di Bruno's formula, cf. \eqref{FaaDiBruno}, is applied  in order to estimate the derivatives in \eqref{est-1a}. We use a special case of the multiplier result from \cite[Sect. 4.6.1, Thm.~1(i)]{RS96}, {which states that for parameters $s_1> s_2$, $s_1+s_2>d\max\left(0,\frac 2p-1\right)$, $s_2>\frac dp$, and $q\geq \max(q_1,q_2)$,  we have
\[
\|uv|F^{s_1}_{p,q_1}\|\lesssim \|u|F^{s_2}_{p,q_2}\|\cdot \|v|F^{s_1}_{p,q_1}\|,  
\]
where $F^s_{p,q}$ denote the Triebel-Lizorkin spaces closely linked with the Besov spaces by interchanging the order in which the $\ell_q-$ and $L_p-$Norms are taken, cf. \cite{RS96} and the references given there.  In particular, choosing $s_1=0$, $s_2=m\geq 2$, $d=3$, $q_1=q_2=p=2$ and using the {identity} $F^0_{2,2}=L_2$ and $F^m_{2,2}=W_2^m$, we obtain }
%for $s_1\leq s_2$, $s_1+s_2>d\max\left(0,\frac{2}{p}-1\right)$, $s_2>\frac dp$, and $q\geq \max(q_1,q_2)$,  it holds $F^{s_1}_{p,q_1}\cdot F^{s_2}_{p,q_2}\hookrightarrow F^{s_1}_{p,q_1}$. 
%For parameters $s_1=0$, $s_2=m>1$, $d=3$, $q_1=q_2=q=2$, and $p=2$ this yields   
\begin{equation}\label{multiplier-lim}
\|uv|L_2\|\lesssim \|u|W^m_2\|\cdot \|v|L_2\|.  
\end{equation}
This is exactly the point where our assumption $m\geq 2$ comes into play, since $s_2=m>\frac dp=\frac 32$ is needed.  
With this we obtain 
\begin{align}
&\Big\|\partial_{t^r}u^j| L_2(K)\Big\| \notag\\
&\leq  c_{r,j}\left\|\sum_{k_1+\ldots +k_r\leq j} u^{j-(k_1+\ldots +k_r)}\prod_{{i}=1}^r \left|\partial_{t^{{i}}}u\right|^{k_{{i}}}| L_2(K)\right\| \notag\\
&\lesssim \sum_{k_1+\ldots +k_r\leq j} \left\| u| W^m_2(K)\right\|^{j-(k_1+\ldots +k_r)} \prod_{{i}=1}^{r-1} \left\| \partial_{t^{{i}}}u| W_2^m(K)\right\|^{k_{{i}}} \left\| \partial_{t^r}u| L_2(K)\right\|^{k_r}. \notag \\ \label{est-3a}
\end{align}
Similar  for $\partial_{t^{k-l-r}}v^{M-1-j}$. As before, from \eqref{cond-kr}  we observe {$k_{r}\leq 1$, therefore the highest derivative $u^{(r)}$} appears at most once.  {Note that since $W^m_2(K)$ is a multiplication algebra for $m\geq 2$, we get \eqref{est-3a} with $L_2(K)$ replaced by $W^m_2(K)$ as well.} 
Now {\eqref{multiplier-lim} and \eqref{est-3a}  inserted in \eqref{est-1a}} gives %\\
%$\|\tilde{L}^{-1}\|\|u^M-v^M|D_2\|$
\begin{align}
\|& \tilde{L}^{-1}\| \|u^M-v^M|D_2\|\notag\\
& {
= \|\tilde{L}^{-1}\|\sum_{k=0}^{\gamma_m+1}\Bigg(\int_0^T\left\|\partial_{t^k}(u-v)\sum_{j=0}^{M-1}u^jv^{M-1-j}| L_2(K)\right\|^2 \ud t\Bigg)^{1/2}
}\notag\\
& {
\lesssim \|\tilde{L}^{-1}\|\sum_{k=0}^{\gamma_m+1}\sum_{l=0}^k\Bigg(\int_0^T\left\|\partial_{t^l}(u-v)|W^m_2(K)\right\|^2 }\notag\\
& \qquad\qquad  {\sum_{j=0}^{M-1}\sum_{r=0}^{k-l}\left\|\partial_{t^r}u^j \cdot \partial_{t^{k-l-r}}v^{M-1-j}| L_2(K)\right\|^2 \ud t\Bigg)^{1/2}
}\notag\\
& {
\lesssim \|\tilde{L}^{-1}\|\sum_{k=0}^{\gamma_m+1}\sum_{l=0}^k\Bigg(\int_0^T\bigg\{\left\|\partial_{t^l}(u-v)|W^m_2(K)\right\|^2 }\notag\\
& \qquad\qquad  {\sum_{j=0}^{M-1}\sum_{r=0, \atop (k-l-r\neq \gamma_m+1)\wedge (r\neq \gamma_m+1)}^{k-l}\left\|\partial_{t^r}u^j|W^m_2(K)\|^2 \|\partial_{t^{k-l-r}}v^{M-1-j}| W^m_2(K)\right\|^2 }\notag \\
& {
\qquad \qquad +\|u-v|W^m_2(K)\|^2\|\partial_{t^{\gamma_m+1}}u^j|L_2(K)\|^2\|v^{M-1-j}|W^m_2(K)\|^2}\notag \\
& \qquad \qquad {+ \|u-v|W^m_2(K)\|^2\|u^j|W^m_2(K)\|^2\|\partial_{t^{\gamma_m+1}}v^{M-1-j}|L_2(K)\|^2
\bigg\}\ \ud t\Bigg)^{1/2}
}\notag\\
%%%& \textcolor{green}{\qquad\qquad+ \|u^j|W^m_2(K)\|^2 \|\partial_{t^{k-l}}v^{M-1-j}|%%%L_2(K)\|^2 } \notag \\
%%%& \textcolor{green}{\qquad \qquad + \|\partial_{t^{k-l}}u^j|L_2(K)\|^2 \|v^{M-1-j}|%%%W^m_2(K)\|^2 \bigg\}\ \ud t\Bigg)^{1/2}
%%%}\notag\\
%
%
%
%}\notag \\
%&\lesssim & \|\tilde{L}\|^{-1}\sum_{k=0}^{l+1}\sum_{l=0}^k\left(\int_0^T\left\|\frac{\partial}{\partial t^k}\left[(u-v)\sum_{j=0}^{n-1} u^jv^{n-1-j}\right]|L_2(K))\right\|^2\ud t\right)^{1/2}\notag\\
&{\lesssim \|\tilde{L}^{-1}\|\sum_{k=0}^{\gamma_m+1}\sum_{l=0}^k\Bigg(\int_0^T\left\|\partial_{t^l}(u-v)| W^m_2(K)\right\|^2 \cdot }\notag \\
& {\qquad \qquad \sum_{j=0}^{M-1} \sum_{r=0}^{k-l}\sum_{\kappa_1+\ldots +\kappa_{r}\leq j, \atop \kappa_1+2\kappa_2+\ldots+ r\kappa_{r}\leq r}
  \left\| u| W^m_2(K)\right\|^{2(j-(\kappa_1+\ldots +\kappa_{r}))}} \notag \\
   & { \qquad 
\left.\begin{cases}   \left\| \partial_{t^{{r}}}u| L_2(K)\right\|^{2\kappa_{{r}}}\prod_{{i}=1}^{r-1} \left\| \partial_{t^{{i}}}u| W^m_2(K)\right\|^{2\kappa_{{i}}}, & r=\gamma_m+1,\\
\prod_{{i}=1}^{r} \left\| \partial_{t^{{i}}}u| W^m_2(K)\right\|^{2\kappa_{{i}}},    & r\neq \gamma_m+1
 \end{cases}  \right\}
}\notag \\
  & {\qquad  \sum_{\kappa_1+\ldots +\kappa_{k-l-r}\leq M-1-j, \atop \kappa_1+2\kappa_2+\ldots+(k-l-r)\kappa_{k-l-r}\leq k-l-r} 
  \left\| v| W^m_2(K)\right\|^{2(M-1-j-(\kappa_1+\ldots +\kappa_{k-l-r}))} }\notag \\
    & {\qquad 
\left.\begin{cases}   \left\| \partial_{t^{{r}}}v| L_2(K)\right\|^{2\kappa_{{r}}}\prod_{{i}=1}^{k-l-r-1} \left\| \partial_{t^{{i}}}v| W^m_2(K)\right\|^{2\kappa_{{i}}}, & k-l-r=\gamma_m+1,\\
\prod_{{i}=1}^{l-k-r} \left\| \partial_{t^{{i}}}v| W^m_2(K)\right\|^{2\kappa_{{i}}},    & k-l-r\neq \gamma_m+1
 \end{cases}  \right\}
\ud t\Bigg)^{1/2}}\notag\\
&{\lesssim \|\tilde{L}^{-1}\|\sum_{k=0}^{\gamma_m+1}\Bigg(\int_0^T\left\|\partial_{t^k}(u-v)| W^m_2(K)\right\|^2 \cdot }\notag \\
& {\qquad \qquad M \sum_{\kappa_1'+\ldots +\kappa_{k}'\leq \min\{M-1,k\}}
  \max_{w\in \{u,v\}}\left\| w| W^m_2(K)\right\|^{2(M-1-(\kappa_1'+\ldots +\kappa_{k}'))}} \notag \\
   & { \qquad 
\left.\begin{cases}   \max(\left\| \partial_{t^{{k}}}w| L_2(K)\right\|^{4\kappa_{k}'}\prod_{{i}=1}^{k-1} \left\| \partial_{t^{{i}}}w| W^m_2(K)\right\|^{4\kappa_{i}'}, 1), & k=\gamma_m+1,\\
\max(\prod_{{i}=1}^{k} \left\| \partial_{t^{{i}}}w| W^m_2(K)\right\|^{4\kappa_{i}'}, 1),    & k\neq \gamma_m+1
 \end{cases}  \right\}
\ud t\Bigg)^{1/2}}\notag\\
%\sum_{()} \| u|\mathcal{K}^{\gamma-2m+2}_{2,a+2}(K)\|^{2j} \cdot \|v|\mathcal{K}^{\gamma-2m+2}_{2,a+2}(K)\|^{2(n-1-j)}
%\ud t\right)^{1/2}\\
%&\leq & \|\tilde{L}\|^{-1}\left(\int_0^T\|(u-v)| \mathcal{K}^{\gamma}_{2,a'}(K)\|^2
%\sum_{j=0}^{n-1} \| u|\mathcal{K}^{\gamma}_{2,a'}(K)\|^{2j} \cdot \|v|\mathcal{K}^{\gamma}_{2,a'}(K)\|^{2(n-1-j)}
%\ud t\right)^{1/2}\\
&\lesssim \|\tilde{L}^{-1}\|M
\|u-v| W^{\gamma_m+1}((0,T),W^m_2(K))\|^2\max_{w\in \{u,v\}}\max_{{i}=0,\ldots, \gamma_m} \max\notag\\
&  \qquad \left(
\left\| \partial_{t^{{i}}}w|L_{\infty}((0,T),W^m_2(K))\right\|,\  \left\| \partial_{t^{\gamma_m+1}}w|L_{\infty}((0,T),L_2(K))\right\|,\ 1\right)^{{{2(M-1)}}}\notag\\ \label{est-4a}
\end{align}
{Similar to \eqref{est-4} in the calculations above the term $k=\gamma_m+1$ required some special care. For the redefinition of the $\kappa_i$'s in the second but last line in \eqref{est-4a} we refer to the explanations given after \eqref{est-4}. }
From Theorem \ref{thm-sob-emb} we see that 
\begin{eqnarray}
u,v \in S&\hookrightarrow & W^{{\gamma_m+1}}_2((0,T),{\mathring{W}^m_2}(K))\cap W^{{\gamma_m+2}}_2((0,T),L_2(K))
\notag\\
&\hookrightarrow & \mathcal{C}^{{\gamma_m,\frac 12}}((0,T),{\mathring{W}^m_2}(K))\cap \mathcal{C}^{{\gamma_m+1,\frac 12}}((0,T),L_2(K))
\notag\\
&\hookrightarrow & {C}^{{\gamma_m}}((0,T),{\mathring{W}^m_2}(K))
\cap {C}^{{\gamma_m+1}}((0,T),L_2(K)), 
\label{est-4aa}
\end{eqnarray}
hence the term  ${\max_{w\in \{u,v\}}\max_{m=0,\ldots, l}\max(\ldots)^{M-1}}$ in \eqref{est-4a} is bounded.  Moreover, since $u$ and $v$ are taken from $B_R(\tilde{L}^{-1}f)$  in $S_2={\mathring{W}^{m,\gamma_m+2}_2(K_T)\hookrightarrow }W^{\gamma_m+1}_2((0,T),{\mathring{W}^m_2}(K))\cap  W^{\gamma_m+2}((0,T),L_2(K)) $, as in \eqref{est-ball} we obtain from \eqref{est-4a} {and \eqref{est-4aa}},  
 \begin{align}
\|\tilde{L}^{-1}\|\|u^M-v^M|D_2\|\leq c_3 
 \|\tilde{L}^{-1}\|M\max(R+\|\tilde{L}^{-1}\| \eta, 1)^{{2(M-1)}}\cdot \|u-v| S\|, \notag\\ \label{est-ball_a}
\end{align}
where we put $\eta:=\|f|D\|$ {and $c_3$ denotes the constant arising from our estimates \eqref{est-4a} and \eqref{est-4aa} above}. 
}
Now \eqref{est-1} together with \eqref{est-ball} and \eqref{est-ball_a} yields 
\begin{align}\label{est-ab}
\|\tilde{L}^{-1}(u^M-v^M)|S\|
& \leq \|\tilde{L}^{-1}\|\|(u^M-v^M)|D\|\notag \\
&\leq c \|\tilde{L}^{-1}\|M\max(R+\|\tilde{L}^{-1}\|\eta, 1)^{M-1}\|u-v|S\|, 
\end{align}
where $c=c_2+c_3$. 
For $\tilde{L}^{-1}\circ N$ to be a contraction, we therefore require 
\[
{c}\varepsilon \|\tilde{L}^{-1}\|M\max(R+\|\tilde{L}^{-1}\|\eta,1)^{{2(M-1)}}<1,
\]
{cf. \eqref{est-0}.} In case of $\ \max(R+\|\tilde{L}^{-1}\|\eta,1)=1$ this leads to 
\begin{equation}\label{cond-01}
\|\tilde{L}^{-1}\|<\frac{1}{{c}\varepsilon M}.
\end{equation}
On the other hand, if  $\ \max(R+\|\tilde{L}^{-1}\|\eta,1)=R+\|\tilde{L}^{-1}\|\eta$, we choose  $R=(r_0-1)\eta\|\tilde{L}^{-1}\|$, which gives {rise to} the condition 
\begin{equation}\label{cond-1}
{c}\varepsilon\|\tilde{L}^{-1}\|M(r_0\|\tilde{L}^{-1}\|\eta)^{{2(M-1)}}<1,\quad {\text{i.e.,}} \quad \eta^{{2(M-1)}} \|\tilde{L}^{-1}\|^{{2M-1}}<\frac{1}{{c}\varepsilon M}\left(\frac{1}{r_0}\right)^{{2(M-1)}}.
\end{equation}
{The next step is to show} that $(\tilde{L}^{-1}\circ N)(B_R(\tilde{L}^{-1}f))\subset B_R(\tilde{L}^{-1}f)$ in $S$. Since $(\tilde{L}^{-1}\circ N)(0)=\tilde{L}^{-1}(f-\varepsilon 0^M)=\tilde{L}^{-1}f$, we only need to apply the above estimate \eqref{est-ab} with $v=0$. This gives 
\begin{align*}
\varepsilon\|\tilde{L}^{-1}u^M|S\|
&\leq {c}\varepsilon \|\tilde{L}^{-1}\|M\max(R+\|\tilde{L}^{-1}\| \eta,1)^{{2(M-1)}}(R+\|\tilde{L}^{-1}\| \eta)\\
&\overset{!}{\leq}R=(r_0-1)\eta\|\tilde{L}^{-1}\|, 
\end{align*}
which, {in case that}  $\max(R+\|\tilde{L}^{-1}\|\eta,1)=1$,  leads to 
\begin{equation}\label{cond-02}
\|\tilde{L}^{-1}\|<\frac{r_0-1}{r_0}\left(\frac{1}{{c}\varepsilon M}\right), 
\end{equation}
whereas for $\max(R+\|\tilde{L}^{-1}\|\eta,1)=R+\|\tilde{L}^{-1}\|\eta$ we get 
\begin{equation}\label{cond-2}
{\eta^{2(M-1)} \|\tilde{L}^{-1}\|^{2M-1}\leq \frac{1}{{c}\varepsilon M}(r_0-1)\left(\frac{1}{r_0}\right)^{2M-1}. }
\end{equation}
We see that condition \eqref{cond-02} implies \eqref{cond-01}. Furthermore, since 
\[
{(r_0-1)\left(\frac{1}{r_0}\right)^{2M-1}=\frac{r_0-1}{r_0}\left(\frac{1}{r_0}\right)^{2(M-1)}<\left(\frac{1}{r_0}\right)^{2(M-1)}, }
\]
also condition \eqref{cond-2} implies \eqref{cond-1}. 
%Now choosing $r_0=\frac{n}{n-1}$ leads to the condition 
%\[
%\|\tilde{L}\|^{-1}\eta<\left(\frac{1}{\|\tilde{L}\|^{-1}}\right)^{\frac{1}{n-1}}\left(\frac{1}{n-1}\right)^{\frac{1}{n-1}}
%\left(\frac{n-1}{n}\right)^{\frac{n}{n-1}}=
%\left(\frac{1}{\|\tilde{L}\|^{-1}}\right)^{\frac{1}{n-1}}\left(\frac{1}{n}\right)^{\frac{1}{n-1}}
%\left(\frac{n-1}{n}\right), 
%\]
%which coincides with \eqref{cond-1}. 
Thus, by applying Banach's fixed point theorem in a sufficiently small ball around the solution of the corresponding linear problem, we obtain {a unique} solution of problem \eqref{parab-nonlin-1}. 
\end{proof}

\remark{
The restriction $m\geq 2$ in Theorem \ref{nonlin-B-reg1} comes from the fact that we require $m>\frac dp=\frac 32$ in \eqref{multiplier-lim}. This assumption can probably be weakened, since we expect   the solution to satisfy  $u\in L_2((0,T), W^{s}_2(K))$ for all $s<\frac 32$, see also Remark \ref{gen-thm-parab-Besov} and the explanations given there.\\ 
Moreover, the restriction  {$a\geq -\frac 12$} in  Theorem \ref{nonlin-B-reg1} comes from Corollary \ref{thm-pointwise-mult-2} that we applied. 
Together with the restriction $a\in [-m,m]$ we are looking for  $a\in [-\frac 12,m]$ if the cone $K$ is smooth. For polyhedral cones with edges $M_k$, $k=1,\ldots, n$,  we   furthermore require $\delta_-^{(k)}<a+2m(\gamma_m-l)+m<\delta^{(k)}_+$ for $l=0,\ldots, \gamma_m$  from Theorem \ref{thm-weighted-sob-reg}. 
%{We also refer to Remark \ref{rem_heat_eq} for further information. For  the heat equation it follows from the calculations there, that  the restriction  $a\geq 0$ is satisfied if the underlying polyhedral cone of our problem is convex, i.e., we need  $\theta_k<\pi$ for all $k=1,\ldots, n$. 
%}
}

From Theorem \ref{nonlin-B-reg1}  we obtain the following result concerning  Besov regularity of our nonlinear parabolic problem \eqref{parab-nonlin-1}.

\begin{theorem}[Nonlinear Besov regularity]\label{nonlin-B-reg3}
Let $\tilde{L}$ and $N$ be as described above {and let} the assumptions of Theorem \ref{thm-weighted-sob-reg} be satisfied. 
%Assume the assumptions of Theorem \ref{nonlin-B-reg1} are satisfied and, 
Additionally, we assume {$\gamma_m\geq 1$}, $m\geq 2$, and {$a\geq -\frac 12$}. Furthermore,  let $D$ and $S$ be as in Theorem  \ref{nonlin-B-reg1} and suppose that $f\in D$.
Put $\eta:=\|f|D\|$ and $r_0>1$. Moreover, we choose $\varepsilon >0$ so small that 
\begin{equation}\label{nonlin-cond1}
{
\eta^{2(M-1)} \|\tilde{L}^{-1}\|^{2M-1}\leq \frac{1}{{c}\varepsilon M}(r_0-1)\left(\frac{1}{r_0}\right)^{2M-1}, \qquad \text{if}\quad  r_0\|\tilde{L}^{-1}\|\eta>1,
}
\end{equation}
and 
\begin{equation}\label{nonlin-cond2}
\|\tilde{L}^{-1}\|<\frac{r_0-1}{r_0}\left(\frac{1}{\varepsilon M}\right), \qquad \text{if}\quad  r_0\|\tilde{L}^{-1}\|\eta<1.
\end{equation}
Let {$\varphi$} denote the cut-off function as described in \eqref{cutoff}. 
\begin{minipage}{0.6\textwidth}
Then there exists a solution $u$ to \eqref{parab-nonlin-1}, whose truncated version  $\varphi u$ satisfies 
 ${\varphi} u\in B_0\subset B$, 
$$B:=L_2((0,T),B^{\alpha}_{\tau,\infty}(K)), % \quad , 
$$ 
{\text{for all} $0<\alpha<\min(3m,\gamma)$}, $\frac 12<\frac{1}{\tau}<\frac{\alpha}{3}+\frac 12$,  and  $B_0$ denotes a small ball  around $\tilde{L}^{-1}f$ (the solution of the corresponding linear problem) with radius $R={C\tilde{C}}(r_0-1)\eta \|\tilde{L}^{-1}\|$. 
\end{minipage}\hfill \begin{minipage}{0.33\textwidth}
\includegraphics[width=4.5cm]{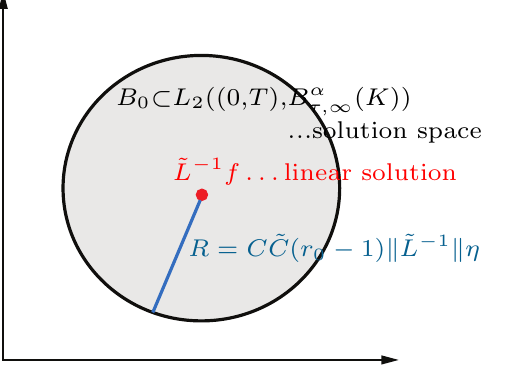}
\captionof{figure}{Nonlinear solution in $B_0$}
\end{minipage}
\end{theorem}

\begin{proof}
This is a consequence of the regularity results in Kondratiev (and Sobolev) spaces from Theorem  \ref{nonlin-B-reg1}. To be more precise, Theorem  \ref{nonlin-B-reg1} establishes the existence of a fixed point $u$ in  
\begin{align*}
S_0\subset S
&{:= \bigcap_{k=0}^{\gamma_m+1}W_2^{k}((0,T),\mathcal{K}^{2m(\gamma_m-(k-1))}_{2,a+2m(\gamma_m-(k-1))}(K)) \cap \mathring{W}^{m,\gamma_m+2}_2(K_T)}\\
&\hookrightarrow  \bigcap_{k=0}^{\gamma_m+1}W_2^{k}((0,T),\mathcal{K}^{2m(\gamma_m-(k-1))}_{2,a+2m(\gamma_m-(k-1))}(K)) \\
&\qquad \qquad \cap W^{\gamma_m+1}_2((0,T),{{W}^m_2}(K))\cap W^{\gamma_{m}+2}_2((0,T), L_2(K)) \\
& \hookrightarrow L_2((0,T), \mathcal{K}^{2m(\gamma_m+1)}_{2,a+2m(\gamma_m+1)}(K)\cap W^m_2(K))=:S'. 
\end{align*}
%By Theorem \ref{thm-appendix-fixpoint}  the fixed points coincide, i.e., $u=u_1=u_2$, on some subset $S_0\subset S_0^{1}\cap S_0^{2}$ of $S$.  
This together with the embedding results for Besov spaces from Theorem  \ref{thm-hansen-gen} {(choosing $k=0$)} completes the proof, {in particular, we calculate for the solution (cf. the proof of Theorem \ref{thm-parab-Besov})
\begin{align}
\| \varphi u&-\varphi \tilde{L}^{-1}f |L_2((0,T),B^{\alpha}_{\tau,\infty}(K))\|\notag \\
&\leq  C \|\varphi u-\varphi \tilde{L}^{-1}f| L_2((0,T), \mathcal{K}^{2m(\gamma_m+1)}_{2,a+2m(\gamma_m+1)}(K)\cap W^m_2(K))\|\notag\\
&\leq  C\tilde{C}\|u-\tilde{L}^{-1}f| S'\|
\leq  C\tilde{C}\|u-\tilde{L}^{-1}f| S\|\notag\\
&\leq  C\tilde{C}(r_0-1)\eta \|\tilde{L}^{-1}\|, \label{est-abc}
\end{align}
where in the second step we used that $\varphi\in C_0^{\infty}(K)$ is a multiplier for Kondratiev spaces, cf. \eqref{multiplier}. Furthermore,  it can be seen from \eqref{est-abc} that new constants $C$ and $\tilde{C}$ appear when considering the radius $R$ around the linear solution where the problem can be solved compared to Theorem \ref{nonlin-B-reg1}.}
\end{proof}

\remark{A few words concerning the parameters appearing in Theorem \ref{nonlin-B-reg3} (and also Theorem  \ref{nonlin-B-reg1}) seem to be in order. 
Usually, the operator norm  $\|\tilde{L}^{-1}\|$ as well as  $\varepsilon$ are fixed; but we can change $\eta$ and $r_0$ according to our needs. From this we see that by choosing $\eta$ small enough the  {condition \eqref{nonlin-cond2} can always be satisfied.} Moreover, one can  see easily that the smaller the nonlinear perturbation $\varepsilon>0$ is, the larger we can choose the radius $R$ of the ball $B_0$ where the  solution {to the nonlinear problem} is unique. 
}

\subsection{Space-time adaptivity and  H\"older-Besov regularity}
\label{sect-spacetime}

So far we have not exploited the fact that Theorem \ref{thm-weighted-sob-reg}  not only provides regularity {properties}  of the  solution $u$ of \eqref{parab-1a} but also of  {its} partial derivatives $\partial_{t^k} u$. In this section we will use this fact together with  Theorem \ref{thm-sob-emb} ({generalized Sobolev's embedding theorem})  in order to obtain some  mixed H\"older-Besov regularity results on the whole space-time cylinder $K_T$. \\ 
For parabolic SPDEs, results in this direction have been obtained in \cite{CKLL13}. {However, for SPDEs, the time regularity is limited in nature. This is caused by the nonsmooth character of the driving processes. Typically, H\"older regularity $\mathcal{C}^{0,\beta}$ can be obtained, but not more. In contrast to this, it is well-known that deterministic parabolic PDEs are smoothing in time. Therefore, we expect}  that  in the deterministic case considered here, higher regularity results in time can be obtained {compared to} the probabilistic setting.   \\

\begin{theorem}[H\"older-Besov regularity]\label{Hoelder-Besov-reg}
{Let $\gamma\in \nat $ with  $\gamma\geq 4m+1$ and put $\gamma_m:=\lfloor \frac{\gamma-1}{2m}\rfloor$. Furthermore, let  $a\in \real$ with   ${a\in [-m,m]}$.  Assume that  the right hand side $f$ of \eqref{parab-1a} satisfies 
\begin{itemize}
\item[(i)] $\partial_{t^k} f\in L_2(K_T)\cap L_2((0,T),\mathcal{K}^{2m(\gamma_m-k)}_{2,a+2m(\gamma_m-k)}(K))$, \ $k=0,\ldots, \gamma_m$, \\ and 
$\partial_{t^{\gamma_m+1}} f\in L_2(K_T)$. 
\item[(ii)] $\partial_{t^k} f(x,0)=0$, \quad  $k=0,1,\ldots, {\gamma_m}.$
\end{itemize}}
{Furthermore, let  Assumption \ref{assumptions}  hold for weight parameters $b=a+2m(\gamma_m-i)$, where $i=0,\ldots, \gamma_m$, and  $b'=-m$.}
%Then for the generalized solution $u\in {\mathring{W}}_2^{m,\gamma_m+2}(K_T)$
%Additionally, suppose that the closed strip between the lines $\mathrm{Re}\lambda=m-\frac 32$ and $\mathrm{Re}\lambda=a-\frac 32$ does not contain eigenvalues of the operator pencils $\mathfrak{U}(\lambda,t)$, $t\in [0,T]$  and 
%\begin{equation}\label{restr-1}
%-\delta_+^{(k)}<\delta_k-m<\delta_{-}^{(k)}, \quad k=1,\ldots, n.
%\end{equation}
Let ${\varphi}$ denote the cutoff function from \eqref{cutoff}. Then for the  solution $u\in {\mathring{W}}_2^{m,\gamma_m+2}(K_T)$ of problem \eqref{parab-1a}, we have 
$$
{\varphi}u\in \mathcal{C}^{{\gamma_m-2},\frac 12}((0,T),B^{\eta}_{\tau,\infty}(K)) \quad \text{for all}\quad 0<\eta<3m, \quad \frac 12<\frac{1}{\tau}<\frac{\eta}{3}+\frac 12.
$$  
In particular, we have the a priori estimate 
\begin{align*}
\|{\varphi}u&|\mathcal{C}^{{\gamma_m-2},\frac 12}((0,T),B^{\eta}_{\tau,\infty}(K))\|\\
&\lesssim  {\sum_{k=0}^{\gamma_m}\|\partial_{t^k} f|L_2((0,T), \mathcal{K}^{2m(\gamma_m-k)}_{2,a+2m(\gamma_m-k)}(K))\|+\sum_{k=0}^{\gamma_m+1}\|\partial_{t^k} f|{L_2(K_T)}\|},
\end{align*}
where the constant is independent of $u$ and $f$. 
\end{theorem}

\begin{proof} Theorem \ref{thm-weighted-sob-reg} and  Proposition \ref{GenSol_Sobolev} show together with Theorems  \ref{thm-hansen-gen}  and  \ref{thm-sob-emb}, that under the given assumptions on the initial data $f$, we have {for $k\leq  \gamma_m-2$, }
\begin{eqnarray*}
{\varphi}u &\in & {W^{k+1}_2((0,T), {\mathcal{K}^{2m(\gamma_m-k)}_{2,a+2m(\gamma_m-k)}(K)})}\cap W^{{\gamma_m}+1}_2((0,T), W^m_2(K))\\
&\hookrightarrow & {W^{k+1}_2((0,T), {\mathcal{K}^{2m(\gamma_m-k)}_{2,a+2m(\gamma_m-k)}(K)}\cap W^m_2(K))}\\
&\hookrightarrow &{\mathcal{C}^{{k},\frac 12}((0,T), {\mathcal{K}^{2m(\gamma_m-k)}_{2,a+2m(\gamma_m-k)}(K)}\cap W^m_2(K)) }\\
&\hookrightarrow &{\mathcal{C}^{{k},\frac 12}((0,T), {\mathcal{K}^{\eta}_{2,a+2m(\gamma_m-k)}(K)}\cap W^m_2(K)) }\\
&\hookrightarrow & \mathcal{C}^{{k},\frac 12}((0,T), B^{\eta}_{\tau,\infty}(K)), 
\end{eqnarray*}
{where in the third step we require {$\eta\leq 2m(\gamma_m-k)$} and by Theorem \ref{thm-hansen-gen} we get the additional restriction  
\[
m=\min(m,a+2m(\gamma_m-k))\geq \frac{\eta}{3}, \quad \text{i.e.}, \quad \eta <3m. 
\]
Therefore, the upper bound on $\eta$ reads as $\eta<\min(3m, 2m(\gamma_m-k))=3m$ since $k\leq \gamma_m-2$.}   
\end{proof}

{
\rem{\hfill 
\begin{itemize}
\item[(i)] For {$\gamma\geq 2m+1$ and} $k=\gamma_m-1$ we have $\eta\leq 2m$ in the theorem above.  For {$\gamma\geq 2m$ and} $k=\gamma_m$ we get $\eta=0$. 
\item[(ii)] From the proof of Theorem \ref{Hoelder-Besov-reg} above it can be seen that the solution satisfies 
$$u\in {\mathcal{C}^{{k},\frac 12}((0,T), {\mathcal{K}^{2m(\gamma_m-k)}_{2,a+2m(\gamma_m-k)}(K)}), }$$
implying that for high regularity in time, which is displayed by the parameter $k$, we have less spacial regularity in terms of $2m(\gamma_m-k)$. 
\end{itemize}
}
}

\section{Parabolic Besov regularity on general Lipschitz domains}
\label{Sect-4a}

We turn our attention towards Besov regularity results for parabolic PDEs on general Lipschitz domains using regularity results in weighted Sobolev spaces from  \cite{Kim11}. 
{
It is important to note that for general Lipschitz domains the definition of the Kondratiev spaces is different when compared to polyhedral domains, since in this case the whole boundary $\partial \mathcal{O}$ coincides with the singular set. Therefore, the singularities induced by the boundary have a much stronger influence. As a consequence, the regularity results for polyhedral cones are much stronger compared to the Lipschitz case, see Remark \ref{rem-B6} below. \\
Surprisingly, it turns out that the spatial regularity results in the deterministic case are more or less the same as for the case of  SPDEs that was  already studied in \cite{Cio-Diss} based on \cite{Kim04, Kim12}. \\
However, for the time regularity we nevertheless expect a significant difference, cf. Subsection \ref{sect-spacetime}. Moreover, the reader should observe that our results stated in Section \ref{Sect-4} are more general in the sense that there differential operators of arbitrary order are considered {but with smooth $C^{\infty}$ coefficients}, whereas the analysis in this section is restricted to second order operators. 
}

Let $\mathcal{O}\subset \rd$ be a bounded Lipschitz domain and put $\varrho(x)=\mathrm{dist}(x,\partial \mathcal{O})$ . We  consider the following class of parabolic equations 
\begin{equation}\label{parab-Lipschitz}
\left\{\begin{array}{lcll}
\frac{\partial}{\partial t}u&=& \sum_{i,j=1}^d a_{ij}{\frac{\partial^2}{\partial  {x_i}\partial {x_j}}} u+\sum_{i=1}^d b_i \frac{\partial}{\partial {x_i}} u+cu+f & \text{on }\mathcal{O}_T, \\[0.1cm]
u(0,\cdot )&=& u_0& \text{on }\mathcal{O},
\end{array} \right\} %\quad (\ast)
\end{equation}
where the coefficients are assumed to satisfy the assumptions listed below. We need some more notation. Put $\varrho(x,y)={\min(\varrho(x),\varrho(y))}$. For $\alpha\in \real$, $\delta\in (0,1]$, and $m\in \nat_0$, we set 
\begin{eqnarray*}
[f]_m^{(\alpha)}&:=& \sup_{x\in \mathcal{O}}\varrho^{m+{\alpha}}(x)|D^mf(x)|,\\
{[}f{]}_{m+\delta}^{(\alpha)}
&:=& \sup_{x,y\in \mathcal{O}\atop |\beta|=m}\varrho^{m+\alpha}(x,y)\frac{|D^{\beta}f(x)-D^{\beta}f(y)|}{|x-y|^{\delta}},\\
|f|_{m}^{(\alpha)}&:=& \sum_{l=0}^m [f]_l^{(\alpha)}\qquad \text{and}\qquad |f|^{(\alpha)}_{m+\delta}:=|f|^{(\alpha)}_m+[f]^{(\alpha)}_{m+\delta},
\end{eqnarray*}
whenever it makes sense. Furthermore, fix a constant $\varepsilon_{0}>0$. Then for $\gamma\geq 0$ we define 
\[
\gamma_+=
\begin{cases}
\gamma, & \text{if }\gamma\in \nat_0, \\
\gamma+\varepsilon_0, & \text{otherwise}. 
\end{cases}
\]

\begin{assumption}[Assumptions on the coefficients] \label{ass-coeff}\hfill 
\begin{itemize}
\item[(i)] Parabolicity: There are constants $\delta_0,K\in (0,\infty)$ such that for all $\lambda \in \rd$ it holds 
$$\delta_0|\lambda|^2\leq a_{ij}(t,x)\lambda_i\lambda_j\leq K|\lambda|^2.$$
\item[(ii)] The behaviour of the coefficients $b_i$ and $c$ can be controlled near the boundary of $\mathcal{O}$: 
$$\lim_{\varrho(x)\rightarrow 0, \atop x\in \mathcal{O}}\sup_{t}\left(\varrho(x)|b_i(t,x)|+\varrho^2(x)|c(t,x)|\right)=0.$$
\item[(iii)] The coefficients $a_{ij}(t,\cdot)$ are uniformly continuous in $x$, i.e., for any $\varepsilon>0$ there is a $\delta=\delta(\varepsilon)>0$ such that 
\[
\left|a_{ij}(t,{x})-a_{ij}(t,y)\right|<\varepsilon
\]
for all $x,y\in \mathcal{O}$ with $|x-y|<\delta$. 
\item[(iv)]  For any $t>0$,  
$$\ds \left|a_{ij}(t,\cdot)\right|^{(0)}_{{\gamma_+}}
+\left|\varrho(x)b_i(t,\cdot)\right|^{(0)}_{{\gamma_+}}
+\left|\varrho^2(x)c(t,\cdot)\right|^{(0)}_{{\gamma_+}}
\leq K. 
$$
\end{itemize}
\end{assumption}

We define  Kondratiev spaces $\mathcal{K}^m_{p,a}(\mathcal{O})$ on bounded Lipschitz domains {similar to}  \eqref{Kondratiev-1}, i.e., 
\begin{equation}\label{Kondratiev-1a}
\|u|\V^m_{p,a}(\mathcal{O})\|:=\left(\sum_{|\alpha|\leq m}\int_{\mathcal{O}} |\varrho(x)|^{p(|\alpha |-a)}|D^{\alpha}_x u(x)|^p\ud x\right)^{1/p}<\infty,
\end{equation}
where $a\in \real$, $1<p<\infty$, $m\in \nat_0$, $\alpha\in \nat^n_0$, and the weight function $\varrho: \mathcal{O}\rightarrow [0,1]$ now stands for  the smooth distance to the singular set of $\mathcal{O}$, {i.e.,   $\varrho(x)=\mathrm{dist}(x,\partial \mathcal{O})$}. 
We generalize the above Kondratiev spaces with the help of complex interpolation, cf. \cite{Lot00}, to non-integer values $m\geq 0$ as follows. If $0<\eta<1$,  $m_0, m_1\in \nat_0$, $p_0, p_1\in (1,\infty)$, and $a_0, a_1\in \real$,  put 
\[
\mathcal{K}^{m}_{p,a}(\mathcal{O}):=\left[\mathcal{K}^{m_0}_{p_0,a_0}(\mathcal{O}),\mathcal{K}^{m_1}_{p_1,a_1}(\mathcal{O})\right]_{\eta}, \qquad \text{where}\quad m=(1-\eta)m_0+\eta m_1, 
\]
$\frac 1p=\frac{1-\eta}{p_0}+\frac{\eta}{p_1}$, and $a=(1-\eta)a_0+\eta a_1$. \\
 
 \medskip

The following theorem was proven in  \cite[Th. 2.5]{Kim11}. For convenience of the reader we adapt the notation from the paper to our needs. 

\begin{theorem}[{Kondratiev regularity}]\label{thm-main}
Let $p\in [2,\infty)$, $\gamma\in [0,\infty)$, and Assumption \ref{ass-coeff} be satisfied. Then there exists $\beta_0=\beta_0(p,d,\mathcal{O})>0$ such that for 
\[
a\in \left(\frac{2-p-\beta_0}{p},\frac{2-p+\beta_0}{p}\right), 
\]
any $f\in L_p([0,T],\mathcal{K}^{\gamma}_{p,a-1}(\mathcal{O}))$, and initial data $u_0\in\mathcal{K}^{\gamma+2-\frac{2}{p}}_{p,a+\frac{p-2}{p}}(\mathcal{O})$, equation \eqref{parab-Lipschitz} admits a unique solution $u\in L_p([0,T],\mathcal{K}^{\gamma+2}_{p,{a+1}}(\mathcal{O}))$ with $\partial_t u\in L_p([0,T],\mathcal{K}^{\gamma}_{p,a-1}(\mathcal{O}))$. % and $u(\cdot, 0)\in \mathcal{K}^{\gamma+2-\frac 2p}_{p,a+\frac{p-2}{p}}(\Omega)$. 
In particular, we have 
\begin{align*}
\sum_{k=0}^1\Big\|\partial_{t^k}u|& L_p([0,T],\mathcal{K}^{\gamma+2-2k}_{p,{a+1-2k}}(\mathcal{O}))\Big\|\\
&\leq C\left(\| f|{L_p([0,T],\mathcal{K}^{\gamma}_{p,a-1}(\mathcal{O}))}\|+\big\|u_0|{\mathcal{K}^{\gamma+2-\frac{2}{p}}_{p,a+\frac{p-2}{p}}(\mathcal{O})}\big\|\right),
\end{align*}
where $C=C(d,p,\gamma,\theta,\delta_0,K,T,\mathcal{O})$. 
\end{theorem}

\remark{In \cite{KiKr04} similar results for $C^1$ domains were established. In particular, the condition on $a$ above in this case has to be replaced by
\[a\in \left(\frac 1p-1, \frac 1p \right).\]
}

We rewrite \cite[Th. 5.1]{Cio-Diss} and obtain the following embedding of weighted Sobolev spaces into the scale of Besov spaces. 

\begin{theorem} \label{emb-besov-han}
Let $\mathcal{O}\subset \real^d$ be a bounded Lipschitz domain. Fix $\gamma\in (0,\infty)$, $p\in [2,\infty)$, and $a\in \real$. Then 
\[
L_p([0,T],\mathcal{K}^{\gamma}_{p,a}(\mathcal{O}))\hookrightarrow L_p([0,T], B^{\alpha}_{\tau,\tau}(\mathcal{O})),
\]
for all $\alpha$ and $\tau$ with 
\[
\frac{1}{\tau}=\frac{\alpha}{d}+\frac 1p\quad \text{and} \quad 0<\alpha<\min\left\{\gamma, a\frac{d}{d-1}\right\}.
\]
\end{theorem}

\remark{{ In contrast to 
Theorems \ref{thm-hansen-1} and  \ref{thm-hansen-gen}  the result in  Theorem \ref{emb-besov-han}  is weaker, since  here we have the restriction $\alpha<a\frac{d}{d-1}$. On the other hand, in the embedding above no knowledge about the Sobolev regularity (or regularity in the spaces $B^s_{p,\infty}$) is needed. 
}} 

Using this, we get the following Besov regularity for the solutions of \eqref{parab-Lipschitz}.

\begin{theorem}[{Besov regularity}]\label{thm-main-2}
Let $p\in [2,\infty)$, $\gamma\in [0,\infty)$, and Assumption \ref{ass-coeff} be satisfied. Then there exists $\beta_0=\beta_0(p,d,\mathcal{O})>0$ such that for 
\[
a\in \left(\frac{2-p-\beta_0}{p},\frac{2-p+\beta_0}{p}\right), 
\]
any $f\in L_p([0,T],\mathcal{K}^{\gamma}_{p,a-1}(\mathcal{O}))$, and initial data $u_0\in\mathcal{K}^{\gamma+2-\frac{2}{p}}_{p,a+\frac{p-2}{p}}(\mathcal{O})$, equation \eqref{parab-Lipschitz} admits a  solution 
$$
u\in L_p([0,T],B^{\alpha}_{\tau,\tau}(\mathcal{O})), \qquad \frac{1}{\tau}=\frac{\alpha}{d}+\frac 1p, \qquad 0<\alpha<\min\left\{\gamma+2, (a+1)\frac{d}{d-1}\right\}. 
$$
In particular, we have 
\[
\left\|u| L_p([0,T],B^{\alpha}_{\tau,\tau}(\mathcal{O}))\right\|
\leq C\left(\| f|{L_p([0,T],\mathcal{K}^{\gamma}_{p,a-1}(\mathcal{O}))}\|+\big\|u_0|{\mathcal{K}^{\gamma+2-\frac{2}{p}}_{p,a+\frac{p-2}{p}}(\mathcal{O})}\big\|\right),
\]
where $C=C(d,p,\gamma,\theta,\delta_0,K,T,\mathcal{O})$. 
\end{theorem}

\medskip 

\remark{\label{rem-B6}
We calculate for $a+1$ appearing in the upper bound for $\alpha$ that 
\begin{itemize}
\item $\ds a+1\in \left(\frac{2-\beta_0}{p},\frac{2+\beta_0}{p}\right)$ for  Lipschitz domains $\mathcal{O}$
\item $\ds a+1\in \left(\frac{1}{p},1+\frac{1}{p}\right)$ for  $C^1$ domains $\mathcal{O}$ 
\end{itemize}
In particular, for $p=2$ and $d=3$ as an upper bound for $\alpha$ we get 
\[
\alpha<
\left.\begin{cases}
\min\left\{\gamma+2,\left(1+\frac{\beta_0}{2}\right)\frac 32\right\} & \text{for Lipschitz domains} \\
\min\left\{\gamma+ 2,\frac 94\right\} & \text{for $C^1$ domains}
\end{cases}\right\}{\leq }\ \frac 94 
\]
{(every $C^1$ domain is also a Lipschitz domain)}, whereas Theorem \ref{thm-parab-Besov} yields $\alpha<3$ (since we have $m=1$ in \eqref{parab-Lipschitz}). 
}

\expl{{\bf (Heat equation)}
%\paragraph{Heat equation} 
In \cite[L. 3.10]{Kim11} it is shown that the heat equation 
\[
u_t=\Delta u\quad \text{on } \mathcal{O}\times [0,T], \qquad u(\cdot, 0)=0,
\]
has a solution  $u\in L_p([0,T],\mathcal{K}^{2}_{p,\frac 2p}(\mathcal{O}))$. Using Theorem \ref{thm-main-2} we obtain for the Besov regularity of the solution (now $\gamma \gg 0$) that 
\[
u\in L_p\left([0,T],B^{\alpha}_{\tau,\tau}(\mathcal{O})\right)\quad \text{with}\quad \frac{1}{\tau}=\frac{\alpha}{d}+\frac 1p\quad \text{and}\quad 0<\alpha <{\frac 2p \cdot}\frac{d}{d-1}. 
\]
For $p=2$ and $d=3$ this {even} yields $\alpha<{\frac 32}$. {However, comparing this with our considerations in Example \ref{ex_heat_eq},} we see that {for} polyhedral cones $K\subset \real^3$ {our results concerning Besov regularity} are  better than what can be expected for the heat equation on  arbitrary Lipschitz domains. We {also} refer to \cite{AG12} in this context, where the {investigations} (subject to some restrictions) lead to ${\alpha}<\frac 32s $, {where $s$ stands for the (fractional) Sobolev regularity of the solution}. 

}

\remark{
Since $\beta_0$ depends also on $p$,  { Theorem \ref{thm-main} is not applicable in general if $p>2$ as there might be a problem to fulfill the assumptions on $a$. In this case Theorem \ref{thm-main} does not yield the existence of a solution $u\in L_p(\mathcal{O}_T)$ even if the data of the equation is assumed to be arbitrarily smooth.}  {As a  deterministic   counterexample the  heat equation is discussed after \cite[Th. 3.13]{Cio-Diss}. }
Thus,  one should distinguish between $p\in [2,\infty)$ and $p\in [2,p_0)$ as was done in \cite[Th. 3.13, Th. 5.2]{Cio-Diss} also in the deterministic case. 
}

\section{Hyperbolic Besov regularity}\label{Sect-5}

Our special Lipschitz domains $\Omega$ from Definition \ref{def_special_Lip} that we deal with  in this section are not bounded polyhedral domains as considered in Theorems \ref{thm-hansen-1}, \ref{thm-hansen-gen}. However, regarding embeddings of the Kondratiev spaces into the scale of Besov spaces,  modifying the arguments from \cite[Sect. 5, Thm. 3]{Han15}, we show that the results can be generalized to our context.

\begin{theorem}\label{emb-hyper-1}
Let $\Omega\subset \real^d$ be a special Lipschitz domain from Definition \ref{def_special_Lip}. Then we have a continuous embedding 
\begin{equation}
\mathcal{K}^m_{p,a}(\Omega)\cap B^s_{p,p}(\Omega)\hookrightarrow B^r_{\tau,\tau}(\Omega), \qquad \frac{1}{\tau}=\frac rd+\frac 1p, \qquad 1<p<\infty,
\end{equation}
for all $0\leq r<\min(m, \frac{sd}{{d-1}})$ and $a>\frac{\delta}{d}r$, where $\delta=d-2=\dim(l_0)$. 
\end{theorem}

\begin{proof}
Since for $r=0$ the result is clear, we assume in the sequel that $r>0$ and $0<\tau<p$. The proof is based on the wavelet characterization of Besov spaces presented in Subsection \ref{sect-Besov}.  Theorem \ref{thm-wavelet-dec} implies that it is enough to show  
\[
\left(\sum_{(I,\psi)\in \Lambda}|I|^{\left(\frac 1p-\frac 12\right)\tau}|\langle \tilde{u},\psi_I\rangle|^{\tau}\right)^{1/\tau}\leq c \max \{\|u|\mathcal{K}^m_{p,a}(\Omega)\|, \|u|B^s_{p,p}(\Omega)\|\}. 
\]

\underline{\em Step 1:} We explain why the first term  in \eqref{besov-decomp}  can be incorporated in the estimates that follow in Step 2.  Since our domain $\Omega$ is Lipschitz, we can extend every $u\in B^s_{p,p}(\Omega)$ to some function $\tilde{u}=Eu\in B^s_{p,p}(\real^d)$. Then the first term reads as 
\[
\sum_{k\in \mathbb{Z}^d}\langle \tilde{u}, {\phi}(\cdot -k)\rangle {\phi}(\cdot - k). 
\] 
Since ${\phi}$ shares the same smoothness and support properties as the wavelets $\psi_I$ for $|I|=1$ (note that below the vanishing moments of $\psi_I$  only become relevant for $|I|<1$), the coefficients $\langle \tilde{u}, {\phi}(\cdot -k)\rangle$ can be treated exactly like any of the coefficients $\langle \tilde{u},\psi_I\rangle$ in Step 2. 

\underline{\em Step 2:}  For our analysis we shall split the index set $\Lambda$ as follows.  For $j\in \nat_0$ the refinement level $j$ is denoted by 
\[
\Lambda_j:=\{(I,\Psi)\in \Lambda: \ |I|=2^{-jd}\}.
\]
Furthermore, for $k\in \nat_0$ put 
\[
\Lambda_{j,k}:=\{(I,\psi)\in \Lambda_j: \ k2^{-j}\leq \rho_I<(k+1)2^{-j}\},
\]
where $\rho_I=\inf_{x\in Q(I)} \rho(x)$. In particular, we have $\Lambda_j=\bigcup_{k=0}^{\infty}\Lambda_{j,k}$ and $\Lambda=\bigcup_{j=0}^{\infty}\Lambda_j$. \\
We consider first the situation when $\rho_I>0$ corresponding to  $k\geq 1$ and therefore put $\Lambda_j^0=\bigcup_{k\geq 1}\Lambda_{k,j}$. {Moreover, we require $Q(I)\subset \Omega$.} Recall Whitney's estimate  regarding approximation with polynomials,  cf. \cite[Sect.~6.1]{DV98}, which states that for every $I$ there exists a polyomial $P_I$ of degree less than $m$, such that 
\[
\|\tilde{u}-P_I|L_p(Q(I))\|\leq c_0 |Q(I)|^{m/d}|\tilde{u}|_{W^m_p(Q(I))}\leq c_1|I|^{m/d}|\tilde{u}|_{W^m_p(Q(I))}
\]
for some constant $c_1$ independent of $I$ and $u$, where 
$$|u|_{W^m_p(Q(I))}:=\left(\int_{Q(I)}|\nabla^m u(x)|^p\ud x\right)^{1/p}.$$
Note that $\psi_I$ satisfies moment conditions of order up to $m$, i.e., it is orthogonal to any polynomial of degree up to $m-1$. Thus, using H\"older's inequality with $p>1$ we estimate 
\begin{eqnarray}
|\langle \tilde{u},\psi_I\rangle|
&=&|\langle \tilde{u}-P_I,\psi_I\rangle|\leq \|\tilde{u}-P_I|L_p(Q(I))\|\cdot \|\psi_I|L_{p'}(Q(I))\|\notag\\
&\leq & c_1 |I|^{m/d}|\tilde{u}|_{W^m_p(Q(I))}|I|^{\frac 12-\frac 1p}\notag\\
&\leq & c_1 |I|^{\frac md+\frac 12-\frac 1p}\rho_I^{a-m}\left(\sum_{|\alpha|=m}\int_{Q(I)}|\rho(x)|^{m-a}\partial^{\alpha}\tilde{u}(x)|^p \ud x\right)^{1/p}\notag\\
&=:&c_1 |I|^{\frac md +\frac 12-\frac 1p}\rho_I^{a-m}\mu_I. \label{est-hyper}
\end{eqnarray}
{Note that in the third step above we require $a<m$.} 
On the refinement level $j$, using H\"older's inequality with $\frac{p}{\tau}>1$, we find
\begin{align*}
\sum_{(I,\psi)\in \Lambda^0_j}&|I|^{\left(\frac 1p-\frac 12\right)\tau}|\langle \tilde{u},\psi_I\rangle|^{\tau}\\
&\leq  \sum_{(I,\psi)\in \Lambda_j^0}\left(|I|^{\frac md}\rho_I^{a-m}\mu_I\right)^{\tau}\\
&\leq  c_1 \left(\sum_{(I,\psi)\in \Lambda_j^0}\left(|I|^{\frac md\tau}\rho_I^{(a-m)\tau}\right)^{\frac{p}{p-\tau}}\right)^{\frac{p-\tau}{p}}\left(\sum_{(I,\psi)\in \Lambda_j^0} \mu_I^p\right)^{\tau/p}. 
\end{align*}

For the second factor we observe that there is a controlled overlap between the cubes $Q(I)$, meaning each $x\in \Omega$ is contained in a finite number of cubes independent of $x$, such that we get
\begin{eqnarray*}
\left(\sum_{(I,\psi)\in \Lambda_j^0} \mu_I^p\right)^{1/p}
&=& \left(\sum_{(I,\psi)\in \Lambda_j^0} \sum_{|\alpha|=m}\int_{Q(I)}|\rho(x)^{m-a}\partial^{\alpha}\tilde{u}(x)|^p\ud x\right)^{1/p}\\
&\leq & c_2 \left(\sum_{|\alpha|=m}\int_{\Omega}|\rho(x)^{m-a}\partial^{\alpha}\tilde{u}(x)|^p\ud x\right)^{1/p}
\leq  c_2\|u|\mathcal{K}^{m}_{p,a}(\Omega)\|. 
\end{eqnarray*}
For the first factor  by choice or $\rho$ we always have $\rho_I\leq 1$, hence the index $k$ is at most $2^j$ for the sets $\Lambda_{j,k}$ to be non-empty. 
The number of elements in $\Lambda_{j,k}$  is bounded by $k^{d-1-\delta}2^{j\delta}$. With this we find 
\begin{align*}
\Bigg(\sum_{(I,\psi)\in \Lambda^0_j}&\left(|I|^{\frac md\tau}\rho_I^{(a-m)\tau}\right)^{\frac{p}{p-\tau}}\Bigg)^{\frac{p-\tau}{p}}\\
&\leq  \left({\sum_{k=1}^{2^j}\sum_{(I,\psi)\in \Lambda_{j,k}}}\left(2^{-jm\tau}(k2^{-j})^{(a-m)\tau}\right)^{\frac{p}{p-\tau}}\right)^{\frac{p-\tau}{p}}\\
&\leq  \left(\sum_{k=1}^{2^j}\sum_{(I,\psi)\in \Lambda_{j,k}}\left(2^{-ja\tau}k^{(a-m)\tau}\right)^{\frac{p}{p-\tau}}\right)^{\frac{p-\tau}{p}}\\
&\leq  \left(
c_3 2^{-ja\frac{p\tau}{p-\tau}}\sum_{k=1}^{2^j}k^{(a-m)\frac{p\tau}{p-\tau}}k^{d-1-\delta} 2^{j\delta}
\right)^{\frac{p-\tau}{p}}\\
&= c_4 2^{-ja\tau} 2^{j\delta\frac{p-\tau}{p}}\left(\sum_{k=1}^{2^j}k^{(a-m)\frac{p\tau}{p-\tau}+d-1-\delta} 
\right)^{\frac{p-\tau}{p}}.\\
\end{align*}
Looking at the value of the exponent in the last sum we see that 
\[
(a-m)\frac{p\tau}{p-\tau}+d-1-\delta>-1 \quad \iff \quad a-m+r\frac{d-\delta}{d}>0,
\]
which leads to 
\begin{align}
\Bigg(\sum_{(I,\psi)\in \Lambda^0_j}& \left(|I|^{\frac md\tau}\rho_I^{(a-m)\tau}\right)^{\frac{p}{p-\tau}}\Bigg)^{\frac{p-\tau}{p}}\notag \\
&\leq  c_5 2^{-ja\tau}2^{j\delta\frac{p-\tau}{p}}
\begin{cases}
2^{j\left((a-m)\tau+(d-\delta)\frac{p-\tau}{p}\right)}, & a-m+r\frac{d-\delta}{d}>0,\\
(j+1)^{\frac{p-\tau}{p}}, & a-m+r\frac{d-\delta}{d}=0, \label{cases-hyper}\\
1,& a-m+r\frac{d-\delta}{d}<0.
\end{cases}
\end{align}
{The case  $a>m$  can be treated in the same way as above by taking out ${\tilde{\rho}_I}^{a-m}$ with $\tilde{\rho}_I:=\inf_{x\in Q(I)}\rho(x)$ instead of $\rho_I^{a-m}$ in the integral appearing in \eqref{est-hyper}. The values of $\rho_I$ and $\tilde{\rho}_I$ are comparable in this situation, since we consider cubes which do not intersect with the boundary, i.e., we have $k\geq 1$. In particular, in \eqref{cases-hyper}  only the first case occurs if $a>m$. }\\ 
\underline{\em Step 3:} We now put $\Lambda^0:=\bigcup_{j\geq 0}\Lambda_j^0$. Summing the first line of the last estimate over all $j$, we obtain 
\begin{align*}
\sum_{(I,\psi)\in \Lambda^0}&|I|^{\left(\frac 1p-\frac 12\right)\tau}|\langle \tilde{u}, \psi_I\rangle|^{\tau}\\
&\leq  c_6\sum_{j=0}^{\infty}2^{-j(m\tau-d\frac{p-\tau}{p})}\|u|\mathcal{K}^m_{p,a}(\Omega)\|^{\tau}\lesssim \|u|\mathcal{K}^m_{p,a}(\Omega)\|^{\tau}<\infty,
\end{align*}
if the geometric series converges, which happens if 
\[
m\tau>d\frac{p-\tau}{p}\quad \iff\quad m>d\frac rd\quad  \iff \quad m>r.
\]
Similarly, in the second case we see that 
\begin{align*}
\sum_{(I,\psi)\in \Lambda^0}&|I|^{\left(\frac 1p-\frac 12\right)\tau}|\langle \tilde{u}, \psi_I\rangle|^{\tau}\\
&\leq c_7 \sum_{j=0}^{\infty}2^{-j(a\tau-\delta\frac{p-\tau}{p})}(j+1)^{\frac{p-\tau}{p}}\|u|\mathcal{K}^m_{p,a}(\Omega)\|^{\tau}\lesssim \|u|\mathcal{K}^m_{p,a}(\Omega)\|^{\tau}<\infty,
\end{align*}
where the series converges if
\[
a\tau>\delta \frac{p-\tau}{p},\quad \text{i.e.,} \quad a>\delta \frac rd, \quad \text{i.e.,} \quad m>r\frac{d-\delta}{d}+\frac{\delta}{d}r=r,\quad \text{i.e.,} \quad m>r,
\]
which is the same condition as before. Finally, in the third case we find 
\begin{align*}
\sum_{(I,\psi)\in \Lambda^0}&|I|^{\left(\frac 1p-\frac 12\right)\tau}|\langle \tilde{u}, \psi_I\rangle|^{\tau}\\
&\leq c_8 \sum_{j=0}^{\infty}2^{-j(a\tau-\delta \frac{p-\tau}{p})}\|u|\mathcal{K}^m_{p,a}(\Omega)\|^{\tau}\lesssim \|u|\mathcal{K}^m_{p,a}(\Omega)\|^{\tau}<\infty,
\end{align*}
whenever
\[
a\tau>\delta\frac{p-\tau}{p}\quad \iff \quad a>\delta\frac rd
\]
as in the second case above. \\
\underline{\em Step 4:} We need to consider the sets $\Lambda_{j,0}$, i.e., the wavelets close to $l_0$. Here, we shall make use of the assumption $\tilde{u}\in B^s_{p,p}(\real^d)$. Since the number of elements in  $\Lambda_{j,0}$ is bounded from above by $c_9 2^{j\delta}$ we estimate using H\"older's inequality with $\frac{p}{\tau}>1$ and obtain 
\begin{align*}
\sum_{(I,\psi)\in \Lambda_{j,0}}&|I|^{\left(\frac 1p-\frac 12\right)\tau}|\langle \tilde{u}, \psi_I\rangle|^{\tau}\\
&\leq  c_9^{\frac{p-\tau}{p}}2^{j\delta\frac{p-\tau}{p}}2^{-jd\left(\frac 1p-\frac 12\right)\tau}\left(\sum_{(I,\psi)\in \Lambda_{j,0}}|\langle\tilde{u},\psi_I\rangle|^p\right)^{\tau/p}\\
&= c_9^{\frac{p-\tau}{p}}2^{j\delta\frac{p-\tau}{p}}2^{-js\tau}\left(\sum_{(I,\psi)\in \Lambda_{j,0}}2^{j\left(s+\frac d2-\frac dp\right)p}|\langle\tilde{u},\psi_I\rangle|^p\right)^{\tau/p}.
\end{align*}
Summing up over $j$ and once more using H\"older's inequality with $\frac{p}{\tau}>1$ gives 
\begin{eqnarray*}
&\ds \sum_{j=0}^{\infty}&\sum_{(I,\psi)\in \Lambda_{j,0}}|I|^{\left(\frac 1p-\frac 12\right)\tau}|\langle \tilde{u}, \psi_I\rangle|^{\tau}\\
&\leq & c_9^{\frac{p-\tau}{p}}\sum_{j=0}^{\infty}2^{j\delta\frac{p-\tau}{p}}2^{-js\tau}\left(\sum_{(I,\psi)\in \Lambda_{j,0}}2^{j\left(s+\frac d2-\frac dp\right)p}|\langle\tilde{u},\psi_I\rangle|^p\right)^{\tau/p}\\
&\leq & c_9^{\frac{p-\tau}{p}}\left(\sum_{j=0}^{\infty}2^{j\delta }2^{-js\tau\frac{p}{p-\tau}}\right)^{\frac{p-\tau}{p}}\cdot \left(\sum_{j=0}^{\infty}\sum_{(I,\psi)\in \Lambda_{j,0}}2^{j\left(s+\frac d2-\frac dp\right)p}|\langle\tilde{u},\psi_I\rangle|\right)^{\tau/p}\\
&\lesssim  &\|\tilde{u}|B^s_{p,p}(\real^d)\|^{\tau}\lesssim \|u|B^s_{p,p}(\Omega)\|^{\tau},
\end{eqnarray*}
under the condition 
\[
\delta<\frac{sp\tau}{p-\tau}\quad \iff \quad \frac{s}{\delta}>\frac{1}{\tau}-\frac 1p=\frac rd \quad \iff \quad r<\frac{sd}{\delta}.
\]
{\underline{\em Step 5:} Finally, we need to consider those $\psi_I$ whose support intersect $\partial \Omega$. In this case we can estimate similar as in Step 4 with $\delta$ replaced by $d-1$. This results in the condition 
\[
\sum_{(I,\psi)\in \Lambda: \ \supp \psi_I\cap \partial \Omega\neq \emptyset}
|I|^{\left(\frac 1p-\frac 12\right)\tau}|\langle \tilde{u},\psi_I\rangle|^{\tau}
\lesssim \|\tilde{u}|B^{s}_{p,p}(\real^d)\|^{\tau}\lesssim \|u|B^s_{p,p}(\Omega)\|^{\tau}
\]
if $r<\frac{sd}{d-1}$. 
Altogether, we have proved 
\[
\|u|B^r_{\tau,\tau}(\Omega)\|\leq \|\tilde{u}|B^r_{\tau,\tau}(\real^d)\|
\lesssim \|u|B^s_{p,p}(\Omega)\|+\|u|\mathcal{K}^m_{p,a}(\Omega)\|,
\]
with constants independent of $u$. } 
 \end{proof}

%%%%%%%%%%%%%%%%%%%%%%%%%%%%%%%%%%%%%%%

\medskip 

As an immediate consequence of Theorem \ref{emb-hyper-1} and the definition of corresponding function spaces on $\Omega_T$ we have the following generalized embedding result. \\

\begin{theorem}\label{emb-hyper-2}
Let $\Omega\subset \real^d$ be a special Lipschitz domain from Definition \ref{def_special_Lip}.  Then for $1<p<\infty$ and $0<q\leq \infty$, we have 
\[
L_q((0,T),\mathcal{K}^{m}_{p,a}(\Omega))\cap L_q((0,T),B^s_{p,p}(\Omega))\hookrightarrow L_q\left((0,T), B^{r}_{\tau,\tau}(\Omega)\right), \quad \frac{1}{\tau}=\frac rd+\frac 1p, 
\] 
for all $0\leq r<\min(m,\frac{sd}{{d-1}})$ and $a>\frac{\delta}{d}r$, where $\delta=d-2=\dim(l_0)$. 
\end{theorem}

Now Lemma \ref{thm-hyperbolic-weight-sob-reg} together with Theorem \ref{emb-hyper-2} give the following result concerning the Besov regularity of the solution to \eqref{hyp-1a}. 

\begin{theorem}[Hyperbolic Besov regularity]
Let $\Omega\subset \real^d$ be a special Lipschitz domain from Definition \ref{def_special_Lip}. Furthermore, let  $m\in \nat$, $m\geq 2$, $m-1\leq a<\min(m,\frac{\pi}{\omega}-1)$, and   assume $f$ satisfies   
\begin{itemize}
\item[(i)] $\partial_{t^j} f\in L_{\infty}((0,T),\mathcal{K}^{m-j}_{2,a-j}(\Omega))$, $0\leq j\leq m-1$, 
\item[(ii)] $\partial_{t^j} f(x,0)=0$, $0\leq j\leq m-2$.
\end{itemize}
Then for the generalized solution $u$ of problem \eqref{hyp-1a}  we have 
\[
u\in L_{\infty}((0,T),B^r_{\tau,\tau}(\Omega)), \qquad \frac{1}{\tau}=\frac rd+\frac 1p, \qquad 0\leq r<\min \left\{m,\frac{d}{{d-1}}\right\}.
\]
In particular, the following a priori estimate holds 
\[
\|u| L_{\infty}((0,T),B^{r}_{\tau,\tau}(\Omega))\|\lesssim \sum_{j=0}^{m-1} \|\partial_{t^j} f|L_{\infty}((0,T),\mathcal{K}^{m-j}_{2,a}(\Omega))\|. 
\]
\end{theorem}

\begin{proof}
According to Lemma \ref{thm-hyperbolic-weight-sob-reg} we know  $u\in L_{\infty}((0,T),W^1_2(\Omega))\cap L_{\infty}((0,T),\mathcal{K}^{m}_{2,a}(\Omega))$ for $m-1\leq a<\min(m,\frac{\pi}{\omega}-1)$.  
Now using Theorem \ref{emb-hyper-2} with $s=1$ and $\max(m-1,\frac{\delta}{d}r)<a<\min(m,\frac{\pi}{\omega}-1)$ yields the desired embedding result. As for the restriction on ${a}$, {using $m\geq 2$ and $d\geq 3$},  we further observe that 
\[
\frac{\delta}{d}r < \frac{d-2}{d}\min\left(m,\frac{d}{{d-1}}\right) = \frac{d-2}{d}\cdot \frac{d}{{d-1}} \leq  1. 
\]
Since $m-1\geq 1>\frac{\delta}{d}r$  the restriction on $a$ reads as  $m-1\leq a<\min(m,\frac{\pi}{\omega}-1)$ {(note that in this case equality in the lower bound is also possible)}. 
\end{proof}

\remark{There are more results in \cite{LT15} compared to what  we used in this section. In particular, in \cite[Thm.~3.5]{LT15}   weighted Sobolev regularity of nonlinear hyperbolic problems was investigated.   
Therefore, using a fixed point theorem as in subsection \ref{Subsect-4.2}  it should  be possible to study Besov regularity of nonlinear hyperbolic problems as well. But this is out of our scope for now and  will  possibly be  treated in a forthcoming paper. 
}

\section{Relations to Adaptive Algorithms} \label{Practice}

In Section \ref{introduction} we already sketched  why we expect that the results proved in this paper will have some impact concerning the theoretical foundation of adaptive algorithms.  In this section, we want to explain these relationships in more detail. 

Let us start with adaptive wavelet algorithms as e.g. discussed in \cite{CDD1, Ste09}.  Let $\Psi=\{\psi_I:(I,\psi)\in\Lambda\}$ be a wavelet system with sufficiently high differentiability and vanishing moments, such that all relevant (unweighted) Sobolev and Besov spaces can be characterized in terms of expansion coefficients w.r.t. $\Psi$, see again Subsection \ref{sect-Besov}.  Then, the best thing  we can expect from an adaptive numerical algorithm based on this wavelet basis is that it realizes the convergence order of  best $N$-term wavelet approximation schemes. In this sense, best $N$-term wavelet approximation serves as the benchmark for the performance of adaptive algorithms.  Let $\mathcal{O}\subset \real^d$ denote some bounded Lipschitz domain. The error of best $N$-term approximation is defined by

\begin{equation}\label{errorbestN} \sigma_N\bigl(u;L_p(\mathcal{O})\bigr)=\inf_{\Gamma\subset\Lambda:\#\Gamma\leq N}\inf_{c_\gamma}
\biggl\|u-\sum_{\gamma=(I,\psi)\in\Gamma}c_\gamma\psi_I\bigg|L_p(\mathcal{O})\biggr\|\,,
\end{equation}
i.e., as the name suggests we consider the best approximation by linear combinations of the basis functions consisting of at most $N$ terms. Of course, in the context of the numerical approximation of the solutions to operator equations, such an approximation scheme would never be implementable because this would require the knowledge of all wavelet coefficients, i.e., of the solution itself. Nevertheless,  it has been shown that the recently developed adaptive wavelet algorithms indeed asymptotically realize the same order of approximation \cite{CDD1, Ste09}! To quantify the approximation rate, we introduce the approximation classes $\ca^\alpha_q(X)$, $\alpha>0$, $0<q\leq\infty$, by requiring
\begin{equation}\label{eq-approxclass}
    \|u|\ca^\alpha_q(L_p(\mathcal{O}))\|
=\biggl(\sum_{N=0}^\infty\Bigl((N+1)^\alpha\sigma_N\bigl(u;L_p(\mathcal{O})\bigr)\Bigr)^q
            \frac{1}{N+1}\biggr)^{1/q}<\infty\,.
\end{equation}

Then a  fundamental result of DeVore, Jawerth, and Popov \cite{DVJP92} states that

\begin{equation*}
\ca^{m/d}_\tau(L_p(\real^d))=B^m_{\tau,\tau}(\real^d)\,,\qquad\frac{1}{\tau}=\frac{m}{d}+\frac{1}{p}\,.
\end{equation*}

Consequently, the optimal approximation rate that can be achieved by adaptive wavelet schemes depends on the Besov smoothness of the unknown solution in the adaptivity scale (\ref{adaptivityscale}).  In contrast to this, the convergence order of classical nonadaptive (uniform) schemes depends on the Sobolev smoothness of the solution, see, e.g.  \cite{Hack92, DDD} for details. The results presented in this paper imply that for each $t \in (0,T)$ the Besov regularity of the unknown solutions of the problems studied here is much higher than the Sobolev regularity,  which justifies the use of spatial adaptive wavelet algorithms.  This corresponds to the classical time-marching schemes such as the Rothe method. We refer e.g. to the monographs \cite{Lan01, Tho06} for a detailed discussion. Of course,  it would be tempting to employ adaptive wavelet strategies in the whole space-time cylinder.  First results in this direction have been reported in \cite{SS09}.  To justify also these schemes,  Besov regularity in the whole space-time cylinder has to be established. This case will be studied in a forthcoming paper.

Quite recently, it has also turned out that the basic relationships outlined above also carry over to discretization schemes based on finite element schemes  \cite{GM09}.  The starting point is an initial triangulation $\mathcal{T}_0$ of the polyhedral domain $D\subset \real^d$. Furthermore,  $\mathbb{T}$ denotes the family of all conforming, shape-regular partitions $\mathcal{T}$ of $D$ obtained from $\mathcal{T}_0$ by refinement using bisection rules. Moreover, $V_{\mathcal{T}}$ denotes the finite element space of continuous piecewise polynomials of degree at most $r$, i.e., 

$$
    V_{\mathcal{T}}=\bigl\{v\in C(\overline{D}):v|_T\in\mathcal{P}_r\text{ for all }T\in\mathcal{T}\bigr\}\,.
$$
Then the counterpart to the quantity $\sigma_N(u)$ is given by
$$
    \sigma^{FE}_N\bigl(u;L_p(D)\bigr)
=\min_{\substack{\mathcal{T}\in\mathbb{T}:\\\#\mathcal{T}-\#\mathcal{T}_0\leq N}}
        \inf_{v\in V_{\mathcal{T}}}\|u-v\|_{L_p(D)},\quad
    0<p<\infty\,.
$$
Then \cite[Thm. 2.2]{GM09} gives direct estimates,
\begin{equation*}
    \sigma_N^{FE}\bigl(u;L_p(D)\bigr)\leq C\,N^{-s/d}\|f|B^s_{\tau,\tau}(D)\|\,. 
\end{equation*}

Therefore, the results presented in this paper also justify the use of adaptive time-marching schemes based on finite elements.

\begin{appendix}

\section{Supplementary results}\label{Appendix-A}

{We provide some auxiliary information on results used throughout the paper.}

\subsection{Embeddings of generalized H\"older spaces}

\remark{\label{quasi-B-deriv}
Let $Y$ be some (quasi-)Banach space such that $X\hookrightarrow Y$. Then it follows that 
\[
C^k(I,X)\hookrightarrow C^k(I,Y)\qquad {\text{and}\qquad \mathcal{C}^{k,\alpha}(I,X)\hookrightarrow \mathcal{C}^{k,\alpha}(I,Y).}
\]
This is an immediate consequence of the definition of the spaces. Let $u: I\rightarrow X\in C^k(I,X)$ with Taylor expansion 
\[
u(t+h)=u(t)+u'(t)h+\frac{1}{2}u''(t)h^2+\ldots + \frac{1}{k!}u^{(k)}(t)h^k+r_k(t,h).
\] 
Then also $u:I\rightarrow Y$ and we have 
\begin{itemize}
\item $u^{(j)}(t)$ depends continuously on $t$ for all $j=0,\ldots, k$, since 
\[
|t-t_0|<\delta \quad \Longrightarrow \quad \|u^{(j)}(t)-u^{(j)}(t_0)|Y\|\leq \|u^{(j)}(t)-u^{(j)}(t_0)|X\|<\varepsilon, 
\]
\item $\ds \lim_{|h|\rightarrow 0}\frac{\|r_k(t,h)|Y\|}{|h|^k}\leq \lim_{|h|\rightarrow 0}\frac{\|r_k(t,h)|X\|}{|h|^k}= 0$,
\end{itemize}
from which we deduce that $u\in C^k(I,Y)$. {Moreover, concerning the generalized H\"older spaces we now observe that 
\begin{eqnarray*}
\|u|\mathcal{C}^{k,\alpha}(I,Y)\|
&=& \|u|C^k(I,Y)\|+\sup_{t,s\in I\atop t\neq s}\frac{\|u(t)-u(s)|Y\|}{|t-s|^{\alpha}}\\
&\leq & \|u|C^k(I,X)\|+\sup_{t,s\in I\atop t\neq s}\frac{\|u(t)-u(s)|X\|}{|t-s|^{\alpha}}=\|u|\mathcal{C}^{k,\alpha}(I,X)\|,
\end{eqnarray*}
which gives the desired result. 
}
}

\subsection{Generalized Sobolev embedding}

{\bf Proof of Theorem \ref{thm-sob-emb}:}  Note that the theorem of Meyers-Serrin extends to the spaces $W^{m}_p(I,X)$, cf.  \cite[Th. 4.11]{Kreu15}. Hence, $C^{\infty}(I,X)$ is dense in $W^{m}_p(I,X)$. It is also shown in  \cite[Prop. 4.3]{Kreu15} that in this case weak derivatives coincide with normal derivatives.  Since $k=m-1$, using \cite[Thm.~1.4.35]{CH98} together with Bochner's Theorem and H\"older's inequality gives for  $u\in C^{\infty}(I,X)$, 
\begin{eqnarray*}
\left\|u^{(k)}(t+h)-u^{(k)}(t)|X\right\|
&= & \left\|\int_t^{t+h}u^{(k+1)}(s) \ud s|X\right\| \\
&\leq & \int_t^{t+h}\left\|u^{(m)}|X\right\|\ud s\\
&\leq &  h^{1-\frac 1p}\left(\int_t^{t+h}\left\|u^{(m)}|X\right\|^p\ud s\right)^{\frac 1p}\\
&\leq & h^{\alpha}\left\|u\right|{W^m_p(I,X)}\|. 
\end{eqnarray*}
Hence,  we see that $\mathrm{id}$ is a linear and bounded operator from the dense subset  $C^{\infty}(I,X)$ of $W^{m}_p(I,X)$ into the Banach-space $\mathcal{C}^{k,\alpha}(I,X)$ admitting  an extension $\widetilde{\mathrm{id}}$ onto $W^{m}_p(I,X)$ with equal norm. This completes the proof.

\end{appendix}

%\bibliographystyle{amsplain}%{amsalpha}
%\bibliographystyle{plain}
%\bibliography{Maths/Research/literatur}
%\bibliography{/Users/cornelia/Maths/Research/literatur}	
%\bibliography{literatur}	

\def\cprime{$'$}

\newpage 

{
\bigskip
\vbox{\noindent Stephan Dahlke\\
Philipps-Universit\"at Marburg\\
FB12 Mathematik und Informatik\\
Hans-Meerwein Stra\ss e, Lahnberge\\
35032 Marburg\\
Germany\\
E-mail: {\tt dahlke@mathematik.uni-marburg.de}\\
Web: {\tt http://www.mathematik.uni-marburg.de/$\sim$dahlke/}\\}
}

{
\bigskip
\smallskip
\vbox{\noindent  Cornelia Schneider\\
Friedrich-Alexander-Universit\"at Erlangen-N\"urnberg\\ 
Department Mathematik, AM3\\
Cauerstr. 11\\ 
91058 Erlangen \\ 
Germany\\
E-mail: {\tt schneider@math.fau.de}\\
Web: {\tt http://www.math.fau.de/$\sim$schneider/}\\}
}

\end{document}